\documentclass[a4paper,numbers=noenddot, fontsize = 10pt, oneside, final]{article}

\usepackage[margin=2.5cm]{geometry} 

\usepackage{amsmath, bm, mathtools}
\usepackage{amsthm}
\usepackage{amssymb}
\usepackage{soul}

\usepackage{enumitem}

\usepackage[final]{hyperref}
\hypersetup{
	unicode,
	colorlinks=true,
	linkcolor=blue,
	citecolor=red, 
	filecolor=magenta,      
	urlcolor=blue,
	pdfauthor={Author One, Author Two, Author Three}, 
	pdfkeywords={article, template, simple},
	pdfproducer={LaTeX}
}

\theoremstyle{plain}
\newtheorem{theorem}{Theorem}

\newtheorem{lemma}[theorem]{Lemma}

\newtheorem{definition}[theorem]{Definition}
\theoremstyle{definition}

\newtheorem{remark}[theorem]{Remark}

\usepackage{graphicx, color}
\graphicspath{{fig/}}

\usepackage{mathrsfs} 

\usepackage{lipsum}

\title{Existence of Weak Solutions to a Cahn--Hilliard--Biot System}
\author{Helmut Abels\thanks{Fakultät für Mathematik, Universität Regensburg, 93040 Regensburg, Germany, email: \href{mailto:helmut.abels@ur.de}{helmut.abels@ur.de},}\and 
Harald Garcke\thanks{Fakultät für Mathematik, Universität Regensburg, 93040 Regensburg, Germany, email: \href{mailto:harald.garcke@ur.de}{harald.garcke@ur.de},} 
\and Jonas Haselböck \thanks{Fakultät für Mathematik, Universität Regensburg, 93040 Regensburg, Germany, email:
		\href{mailto:jonas.haselboeck@ur.de}{jonas.haselboeck@ur.de} } }

\date{}

\newcommand\numberthis{\addtocounter{equation}{1}\tag{\theequation}}
\newcommand{\norm}[1]{ \|   #1 \|  }

\newcommand{\abs}[1]{ |    #1 |}
\newcommand{\N}{\mathbb{N}}
\newcommand{\inN}{\in \N}
\newcommand{\OT}{\Omega_T}
\newcommand{\Ot}{\Omega_t}
\newcommand{\vphi}{\varphi}
\newcommand{\vepsilon}{\varepsilon}
\newcommand{\Tau}{\mathcal{T}}
\newcommand{\pt}{\partial_t}

\newcommand{\Wp}{W_{,\vphi}}
\newcommand{\E}{\mathcal{E}}
\newcommand{\WE}{W_{,\E}}
\newcommand{\F}{\mathcal{F}}

\newcommand{\R}{\mathbb{R}}
\newcommand{\bu}{\bm{u}}

\newcommand{\dx}{\, d\mathbf{x}}
\newcommand{\dy}{\, d\mathbf{y}}

\newcommand{\dt}{\, dt}
\newcommand{\dtx}{ \, d(t, \bm{x})}
\newcommand{\dH}{\, d\mathcal{H}^{n-1}}
\newcommand{\vr}{\varrho}

\newcommand{\C}{\mathbb{C}}

\DeclareMathOperator\supp{supp}
\DeclareMathOperator\spn{span}

\newcommand{\new}[1]{{#1}}

\def\Xint#1{\mathchoice
	{\XXint\displaystyle\textstyle{#1}}%
	{\XXint\textstyle\scriptstyle{#1}}%
	{\XXint\scriptstyle\scriptscriptstyle{#1}}%
	{\XXint\scriptscriptstyle\scriptscriptstyle{#1}}%
	\!\int}
\def\XXint#1#2#3{{\setbox0=\hbox{$#1{#2#3}{\int}$ }
		\vcenter{\hbox{$#2#3$ }}\kern-.6\wd0}}

\def\dashint{\Xint-}

\begin{document}
	
	\maketitle
	
	\begin{abstract}
		We prove existence of weak solutions to a diffuse interface model describing the flow of a fluid through a deformable porous medium consisting of two phases. 
		The system non-linearly couples Biot's equations for poroelasticity, including phase-field dependent material properties, with the Cahn--Hilliard equation to model the evolution of the solid, and is further augmented by a visco-elastic regularization \new{of Kelvin--Voigt type.}
		To obtain this result, we approximate the problem in two steps, where first a semi-Galerkin ansatz is employed to show existence of weak solutions to regularized systems, for which later on compactness arguments allow limit passage. Notably, we also establish a maximal regularity theory for linear visco-elastic problems. 
	\end{abstract}
	
\textbf{AMS-Classification:} 35A01, 35D30, 35K41, 74B10


\section{Introduction}
\label{sec:intro}

Interactions of fluid flow, elastic effects and phase-separation can be observed in diverse natural, biological and mechanical situations with many practical applications of relevance. 
Due to the intricate interplay between several physical laws and thermodynamic relations, accurate mathematical models of such are often challenging and require subtle techniques to analyze. 
A prime example is the flow of a fluid through a deformable porous medium consisting of different phases with distinct properties, as exhibited in, e.g., biogrout processes and tumor growth. 
\par 
This paper is concerned with the analysis of a diffuse interface, Cahn--Hilliard--Biot model recently proposed on this topic by Storvik et al. \cite{STORVIK2022107799} and aims to establish the existence of weak solutions under general assumptions on material parameters, boundary conditions and source terms. While the quasi-static Biot equations are a standard model to describe single-phase flow through porous materials subject to linear elasticity, the coupling to the Cahn--Hilliard equation allows for the inclusion of phase changes in the solid, where the resulting elastic deformations, changing material properties and fluid pressure all mutually affect each other. This leads to a three-way coupled system of parabolic-elliptic type featuring several nonlinearities. Moreover, we account for \new{viscous effects} by augmenting the equation for linear elasticity with a visco-elastic term that has a dampening effect on the system and allows us to deduce higher regularity. 
The analysis is further aided by an underlying generalized gradient flow structure, cf. \cite{STORVIK2022107799}, which is crucial for the derivation of \textit{a priori} estimates. We note that the visco-elastic term is purely dissipative and therefore retains this property. 
\par
The Cahn--Hilliard equation was originally proposed in \cite{cahn1958free} as a model for phase separation in nonuniform, binary alloys and has been studied extensively over the past decades, cf.~e.g.~\cite{blowey1991cahn, cahn1996cahn, elliot_garcke00}. The coupling of this system with linear elasticity goes back to Cahn and Larché \cite{larcht1982effect} and Onuki \cite{onuki1989ginzburg}, for which analytical results can be found in \cite{bonetti2002model,MR1807441, garcke_2003, GARCKE2005165}. Moreover, there are various extensions of these models to describe different phenomena. In particular, for well-posedness results which also incorporate flow fields, we refer to \cite{ebenbeck2019analysis, colli2017asymptotic, garcke16darcy, lowengrub2013analysis, MR4126782}. 
\par 
Poroelasticity, as proposed by Biot \cite{biot1941general}, combines Darcy's law for fluid flow within a saturated porous medium with elastic effects and has many applications ranging from petroleum engineering to biological tissues. Fundamental analytical results on these equations can be found in \cite{auriault1980dynamic,vzenivsek1984existence, SHOWALTER2000310}, while more recent works include \cite{ bociu2016analysis,bociu2021multilayered, van2023mathematical}. We particularly refer to \cite{bociu2023mathematical} and the references cited therein. 
\new{Moreover, results that include nonlinear relations in the Biot model can be found in \cite{MR1876882, MR2102316, MR3794344, MR4649995}.}
\par
As mentioned earlier, we include a visco-elastic regularization in the equation for linear elasticity, which is not only an established approach in nonlinear poroelasticity, cf. e.g. \cite{SHOWALTER2000310,  bociu2016analysis, both21global}, but can also be justified in view of physical applications, as we observe both elastic and visco-elastic behavior in biological tissues, cf. \cite{bociu2023mathematical, mow1980biphasic, sacco2019comprehensive}. Moreover, poro-visco-elasticity was already discussed by Biot in \cite{biot1956theory} and a chapter in Coussy's book \cite{coussy2004poromechanics} is dedicated to this subject. 
\par
The principle unknowns we consider in this work are the phase-field variable $\vphi$ with an associated chemical potential $\mu$, the displacement $\bm{u}$ of the solid due to (visco-)elastic deformations and the volumetric fluid content $\theta$ along with the pore pressure $p$. The full system, including a detailed derivation, can be found in Section \ref{sec:derivation}. \\ 
The main difficulties in the analysis arise due to the nonlinear coupling of the Cahn--Hilliard equation with linear elasticity and fluid flow. As we will discuss later on, the equation for the chemical potential $\mu$ includes quadratic dependencies on the displacement $\bm{u}$ and the volumetric fluid-content $\theta$. We also allow for phase-field depended material properties, mixed boundary conditions on the displacement $\bm{u}$, body forces and external loading, as well as general source terms, which will further complicate the analysis.  
\par 
Our strategy can be briefly summarized as follows: We will first consider a regularized system, featuring an additional bi-Laplacian of $\vphi$ in the equation for $\mu$ along with an additional Laplacian of $\theta$ in the equation for $\mu$, and employ a semi-Galerkin scheme to show existence of weak solutions. This includes splitting the equations into two subsystems and solving these separately in a first step utilizing both ODE methods and maximal regularity. Finding an approximate solution to the whole system is then reduced to a fixed-point problem, which we solve by using the Leray--Schauder principle. With the help of the Aubin--Lions--Simon theorem (in the version with translation), we can deduce strong convergence for the pressure $p$, which in turn yields the strong convergences of $\bm{u}$. 
Finally, we pass to the limit in which the previously added regularization vanishes. \par 
At this point, it is important to point out that a Faedo--Galerkin ansatz should also suffice for the existence proof. 
Our approach, however, exploits the evolutionary structure of the visco-elastic equation, which allows us to consider the problem in the framework of maximal regularity and establish well-posedness results, while also providing useful tools for further investigations regarding higher regularity and the existence of strong solutions. 
\par 
\medskip 
We would like to discuss two other recently published, independent papers \cite{fritz2023wellposedness, riethmüller2023wellposedness}, which also study the Cahn--Hilliard--Biot system we present here. While both discuss well-posedness of these equations, \cite{fritz2023wellposedness} also includes numerical experiments and highlights its application in tumor growth. A study of solution strategies can also be found in \cite{storvik2024sequential} and a structure-preserving approximation is the subject of \cite{brunk2024structurepreservingapproximationcahnhilliardbiot}. Since the present work is merely concerned with existence of weak solutions, the following comparison of results is limited to this aspect. \par 
In the paper by Fritz \cite{fritz2023wellposedness}, the underlying assumptions are very general and permit all material parameters and source terms to depend on the phase-field $\vphi$. Moreover, he considers mixed boundary conditions for the displacement $\bm{u}$, albeit restricted to the homogeneous case, and forgoes any regularization in the form of visco-elastic terms. To obtain existence of solutions, he uses a Faedo--Galerkin ansatz and derives \textit{a priori} estimates from energy methods, exploiting the gradient flow structure. However, the treatment of the nonlinearities in the equation \new{defining the} chemical potential \new{contain erroneous arguments for the limit passage from the discrete to the continuous problem}. In particular, there are quadratic dependencies \new{on the gradient of the displacement $\bm{u}$ and the volumetric fluid content $\theta$} but the author \new{merely} discusses weak convergence for the involved functions. \par 
On the other hand, Riethmüller et al.~\cite{riethmüller2023wellposedness} enhance the Cahn--Hilliard--Biot model with the regularizing, visco-elastic term $\nabla \cdot (\eta \pt (\nabla \cdot \bm{u}))$, which is a slightly weaker assumption compared to $\nabla \cdot (\C_\nu(\vphi) \E(\pt \bm{u}))$, controlling the full $H^1$-norm, we use. We note that they also mention general laws of our form, but only consider homogeneous visco-elastic properties. We further remark that while some of our arguments would not work in their visco-elastic setting due to missing ellipticity (cf. Theorem \ref{thm:elliptic_sobolev}), it is easy to see that the main arguments yielding the necessary convergence properties in our study also work in their setting. \\ 
Moreover, unlike our analysis, the source terms in \cite{riethmüller2023wellposedness} are assumed to be independent of the evolution, and homogeneous Dirichlet boundary conditions are prescribed for the displacement $\bm{u}$. Our results further differ in the notion of weak solutions, where we use a weak identity for the pressure $p$, while the formulation of Riethmüller et al. exploits the fluid flux. \\ 
Most significant is the dependency of the Biot--Willis coefficient $\alpha$ and the compressibility $M$ on the phase-field $\vphi$ in our model, whereas $\alpha, M$ are assumed to be constant by the other authors. This relaxation is also discussed in their work, cf. \cite[Rem.~4, Rem.~5, §5]{riethmüller2023wellposedness}. Taking advantage of the weaker assumption, a Faedo--Galerkin ansatz is employed by Riethmüller et al.~to show the existence of weak solutions, where the strong compactness result for the projected pressure $\pi_k$ in \cite[Lem.~12]{riethmüller2023wellposedness} follows with the help of the Aubin--Lions--Simon theorem. However, they proceed by trying to establish the convergence $\norm{p_k - \pi_k} \rightarrow 0$ of the difference between the pressure $p_k$ and its projection, which would yield the crucial strong convergence of $p_k$. To this end, they use that the orthogonal projection converges to the identity in operator-norm, see \cite[Lem.~12]{riethmüller2023wellposedness}, which does not hold. \new{Hence, the compactness arguments are insufficient, rendering the proof in the current version of \cite{riethmüller2023wellposedness} incomplete.}\\ 
While we encounter a similar problem in Lemma \ref{lem:p_strong}, we exploit the additional regularity of $\vphi$, originating from the regularization, which allows us to apply another version of the Aubin--Lions--Simon theorem, utilizing translations in time instead of derivatives, and consider these in a weaker space, thus obtaining the strong convergence of $p$. We further point out that at this point the regularization is necessary since the projection $\Pi_k^y$ is not regular enough due to the mixed boundary conditions. \\ 
\indent
\new{Before concluding the discussion of related literature, let us mention that the results on uniqueness and continuous dependence in both \cite{fritz2023wellposedness} and \cite{riethmüller2023wellposedness} are limited to quite restrictive cases and assume $M, \kappa, \alpha, M$ and $\C$ to be constant. Moreover, there are also further restrictions to $\psi, \Tau$ and in \cite{riethmüller2023wellposedness} it is assumed that the partial visco-elastic regularization vanishes. Since, in general, uniqueness even for the Cahn--Hilliard equation with variable mobility is still an open problem, better results seem to be out of reach for now. While some results concerning weak-strong uniqueness for systems which include the Cahn--Hilliard equation are known,  e.g. \cite{juengelweak,han2014existence}, the complexity of the system at hand adds many difficulties to finding an appropriate relative entropy functional and could be subject of future investigations. 
}\\ 
\indent
\new{A main novelty of our work is that we use maximal regularity theory to obtain an existence result under very general assumptions on the constitutive relationship between material properties and coefficient functions, while also allowing for the physically more relevant mixed boundary conditions.} 
\indent
\par 
\medskip 
This paper is structured as follows. We start in Section \ref{sec:derivation} with the derivation of the system as a generalized gradient flow, especially focusing on the constitutive relations, balance laws and \new{Kelvin--Voigt visco-elastictiy}. This is followed by the precise formulation of the required assumptions and our main theorem in Section~\ref{sec:assumptions}. Before starting with the analysis, Section~\ref{sec:preliminaries} introduces notation and the principle spaces for our solutions along with relevant results on elliptic regularity theory. Moreover, we recall an important theorem on maximal $L^p$-regularity of non-autonomous abstract Cauchy-problems, for which an additional bound is shown under suitable assumptions. \par 
The main part of this article starts in Section \ref{sec:regularized_approx}, where the Cahn--Hilliard--Biot system is enhanced with regularizing terms and a semi-Galerkin ansatz is employed to show existence of weak solutions. More precisely, we use ODE-methods to find solutions of one subsystem in linear subspaces, while maximal regularity theory yields solvability on the whole space for the other. After reducing the problem to a quasi-linear partial differential equation, a fixed-point argument exploiting the Leray--Schauder principle establishes the existence of consolidated approximate solutions. We proceed with the derivation of \textit{a priori} estimates, relying on the generalized gradient flow structure, and deduce first compactness properties. Next, we derive additional strong convergence for $p$, which follows immediately with the help of a version of the Aubin--Lions--Simon theorem once we have found an estimate for the differences of time-translations of $p$ in the space $W^{-3, 2}(\Omega)$. This allows us to also deduce strong convergence for $\bm{u}$. 
\par 
It is in Section \ref{sec:existence} that we obtain a solution to our original problem as the limit of weak solutions to the regularized problems. Without being restricted to a projection identity for the pressure, strong convergence for $p$ can be derived without the aid of a more regular $\vphi$. Finally, it remains to show strong convergence of $\theta$ without relying on the regularization, for which we exploit that the operators $(- \vr \Delta + 1)^{-1}$ pointwise converge to the identity in $L^2(\Omega)$ as $\vr \searrow 0$.

\section{Derivation of the system} \label{sec:derivation}

In this section, we carefully introduce all principle variables and derive the full Cahn--Hilliard--Biot system under investigation. In particular, we highlight the underlying balance laws, thermodynamic relations and constitutive assumptions, especially stressing the generalized gradient flow structure and inclusion of visco-elastic effects.  

\subsubsection*{Free energy}
The total energy of the system can be split into the following three parts: 
\begin{equation*}
	\F(\vphi, \bm{u}, \theta) = \F_{\textnormal{gl}}(\vphi) + \F_{\textnormal{el}}(\vphi, \bm{u}) + \F_{\textnormal{f}}(\vphi, \bm{u}, \theta),
\end{equation*}
where $\vphi$ is a phase-field variable acting as an order parameter of the diffuse interface, $\bm{u}$ is the displacement of the solid with respect to a reference configuration and $\theta$ is the volumetric fluid content. \par   
Here, $\F_{\textnormal{gl}}$ is the Ginzburg-Landau interfacial energy defined by 
\begin{equation*}
	 \F_{\textnormal{gl}}(\vphi) \coloneqq \int_\Omega \frac{\vepsilon}{2} | \nabla \vphi|^2 + \frac{1}{\vepsilon} \psi(\vphi) \dx, 
\end{equation*}
which penalizes rapid changes of concentration through the first term and deviations from the pure phases through the second one. Note that the small parameter $\vepsilon > 0$ is proportional to the thickness of the interface between the phases. A typical choice for the free energy density $\psi$ is
\begin{equation*}
	\psi(\vphi) \coloneqq \alpha (1- \vphi^2)^2, \quad \alpha \in \R_{>0}, 
\end{equation*}
featuring a double-well shape with global minima at $\psi(-1) = \psi(1) = 0$. Other examples include the so-called logarithmic potential, which becomes singular as $|\vphi| \rightarrow 1$ and was already considered by Cahn and Hilliard in the original derivation \cite{cahn1958free}, as well as the obstacle potential, see \cite{PhysRevA.38.434, blowey1991cahn}, as its pointwise limit.\\
\new{
For the promising application of the Cahn--Hilliard--Biot system in tumor growth simulations, this diffuse interface approach allows us to model malignant and healthy tissue by associating them with the values $ \pm 1$ of the order parameter $\vphi$, respectively. Moreover, the system is also fit to account for interactions between the elastic displacement in the cellular matrices and the interstitial fluid flow and pressure, which will be discussed below. 
} 
\par 
\medskip  
The second contribution to the energy is due to inherent elastic effects and takes the form 
\begin{equation*}
	\F_{\textnormal{el}}(\vphi, \bm{u}) \coloneqq \int_\Omega \C(\vphi) (\E(\bm{u})- \Tau(\vphi)):  (\E(\bm{u}) - \Tau(\vphi)) \dx , 
\end{equation*}
where $\E(\bm{u}) = \frac{1}{2} (\nabla \bm{u} + \nabla \bm{u}^{T})$ denotes the symmetrized gradient of the displacement $\bm{u}$, corresponding to the linearized strain tensor under the assumption of infinitesimal deformations. The function $\Tau$ is a suitable interpolation of the eigenstrains of the pure phases, i.e., the strain the material would attain if it was uniform in the phases associated with the values $\vphi = \pm 1$ and unstressed. Moreover, we denote by $\C(\vphi)$ the elasticity tensor, characterizing the stiffness of the material depending on the order parameter $\vphi$. 
\par 
\medskip 
Lastly, the fluid energy $\F_{\textnormal{f}}$ is given by 
\begin{equation*}
	\F_{\textnormal{f}}(\vphi, \bm{u}, \theta) \coloneqq \int_\Omega \frac{M(\vphi)}{2} (\theta - \alpha(\vphi) \nabla \cdot \bm{u})^2 \dx, 
\end{equation*}
where the compressibility $M(\vphi)$ and the pressure-deformation coupling coefficient $\alpha(\vphi)$, usually referred to as Biot--Willis coefficient, cf. \cite{biot1957elastic, SHOWALTER2000310},  are functions of the phase-field. \\
\new{Note that in order to exploit the gradient-flow structure of the system, this study assumes uniform positivity for the compressibility coefficient $M$, i.e., $M > c > 0$; since the fluid content is usually given as $\theta = c_0 p + \alpha \nabla \cdot \bm{u}$, other treatments of the Biot equation also consider the degenerate case $c_0 = 0$, cf. e.g. \cite{bociu2023mathematical, SHOWALTER2000310} and the references cited therein. 
}
\subsubsection*{Balance laws}

We assume the evolution of the phase field to adhere to 
\begin{equation*}
	\pt \vphi + \nabla \cdot \bm{J} = R, 
\end{equation*}
where $\bm{J}$ is the phase-field flux and $R$ is a reaction term. Denoting by $\bm{q}$ the fluid flux and by $S_f$ some source term, we impose the following volume balance law on the fluid
\begin{equation*}
	\pt \theta + \nabla \cdot \bm{q} = S_f. 
\end{equation*}
With the assumption that the mechanical equilibrium is attained at a much faster time scale than the diffusion processes take place, we obtain the quasi-static law 
\begin{equation*}
	- \nabla \cdot \bm{\sigma} = \bm{f}, 
\end{equation*}
where $\bm{\sigma}$ is the stress tensor, describing the balance of elastic forces within the medium and external body forces $\bm{f}$.

\subsubsection*{Thermodynamic relations}

We start by defining the chemical potential as the derivation of the total energy with respect to the phase-field variable $\vphi$
\begin{equation*}
	\mu = \delta_{\vphi} \F(\vphi, \bm{u}, \theta) =  
	\delta_{\vphi} \F_{\textnormal{gl}}(\vphi) 
	+  \delta_{\vphi}  \F_{\textnormal{el}}(\vphi, \bm{u}) 
	+ \delta_{\vphi} \F_{\textnormal{f}}(\vphi, \bm{u}, \theta) 
\end{equation*}
where 
\begin{align*}
	\delta_{\vphi} \F_{\textnormal{gl}}(\vphi) &= - \vepsilon \Delta \vphi + \psi'(\vphi), \\ 
	\delta_{\vphi} \F_{\textnormal{el}}(\vphi, \bm{u}) &= \frac{1}{2} \C'(\vphi)(\E(\bm{u}) - \Tau(\vphi)): \E(\bm{u} - \Tau(\vphi)  ) -  \C(\vphi)(\E(\bm{u}) - \Tau(\vphi)): \Tau'(\vphi),  \\ 
	\delta_{\vphi} \F_{\textnormal{f}}(\vphi, \bm{u}, \theta) &= \frac{M'(\vphi)}{2} (\theta - \alpha(\vphi) \nabla \cdot \bm{u})^2 - M(\vphi)(\theta - \alpha(\vphi) \nabla \cdot \bm{u})\alpha'(\vphi)\nabla \cdot \bm{u}. 
\end{align*}
\par 
\medskip 
Assuming the solid material to exhibit Kelvin--Voigt type visco-elastic behavior, we can decompose the total stress $\bm{\sigma}$ into the elastic stress $\bm{\sigma}_{e}$ and a visco-elastic part $\bm{\sigma}_v$, where the former is defined by
\begin{equation*}
	\bm{\sigma}_{e} = \delta_{\bm{u}} \F(\vphi, \bm{u}, \theta) =  
	\delta_{\bm{u}}  \F_{\textnormal{el}}(\vphi, \bm{u})
	+ \delta_{\bm{u}} \F_{\textnormal{f}}(\vphi, \bm{u}, \theta) 
\end{equation*}
with 
\begin{align*}
	\delta_{\bm{u}} \F_{\textnormal{el}}(\vphi, \bm{u}) &= \C(\vphi) (\E(\bm{u} - \Tau(\vphi))),  \\ 
	\delta_{\bm{u}} \F_{\textnormal{f}}(\vphi, \bm{u}, \theta) &= M(\vphi) \alpha (\vphi) (\theta - \alpha (\vphi) \nabla \cdot \bm{u}) \bm{I}. 
\end{align*}
Denoting by $\C_{\nu}(\vphi)$ the phase-field dependent modulus of visco-elasticity, the viscous contribution is given as 
\begin{equation*}
	\bm{\sigma}_v = \C_{\nu}(\vphi) \E( \pt \bm{u}). 
\end{equation*}
For the inclusion of viscous effects in this manner, we to refer to \cite[Sec.~2.1]{BOCIU2023127462} and the references cited therein. These ideas are also exploited in \cite{bociu2016analysis} for the analysis of a nonlinear poro-visco-elastic model. Moreover, similar assumptions are made in \cite[Sec. 4]{both2019gradient}, with the difference that the authors distinguish between the standard elastic strain and stored visco-elastic energy.  \par 
\medskip 
Lastly, the pore pressure is given as the derivation of the total energy with respect to volumetric fluid content 
\begin{equation*}
	p  = \delta_{\theta} \F_{\textnormal{f}}(\vphi, \bm{u}, \theta)  = M(\vphi) (\theta - \alpha(\vphi) \nabla \cdot \bm{u}). 
\end{equation*}

\subsubsection*{Constitutive equations}
We assume the phase-field flux to be given by
\begin{equation*}
	\bm{J} = - m(\vphi) \nabla \mu
\end{equation*}
according to Fick's law. Moreover, the fluid flux is required to obey Darcy's law, i.e., 
\begin{equation*}
	\bm{q} = - \kappa(\vphi) \nabla p.
\end{equation*}
Here, the mobility $m(\vphi)$ and the permeability $\kappa(\vphi)$ are both functions of the phase-field. 
\par 
\medskip 

\subsubsection*{Gradient flow structure}
The system at hand was derived by Storvik et al. in \cite{STORVIK2022107799} as a generalized gradient flow of the energy 
	\begin{align}
		\F (\vphi, {\bm{u} }, {\theta} ) 
		= \int_\Omega \frac{\vepsilon}{2} |\nabla \vphi|^2 + \frac{1}{\vepsilon} \psi(\vphi) \dx 
		+ \int_\Omega W(\vphi, \E(\bm{u})) \dx 
		+ \int_\Omega  \frac{M(\vphi)}{2} (\theta - \alpha(\vphi) \nabla \cdot \bm{u})^2 \dx, 
	\end{align}
	modeling the flow of a fluid through a deformable porous media consisting of multiple phases. We wish to point out that they did not include the linear visco-elastic term $\nabla \cdot ( \C_\nu \E(\pt \bm{u}))$, which is purely dissipative and therefore preserves the generalized gradient flow structure.\\ 
	\begin{subequations} 
	This work investigates the following system of partial differential equations: 
	\begin{align}
		\partial_t \vphi  &= \nabla \cdot \left( m(\vphi ) \nabla \mu \right) + R(\vphi, \E(\bu), \theta)  \label{eq:strong_formulation_phi} & & \textrm{in }  (0, T) \times \Omega ,\\
		&\begin{aligned}[b]
			\! \! \! \mu = -\vepsilon \Delta \vphi + \frac{1}{\vepsilon} \psi'(\vphi) + W_{,\vphi} (\vphi, \E(\bu)) &-M(\vphi) (\theta - \alpha(\vphi) \nabla \cdot \bm{u}) \alpha'(\vphi) \nabla \cdot \bm{u}  \\ 
			& +  \frac{M'(\vphi)}{2} (\theta - \alpha(\vphi) \nabla \cdot	 \bm{u})^2
		\end{aligned} & & \textrm{in } (0, T) \times \Omega , \label{eq:strong_formulation_mu} \\ 
		 \nabla \cdot \bm{\sigma} &= \bm{f} & & \textrm{in } (0, T) \times \Omega , \label{eq:strong_formulation_u} \\ 
		\bm{\sigma} &= W_{,\E} (\vphi, \E(\bu))  +\C_{\nu}(\vphi)  \E( \partial_t \bm{u})  - \alpha (\vphi) M(\vphi) (\theta - \alpha(\vphi) \nabla \cdot \bm{u})\, \bm{I} & & \textrm{in } (0, T) \times \Omega , \\ 
		\partial_t \theta &=  \nabla \cdot (\kappa(\vphi) \nabla p) + S_f(\vphi, \E(\bu), \theta) 	& & \textrm{in } (0, T) \times \Omega , \label{eq:strong_formulation_theta}\\   
		p& = M(\vphi) (\theta - \alpha (\vphi) \nabla \cdot \bm{u}) & & \textrm{in }  (0, T) \times \Omega \label{eq:strong_formulation_p}
	\end{align}
	together with the boundary conditions
	\begin{align}
		\nabla \vphi \cdot \bm{n} &= 0  \quad  \textrm{and} \quad 
		\nabla \mu \cdot \bm{n} = 0  \quad  &\textrm{on }  (0, T) \times \partial \Omega, \\ 
		\bu &= \bm{0}   \quad \textrm{and} \quad 
		 p = 0 \quad   &\textrm{on }  (0, T) \times \Gamma_D,\\
		 \bm{\sigma}\,  \bm{n} &= \bm{g}  \quad  &\textrm{on }  (0, T) \times \Gamma_N, \\ 
		\nabla p \cdot \bm{n} &= 0 \quad  &\textrm{on }  (0, T) \times \Gamma_N
	\end{align}
	and the initial conditions 
	\begin{align}
		\vphi (0) = \vphi_0  \quad  \textrm{and} \quad
		\bm{u}(0) = \bm{u}_0 \quad  \textrm{and} \quad 
		\theta(0) = \theta_0  \quad  \textrm{in } \Omega . 
	\end{align}
	Note that here, we split the boundary in a Dirichlet part $\Gamma_D$ and a Neumann part $\Gamma_N$.  \\ 
		For sufficiently smooth solutions, we obtain the following energy dissipation differential inequality 
	\begin{align*}
		\frac{d}{dt} \mathcal{F}(\vphi, \bm{u}, \theta) 
		+ \norm{\pt \bm{u}}_{\bm{H}^1}^2 
		&+  \norm{\nabla \mu}_{L^2}^2 
		+  \norm{\nabla p}_{L^2}^2\\ 
		& \leq 
		C\left( \int_\Omega R(\vphi, \E(\bm{u}), \theta) \mu \dx  +\norm{S_f(\vphi, \E(\bm{u}), \theta)}_{L^2}^2 +  \norm{\bm{f}}_{\bm{L}^2}^2 +   \norm{\bm{g}}_{\bm{L}^2(\Gamma_N)}^2 \right). 
	\end{align*}  
	Note that since we can only use the Poincarè--Wirtinger inequality for $\mu$, in order to obtain an estimate for the term $\int_\Omega R(\vphi, \E(\bm{u}))\mu \dx$, we first have to estimate the mean value $\dashint \mu  \coloneqq \dashint_\Omega \mu \dx$ . 
	
	\new{
	\begin{remark}
		\begin{enumerate}[nosep, label*=(\roman*)]
			\item Note that in contrast to the partial visco-elastic regularization $\nabla \cdot (\eta \pt (\nabla \cdot \bm{u}))$ employed by Riethmüller et al. in \cite{riethmüller2023wellposedness}, which only accounts for volumetric effects, we exploit full Kelvin--Voigt visco-elasticity via $\nabla \cdot (\C_{\nu}(\vphi)  \E( \partial_t \bm{u}))$, utilizing the full stain rate tensor. Though this assumption is slightly stronger, we obtain the fully parabolic evolution equation \eqref{eq:strong_formulation_u} for the displacement $\bm{u}$, giving us the chance to apply maximal regularity theory and set up helpful tools for further investigations regarding the existence of strong solutions. \par 
			Moreover, we stress that visco-elastic effects are physically meaningful and do not only serve as a mathematical regularization. \\  
			Lastly, let us point out that full visco-elasticity is sorely needed to deduce invertibility and boundedness of the operators $ \bm{v} \mapsto \int_\Omega \C_{\nu}(\vphi ) \E(\bm{v}): \E(\cdot) \dx$ in suitable spaces, cf. \eqref{iq:operators_uniform}, which is essential for the application of maximal regularity and the fixed point argument. In particular, the main arguments yielding compactness for the pressure $p$, and hence for $\bm{u}$ and $\theta$, cf. Lemma~\ref{lem:p_strong}, Lemma~\ref{lem:convergence_u_reg}, Lemma~\ref{lem:p_strong_non_regular}, Lemma~\ref{lemma:u_strong_not_reg}, do not rely on this property and can easily be adopted for the more general case in \cite{riethmüller2023wellposedness}. 
			
			\item We continue with a short discussion, whether other nonlinear dependencies, specifically in the permeability $\kappa$, could be allowed, as is the case in other studies, e.g., \cite{MR1876882, bociu2016analysis, van2023mathematical}. While there are some technicalities, in particular in the proof of the "Uniqueness and continuous dependence"-part of Lemma \ref{lemma:existence_galerkin}, our strong convergence results could accommodate a permeability of the form 
			\begin{equation*}
					\kappa = \kappa (\vphi, \nabla \cdot \bm{u}, \theta, p). 
			\end{equation*} 
			Permitting similar dependencies in $\alpha, M$ would give raise to chain rules due to the gradient flow structure and (possibly) significantly complicate the system.  
			\item Observe that by inserting \eqref{eq:strong_formulation_p} into \eqref{eq:strong_formulation_theta} we could eliminate the pressure $p$ altogether, but would obtain an equation that is of third order in the displacement $\bm{u}$. As regularity theory for the elastic subproblem is rather intricate, splitting the system as in \eqref{eq:strong_formulation_phi}-\eqref{eq:strong_formulation_p} is analytically preferable. 
		\end{enumerate}
 	\end{remark}
	}
\end{subequations}


\section{Assumptions and main theorem}\label{sec:assumptions} 

We start by introducing all necessary assumptions and the notion of a weak solution to the Cahn--Hilliard--Biot system. Moreover, we precisely state the main theorem we aim to establish in this work.  

\begin{enumerate}[label = {(A\arabic*)} ]
	\item  \label{A:domain}
	Let $\Omega \subset \mathbb{R}^n$ be a bounded and connected $C^{1,1}$-domain in dimension $n \leq 3$;
	 additionally, we assume that the subset $\Gamma_D \subset \partial \Omega$ is relatively closed and satisfies $\mathcal{H}^{n-1} (\Gamma_D) > 0$. Setting $\Gamma_N\coloneqq \partial \Omega \setminus \Gamma_D$, it holds that $\Gamma_N \cap \Gamma_D = \emptyset $ and $ \Gamma_N \cup \Gamma_D = \partial \Omega $, which we further require to be regular in the sense of Gröger, cf. \cite[Def. 2]{groger1989aw}, Remark \ref{remark:assumptions}\ref{groeger}.  
	\item  \label{A:psi} 
	The potential is of the form
	\begin{equation*}
		\psi(\vphi) = \psi_1(\vphi) + \psi_2(\vphi) \quad \textrm{for all} \quad \vphi \in \mathbb{R},
	\end{equation*}	
	with $\psi_1, \psi_2 \in C^1 (\mathbb{R}, \mathbb{R})$ and $\psi_1$ convex. Moreover, $\psi'$ is required to be Lipschitz continuous on all bounded intervals.  
	\begin{enumerate} [label= {(A\arabic{enumi}.\arabic*)}]
		\item \label{A:psi_0}
		The potential is non-negative, i.e., 
		\begin{equation*}
			\psi(z) \geq 0  \quad \textrm{for all} \quad z \in \mathbb{R}. 
		\end{equation*}
		\item \label{A:psi_1}
		We further require $\psi_1$ to satisfy the following growth conditions: there exist $\gamma_{\psi_1}, c_{\psi_1} > 0 $ and $p > 2$ such that 
		\begin{align*}
			\gamma_{\psi_1} \abs{z} ^p -c_{\psi_1} \leq \psi_1(z), \\ 
			|\psi'(z)| \leq \rho_{\psi_1} \psi(z)+ C_{\rho_{\psi_1}}
		\end{align*}
		for all $z \in \R$, $\rho_{\psi_1} > 0$ and constants $c_{\psi_1}, C_{\rho_{\psi_1}} > 0$, where the latter may depend on $ \rho_{\psi_1}$. 
		\item  \label{A:psi_2}
		There exists a constant $C_2 > 0$ such that 
		\begin{equation*}
			\abs{ \psi''_2(z)}  \leq C_2   \quad \textrm{for all} \quad z \in \mathbb{R}. 
		\end{equation*} 
	\end{enumerate}
	\item \label{A:W}
	The elastic free energy density $W \in C^1(\mathbb{R} \times \mathbb{R}^{n \times n}_{sym})$ is of the form 
	\begin{equation*}
		W(\vphi', \E') = \C(\vphi') (\E' - \Tau(\vphi')):  (\E' - \Tau(\vphi')), 
	\end{equation*} 
	where $\C: \R \rightarrow \mathcal{L}(\R^{n \times n}_{sym})$ is a bounded, Lipschitz continuous and differentiable tensor whose derivative $\C'$ is also bounded and Lipschitz continuous. We require it to fulfill the standard assumptions of linear elasticity, i.e., $\C(\vphi')$ is  symmetric and uniformly positive definite on $\R^{n \times n}_{sym}$, mapping symmetric matrices to symmetric matrices such that 
	\begin{align*}
		\E: \C(\vphi') \E &\geq c |\E|^2,\\ 
		\mathcal{D} : \C(\vphi') \E &= \C(\vphi')\mathcal{D} : \E
	\end{align*}
	for all symmetric matrices $\E, \mathcal{D} \in \R^{n\times n}_{sym}$ and all $\vphi' \in \R$.  \\ 
	The eigenstrain $\Tau : \R \rightarrow \R^{n \times n}_{sym}$ is a Lipschitz continuous, differentiable, matrix-valued function with Lipschitz continuous derivative $\Tau'$,  such that $\Tau(\vphi)$ is symmetric for any $\vphi' \in \R$.
	\item \label{A:C_nu} The modulus of visco-elasticity $\C_{\nu}$ satisfies the same assumptions as the elasticity tensor $\C$. 
	\item \label{A:g}
	The function $\bm{g} : \Gamma_N \rightarrow \mathbb{R}^n$, modeling applied outer forces, fulfills $\bm{g} \in \bm{L}^2(\Gamma_N)$.
	\item \label{A:f}
	The function $\bm{f} \colon \Omega \rightarrow \R^n$, modeling body forces, satisfies $\bm{f} \in \bm{L}^2(\Omega)$.
	\item \label{A:m}
	There exist constants $\underline{m}, \overline{m} > 0$ such that the mobility $m \in C^{0, 1}(\mathbb{R})$ fulfills 
	\begin{equation*}
		\underline{m} \leq m(z) \leq \overline{m} \quad \textrm{for all} \quad z \in \mathbb{R}. 
	\end{equation*}
	\item \label{A:kappa}
	There exist constants $\underline{\kappa}, \overline{\kappa} > 0$ such that the function $\kappa \in C^{0, 1}(\mathbb{R})$ fulfills 
	\begin{equation*}
		\underline{\kappa} \leq \kappa(z) \leq \overline{\kappa} \quad \textrm{for all} \quad z \in \mathbb{R}. 
	\end{equation*}
	\item \label{A:phase_coefficients}
	The maps $\alpha, M \in C^{1, 1}(\R)$ are non-negative and satisfy
	\begin{align*}
		 \alpha(z) + |\alpha'(z)| \leq \overline{\alpha}  &\quad  \textrm{for all} \quad z \in \mathbb{R} \quad \textrm{and some } \overline{\alpha} > 0, \\ 
		 \underline{M} \leq M(z) \leq \overline{M}  &\quad  \textrm{for all} \quad z \in \mathbb{R} \quad \textrm{and some } \overline{M}, \underline{M}  > 0. 
	\end{align*}
	We further assume $M'$ to be bounded, i.e., $|M'(z)| \leq \overline{M}$ for all $z \in \mathbb{R}$. 
	\item \label{A:source_terms}
	The source terms $R, S_f$ are in the space $C^{0, 1}(\R \times \R^{n \times n} \times \R)$. Moreover, we require $R$ to be a bounded function.
	\item \label{A:initial}
	The initial datum $\vphi_0$ satisfies $\vphi_0 \in H^1(\Omega)$ and $\int_\Omega \psi(\vphi_0) d\bm{x} < \infty$. 
\end{enumerate}

\begin{remark} \label{remark:assumptions}
	
	\begin{enumerate}[nosep, label*=(\roman*)]
		\item \label{groeger} \new{
			The subset $\Omega \cup \Gamma_N$ is regular in the sense of Gröger, cf. \cite[Def. 2]{groger1989aw}, if for every $y \in \partial \Omega$ there exists an open neighborhood $U \subset \R^n$ of $y$ and a bi-Lipschitz map $\Phi : U \rightarrow \R^n$ such that $\Phi(y) = 0$ and $\Phi(U \cap ( \Omega \cup \Gamma_N))$ equals one of the following sets: 
			\begin{align*}
				E_1 &\coloneqq \{ x \in \R^n \colon \abs{x} < 1, x_n < 0\}, \\ 
				E_2 &\coloneqq \{ x \in \R^n \colon \abs{x} < 1, x_n \leq 0\}, \\ 
				E_3 &\coloneqq \{ x \in \R^n \colon x_n < 0 \textnormal{ or } x_1 > 0\}. 
			\end{align*}
			In particular, it is explicitly allowed that $\overline{\Gamma_D} \cap\overline{ \Gamma_N}  \neq \emptyset$. 
		}
		\item The following results are immediate consequences of the assumptions above. 
		
		\begin{enumerate}[label = (\roman*)]
			\item [{(A3.1)}] \label{A:W_strongly_monotone}
			The mapping $\E' \mapsto W_{,\E}(\vphi', \E')$ is strongly monotone uniformly in $\vphi'$ in the following sense: there exists a $C_1 > 0$ such that for all symmetric $\E_1', \E_2' \in  \mathbb{R}^{n \times n} $ it holds, 
			\begin{equation*}
				\left(W_{,\E}(\vphi', \E_2') - W_{,\E}(\vphi', \E_1') \right) : \left( \E_2' - \E_1' \right) \geq C_1 \abs{\E_1' -\E_2'}^2. 
			\end{equation*}
			\item [ {(A3.2)}] \label{A:W_growth_conditions}
			There exists a constant $C_2 > 0$ such that for all $\vphi' \in \mathbb{R}$ and all symmetric $\E' \in  \mathbb{R}^{n \times n}$,
			\begin{align*}
				\abs{W(\vphi', \E') } & \leq C_2 \left( \abs{\E'}^2 + \abs{\vphi'}^2 +1 \right), \\
				\abs{W_{,\vphi}(\vphi', \E') } & \leq C_2 \left( \abs{\E'}^2 + \abs{\vphi'}^2 +1 \right), \\
				\abs{W_{,\E}(\vphi', \E')} & \leq C_2 \left( \abs{\E'} + \abs{\vphi'} +1 \right). 
			\end{align*}
			\item [(ii.i)]
			Using the fundamental theorem of calculus, \hyperref[A:W_strongly_monotone]{(A3.1)}, \hyperref[A:W_growth_conditions]{(A3.2)} and Young's inequality, it easily follows that 
			\begin{align} \label{ineq:W}
				W(\vphi', \E') 
				& \geq -C' \left( \abs{\vphi'}^2 +1   \right) + \frac{C_1}{4} \abs{\E'}^2. 
			\end{align}
		\end{enumerate}
	\end{enumerate}
\end{remark}

\par 
\medskip 

A weak solution of the system derived in Section \ref{sec:derivation} is defined as follows, where the notation for the function spaces will be defined in the following section. 

\begin{definition}[Weak solutions]\label{def:weak_solution}
	For any $T  > 0$ a quintuple $(\vphi, \mu, \bm{u}, \theta, p)$ with the properties 
	\begin{align*}
		\vphi &\in L^2(0, T; H^1(\Omega)) \cap H^1(0, T; H^1(\Omega)'), \\ 
		\mu  &\in L^2(0, T; H^1(\Omega)), \\ 
		\bm{u} & \in H^1(0, T; \bm{H}^1_{\Gamma_D}(\Omega)), \\ 
		\theta &\in L^2(\OT) \cap \new{H^1}(0, T; (H^{1}_{\Gamma_D}(\Omega))'), \\ 
		p &\in L^2(0, T; H^1_{\Gamma_D}(\Omega ))
	\end{align*}
	is a weak solution to the Cahn--Hilliard--Biot system if the following equations are satisfied: 
	\begin{subequations}
		\begin{align}\label{eq:weak1}
			\int_0^T  \prescript{}{(H^1)'}{}\langle \pt \vphi, \zeta \rangle_{H^1}\, \dt =   \int_{\OT} - m(\vphi) \nabla \mu \cdot \nabla \zeta + R(\vphi, \E(\bu), \theta) \zeta \; \dtx
		\end{align}
		for all $\zeta \in L^2(0, T; H^1(\Omega) )$;
		\begin{align*}\label{eq:weak2}
			\int_{\OT} \mu\, \zeta \dtx 
			=   \int_{\OT} \vepsilon \nabla \vphi \cdot \nabla \zeta &+
			\Big[ \frac{ 1}{\vepsilon} \psi'(\vphi)  
			 \begin{aligned}[t]
			 	&+ \Wp(\vphi, \E(\bm{u}))  
				+ \frac{M'(\vphi)}{2} (\theta - \alpha(\vphi) \nabla \bm{u})^2  \\ 
				&-M(\vphi) (\theta - \alpha(\vphi) \nabla \cdot \bm{u}) \alpha'(\vphi) \nabla \cdot \bm{u} \Big]\,   \zeta \; \dtx
			\end{aligned}		
			\numberthis
		\end{align*}
		for all $\zeta \in L^\infty(\OT) \cap L^2(0, T; H^1(\Omega))$;
		\begin{align*} \label{eq:weak3}
			&\int_{\OT} \C_\nu (\vphi) \E(\partial_t \bm{u}) :  \E(\bm{\eta}) + \WE(\vphi, \E(\bm{u})) : \E(\bm{\eta}) - M(\vphi) \alpha(\vphi) (\theta - \alpha(\vphi) \nabla \cdot \bm{u}) \nabla \cdot \bm{\eta} \dtx \\ 
			& \quad = \int_{\OT} \bm{f} \cdot \bm{\eta} \dtx + \int_0^T \int_{\bm{\Gamma_N}} \bm{g} \cdot \bm{\eta} \dH \dt \numberthis 
		\end{align*}
		for all $\bm{\eta} \in L^2(0, T; \bm{H}^1_{\Gamma_D}(\Omega))$;
		\begin{align}\label{eq:weak4}
			\int_0^T  \prescript{}{(H^1_{\Gamma_D})'}{}\langle \pt \theta, \xi \rangle_{H^1_{\Gamma_D}}\, \dt =   \int_{\OT} - \kappa(\vphi) \nabla p \cdot \nabla \xi + S_f(\vphi, \E(\bu), \theta) \xi \; \dtx
		\end{align}
		for all $\xi \in L^2(0, T; H^1_{\Gamma_D}(\Omega))$; 
		\begin{align}\label{eq:weak5}
		p = M(\vphi) (\theta - \alpha (\vphi) \nabla \cdot \bm{u}) \quad \textrm{a.e. in } \ \ \Omega_T 
		\end{align}
		and 
		\begin{equation}\label{eq:weak_initial}
			\theta(0) = \theta_0 \quad \textrm{in} \quad (H^1_{\Gamma_D}(\Omega))' \quad \textrm{and} \quad 
			(\vphi, \bm{u})_{|t= 0} = (\vphi_0, \bm{u}_0) \quad \textrm{a.e. in } \ \ \Omega
		\end{equation}
	\end{subequations}	
\end{definition} 
\par 
\medskip 
The primary goal of this work is to show the following existence theorem. In what follows, we refer to later sections for the definitions of the function spaces involved.  

\begin{theorem}\label{th:existence}
	Assume $T > 0$ and that  \ref{A:domain}-\ref{A:initial} hold. Moreover, suppose that $\theta_0 \in L^2(\Omega)$ and $\bm{u}_0 \in \bm{H}^1_{\Gamma_D}(\Omega)$. Then, there exists at least one weak solution to the Cahn--Hilliard--Biot system in the sense of Definition \ref{def:weak_solution} such that 
		\begin{align*}
		\vphi &\in L^\infty(0, T; H^1(\Omega)) \cap H^1(0, T; H^1(\Omega)'), \\ 
		\mu  &\in L^2(0, T; H^1(\Omega)), \\ 
		\bm{u} &\in  H^1(0, T; \bm{H}^1_{\Gamma_D}(\Omega)), \\ 
		\theta &\in L^\infty(0, T; L^2(\Omega)) \cap\new{ H^1}(0, T; (H^{1}_{\Gamma_D}(\Omega))'), \\ 
		p &\in L^2(0, T; H^1_{\Gamma_D}(\Omega )) \cap L^\infty(0, T; L^2(\Omega)). 
	\end{align*}
	Moreover, the weak solution satisfies the following estimate 
		\begin{align*}
		\norm{\vphi}_{L^\infty(H^1)} 
		&+ \norm{\vphi}_{H^1((H^1)')}
		+ \norm{\mu}_{L^2(H^1)}
		+ \norm{\psi(\vphi)}_{L^\infty(L^1)} \\ 
		&+ \norm{\bm{u}}_{H^1(\bm{H}_{\Gamma_D}^1)}
		+  \norm{\theta}_{L^\infty(L^2)} 
		+ \norm{\theta}_{H^1((H^{1}_{\Gamma_D})')}
		+ \norm{p}_{L^2(H^1_{\Gamma_D})}
		\leq C, 
	\end{align*}
	where $C > 0$ only depends on the data and on $T$. \\ 
	 If $\bm{u}_0 \in \bm{W}^{1, r}_{\Gamma_D}(\Omega)$ with $r > 2$, then there exists some $2 < q \leq p$ such that $\bm{u} \in  L^2(0, T; \bm{W}^{1, q}_{\Gamma_D}(\Omega))$. 
\end{theorem}


\section{Notation and preliminaries} \label{sec:preliminaries}

This section begins with the introduction of relevant function spaces, particularly Sobolev spaces and Bessel potential spaces, along with associated notation and embedding properties.
We proceed by recalling a theorem on elliptic regularity theory, which will be used later on. This is followed by a brief excursion into maximal $L^p$-regularity theory, where we revisit an existence theorem for non-autonomous abstract Cauchy-problems by Arendt et al. \cite{ARENDT20071} and, under similar assumptions as we will encounter when investigating the Cahn--Hilliard--Biot model, extend the proof by an additional estimate for the solutions in terms of the right-hand side. Lastly, we also recall the Leray--Schauder principle. \par 
\medskip 
For brevity we write $\OT$ instead of $(0, T) \times \Omega$ for any $T > 0$. 

\subsubsection*{Notation}

For $p \in [1, \infty]$, we denote by $W^{1, p}(\Omega)$ the usual Sobolev spaces and use the abbreviations
\begin{equation*}
	H^1(\Omega) \coloneqq W^{1, 2}(\Omega), \quad H^2_{\bm{n}}(\Omega) \coloneqq \{ f \in W^{2, 2}(\Omega):  \nabla f \cdot \bm{n}  = 0 \quad \textrm{on } \partial \Omega\}, 
\end{equation*}
where $\bm{n}$ is the outer normal vector on $\partial \Omega$. 
Moreover, we define 
\begin{equation*}
	W^{1, p}_{\Gamma_D}(\Omega) \coloneqq \overline{\{ f_{| \Omega} : f \in C^\infty_c(\R^n, \R), \ \supp f \cap \Gamma_D = \emptyset \}}^{\norm{\cdot}_{W^{1, p}}}, 
	\quad 
	W^{-1, p}_{\Gamma_D} (\Omega) \coloneqq (W^{1, p}_{\Gamma_D} (\Omega))'
\end{equation*}
and set 
\begin{equation*}
	X (\Omega) \coloneqq W^{1, 2}_{\Gamma_D}(\Omega), \quad 
	\bm{W}^{1, p}_{\Gamma_D} (\Omega) \coloneqq \prod_{i = 1}^n W^{1, p}_{\Gamma_D}(\Omega), \quad 
	\bm{X} (\Omega) \coloneqq \bm{W}^{1, 2}_{\Gamma_D}(\Omega).
\end{equation*}
We further denote by $\bm{W}^{-1, p}_{\Gamma_D}(\Omega)$ the duals of the vector-valued Sobolev spaces and also use the abbreviation $X'(\Omega)$ and $H^1(\Omega)'$ for the duals of $X(\Omega)$ and $H^1(\Omega)$, respectively. \\ 
Finally, we write $\prescript{}{\cdot} \langle \cdot , \cdot \rangle_{\cdot}$ for the duality product between spaces specified in the indices and denote by $\langle\cdot , \cdot \rangle$ the $L^2$-inner product. \\ 
\new{For a given function, e.g., $f(x, y)$ with multiple variables $x, y$, we express the directional derivatives by $  f_{,x}=  f_x = \frac{\partial }{\partial x}f$ and $f_{,x} = f_y =  \frac{\partial }{\partial y} f$, respectively.}

\begin{definition}[Bessel-potential spaces, {\cite[Def. 2.42]{Abels+2012}, \cite[Thm. 2.3.3]{triebel1978interpolation}, \cite[Sec. 4]{haller2019higher}}]\hfill 
	\begin{enumerate}[nosep, label = (\roman*)] 
		\new{\item 	Let $s \in \R$ and let $1 < p < \infty$. Then the ($L^p$)-Bessel potential space of order $s$ is denoted as $H^{s,p}(\R^n)$.}
		\item For $\frac{1}{p} < s < 1+ \frac{1}{p}$ set
		\begin{equation*}
			H^{s, p}_{\Gamma_D} (\R^n) \coloneqq \Biggl\{ f \in H^{s, p}(\R^n) : 
			\lim_{r \searrow 0} \frac{1}{|B_r(x)|} \int_{B_r(x)} f (y) \dy= 0 \quad \textrm{for } \mathcal{H}^{d-1}\textrm{-a.e. } x \in \Gamma_D
			\Biggr\}
		\end{equation*}
		with $\norm{\cdot}_{H^{s, p}_{\Gamma_D} (\R^n)} = \norm{\cdot}_{H^{s, p} (\R^n)} $. 
		\item Finally, set $H^{s, p}_{\Gamma_D}(\Omega) \coloneqq \{ f_{| \Omega} : f \in H^{s, p}_{\Gamma_D} (\R^n) \} $ endowed with the norm
		\begin{equation*}
			\norm{f}_{H^{s, p}_{\Gamma_D}(\Omega) } \coloneqq \inf \biggl\{ \norm{g}_{H^{s, p}_{\Gamma_D}(\R^n)} : g \in H^{s, p}_{\Gamma_D}(\R^n), g_{| \Omega} = f \biggr\}. 
		\end{equation*}
	\end{enumerate}
\end{definition}

\begin{remark}
	\begin{enumerate}[nosep, label*=(\roman*)]
		\item As before, we define $\bm{H}^{s, p}_{\Gamma_D}(\Omega) \coloneqq \prod_{i = 1}^{n} H^{s, p}_{\Gamma_D}(\Omega)$. 
		\item To shorten notation, we set $H^{s}_{\Gamma_D}(\Omega) \coloneqq H^{s, 2}_{\Gamma_D}(\Omega)$ and $\norm{\cdot}_{H^{s}_{\Gamma_D}} \coloneqq\norm{\cdot}_{H^{s, 2}_{\Gamma_D}(\Omega)}$, with the analogous conventions for $\bm{H}^{s}_{\Gamma_D}(\Omega)$. 
		\item As discussed in \cite[Prop. 1]{haller2019higher}, \cite[Cor. 3.8]{EGERT20141419}, the spaces $W^{1, 2}_{\Gamma_D}(\Omega)$ and $H^{1}_{\Gamma_D}(\Omega)$ are isomorphic. Therefore, it holds $X(\Omega) \cong H^{1}_{\Gamma_D}(\Omega)$ and $\bm{X}(\Omega) \cong \bm{H}^{1}_{\Gamma_D}(\Omega)$. 
		\item It is well known, e.g. cf. \cite[Ch. 2.5]{edmunds_triebel_1996}, that the embedding 
		\begin{equation*}
			H^{1+ \delta, p}(\Omega) \xhookrightarrow{cpt} H^{1, p}(\Omega)
		\end{equation*}
		is compact.
	\end{enumerate}
\end{remark}

\subsubsection*{Elliptic regularity}

Treating problems in linear elasticity with classic elliptic theory yields a unique solution in $\bm{H}^1_{\Gamma_D}(\Omega)$. The following theorem is more subtle, allowing us to find solutions in $\bm{W}^{1, p}_{\Gamma_D}(\Omega)$,
and provides uniform estimates in terms of the right-hand side. 

\begin{theorem}[{\cite[Thm.~1.1, Prop.~1.2]{HERZOG2011802}}] \label{thm:elliptic_sobolev}
	Assume that $\Omega$ satisfies \ref{A:domain} and let $\mathcal{M}$ be a set of (nonlinear) measurable functions $\bm{c}: \Omega \times \R^{n \times n}_{sym} \rightarrow \R^{n \times n}_{sym}$ such that for all $\bm{c} \in \mathcal{M}$ it holds
	\begin{enumerate}[label*=(\roman*), nosep]
		\item $\bm{c}(\cdot, \bm{0}) \in L^\infty(\Omega,  \R^{n \times n}_{sym})$, 
		\item $\bm{c}(\cdot, \E')$ is measurable for all $\E' \in \R^{n \times n}_{sym}$, 
		\item $(\bm{c}(\bm{x}, \E'_1)- \bm{c} (\bm{x}, \E'_2)): (\E'_1 - \E'_2) \geq c_{\mathcal{M}} | \E'_1 - \E'_2|^2 $, 
		\item $|\bm{c}(\bm{x}, \E'_1) - \bm{c}(\bm{x}, \E'_2)| \leq C_{\mathcal{M}} | \E'_1 - \E'_2|$
	\end{enumerate}
	for almost all $\bm{x} \in \Omega$ and all $\E'_1, \E'_2 \in \R^{n \times n}_{sym}$ with constants $0 < c_{\mathcal{M}} \leq C_{\mathcal{M}}$ independent of $\bm{c} \in \mathcal{M}$. Then there exists some $p^* > 2$ such that for all $p \in [2, p^*]$ the operators 
	\begin{equation*}
		\mathfrak{C}: \bm{W}^{1, p}_{\Gamma_D}(\Omega) \rightarrow \bm{W}^{-1, p}_{\Gamma_D}(\Omega), \quad 
		\bm{v} \mapsto \int_\Omega \bm{c}(\bm{x}, \E(\bm{v})) : \E(\cdot) \dx 
	\end{equation*}
	are continuously invertible. Moreover, their inverses share a uniform Lipschitz constant, i.e., there exists some $L_{\mathcal{M}} > 0$ such that for all $f_1, f_2 \in \bm{W}^{-1, p}_{\Gamma_D}(\Omega)$
	\begin{equation*}
		\norm{\mathfrak{C}^{-1}f_1 - \mathfrak{C}^{-1}f_2 }_{\bm{W}^{1, p}_{\Gamma_D}} \leq L_{\mathcal{M}} \norm{f_1 - f_2}_{\bm{W}^{-1, p}_{\Gamma_D}}. 
	\end{equation*}
\end{theorem}

\begin{remark}\label{rem:applicability_elliptic}
	Since the functions $\C$ and $\C_{\nu}$ satisfy \ref{A:W} and \ref{A:C_nu}, respectively, the assumptions of Theorem~\ref{thm:elliptic_sobolev} are fulfilled with constants $c_{\mathcal{M}}, C_{\mathcal{M}}$ independent of $\vphi' \in L^2(\Omega)$. Moreover, as $C(\vphi'), \C_{\nu}(\vphi')$ are linear maps on $\R^{n \times n}_{sym}$ for all $\vphi' \in \R$, the induced operators are topological isomorphisms of $ \bm{W}^{1, p}_{\Gamma_D}(\Omega)$ and $ \bm{W}^{-1, p}_{\Gamma_D}(\Omega)$ with a common bound for the inverses. 
\end{remark}

\subsubsection*{Maximal $L^p$-regularity}

The following definitions and results are taken from \cite{ARENDT20071}. Here, we assume $D, Y$ to be Banach spaces such that the embedding $D \xhookrightarrow{d} Y$ is dense. 

\begin{definition}
	An operator $A \in \mathcal{L}(D, Y)$ is said to have $L^p$-\textit{maximal regularity} for $p \in (1, \infty)$ if for some interval $(a, b)$ and all $f \in L^p(a, b; Y)$ there exists a unique $u \in W^{1, p}(a, b; Y) \cap L^p(a, b; D)$ satisfying
	\begin{equation*}
		\pt u + Au = f \quad \textit{a.e. on } (a, b), \quad u(a) = 0. 
	\end{equation*}
	In this case, $A$ already has $L^p$-\textit{maximal regularity} for all bounded intervals and all $p \in (1, \infty)$, cf. \cite{ARENDT20071}, and we write $A \in \mathcal{MR}$. 
\end{definition}

The following theorem is a modification of \cite[Thm.~2.7]{ARENDT20071} and yields a uniform estimate for the solution under some additional assumptions. 

\begin{theorem}\label{th:max_reg_non_autonomous}
	Let the family $\mathcal{A}_i : [0, \tau] \rightarrow \mathcal{L}(D, Y)$, $i \in \mathcal{I}$, be strongly measurable, $\tau > 0$, and suppose that there are $t^*_i \in [0, T]$ such that $\mathcal{A}_i(t^*_i) \in \mathcal{MR}$ with $\mathcal{A}_i(t^*_i) - \mathcal{A}_j(t^*_j) \in \mathcal{L}(Y)$ and $\norm{\mathcal{A}_i(t^*_i) - \mathcal{A}_j(t^*_j)}_{\mathcal{L}(Y)} \leq C_{\mathcal{I}}$ for some $C_\mathcal{I} > 0$ and any two $i, j \in \mathcal{I}$.   
	Moreover, assume that there exists $\eta = \eta(M) \geq 0$ such that for all $x \in D$ and $s, t \in [0, T]$, $i \in \mathcal{I}$
	\begin{equation}\label{eq:relative_continuity}
			\norm{\mathcal{A}_i(t)x - \mathcal{A}_i(s)x}_{Y} \leq \frac{1}{2M} \norm{x}_{D} + \eta \norm{x}_{Y}, 
	\end{equation}
	where $M$ is the constant from Lemma \ref{lemma:bound_M}. 
	Then $\mathcal{A}_i \in \mathcal{MR} = \mathcal{MR}_p(0, T)$ for all $p \in (1, \infty)$ and all $i \in \mathcal{I}$. \par 
	In particular, for each $f \in L^p(0, \tau; Y)$ and each $x \in (Y, D)_{\frac{1}{p'}, p}$ there exists a unique function $u \in W^{1, p}(0, \tau; Y) \cap L^p(0, \tau; D)$ satisfying
	\begin{align*}
		\begin{cases}
			\pt \bm{u} +\mathcal{A}_i(t) \bm{u} = \bm{f} \quad  \textrm{a.e. on } (0, \tau), \\ 
			\bm{u}(0)  = x. 
		\end{cases}
	\end{align*}
	Moreover, there exists a constant $C > 0$, independent of $i \in \mathcal{I}$, such that 
	\begin{equation}\label{eq:estimate_solution}
		\norm{u}_{\textnormal{MR}(0, T)} 
		\leq C \norm{x}_{(Y, D)_{\frac{1}{p'}, p}} +  4 M\textnormal{e}^{T\lambda} \norm{f}_{L^p(a, b; Y)}, 
	\end{equation}
	where $C, M, \lambda$ depend on $C_{\mathcal{I}}$. 
\end{theorem}

\begin{proof}
	The proof of this theorem can be found in Appendix \ref{app:proof_max_reg}. 
\end{proof}

\begin{remark}
	Comparing this theorem to \cite[Thm.~2.7]{ARENDT20071}, we would like to point out that while we only require $\mathcal{A}(t) \in \mathcal{MR}$ for a single $t^* \in [0, T]$, our assumption \eqref{eq:relative_continuity} is much stricter than in the original, where a similar estimate is only required to hold in small neighborhoods. However, without any further assumptions, an estimate as in \eqref{eq:estimate_solution} could not be obtained. 
\end{remark}

\subsubsection*{Leray--Schauder principle}

The Leary--Schauder principle, which is sometimes also referred to as Schaefer's theorem, is a well-known consequence of Schauder's fixed point theorem. The following formulation was taken from {\cite[II Lem.~3.1.1]{sohr2001navier}}. 

\begin{definition}
	A, not necessarily linear, operator $B: Y \rightarrow Y$ is called \textit{completely continuous } if
	\begin{enumerate}[label*=(\roman*), nosep]
		\item $B$ is continuous; 
		\item for each bounded sequence $(y_n)_{n \inN} \subset Y$ the sequence $(By_n)_{n \inN}$ contains an in $Y$ strongly converging subsequence.
	\end{enumerate}
\end{definition}

\begin{theorem}[Leray-Schauder principle]\label{leray_schauder}
	Let $Y$ be a Banach space and $B: Y \rightarrow Y$ be a completely continuous operator. Assume there exists some $r >0$ such that for all $x \in Y$, $\lambda \in [0, 1]$ satisfying $x = \lambda Bx$ it holds $\norm{y}_{Y} \leq r$. Then there exists at least one $y^* \in Y$ such that $y^* = By^*$ and $\norm{x^*}_{Y} \leq r$.   
\end{theorem}


\section{Existence of weak solutions to a regularized system} \label{sec:regularized_approx}

The goal of this section is to prove the existence of weak solutions to a regularized Cahn--Hilliard--Biot system with an additional bi-laplacian of $\vphi$ in the equation for the chemical potential $\mu$. To this end, we employ a semi-Galerkin ansatz, solving the subsystem consisting of \eqref{eq:strong_formulation_phi}, \eqref{eq:strong_formulation_mu}, \eqref{eq:strong_formulation_theta}, \eqref{eq:strong_formulation_p} and some fixed $\bm{u}$ for $\vphi, \mu, \theta, p$ in subspaces of $H^1(\Omega)$ and $X(\Omega)$, respectively. Here, we reduce the differential-algebraic system to a system of ordinary differential equations for $\vphi, \theta$ and exploit a Lipschitz-estimate along with a Gronwall-type argument to show uniqueness and continuous dependence on the data. On the other hand, for fixed $\vphi, p$, maximal $L^q$-regularity yields a unique solution $\bm{u}$ to \eqref{eq:strong_formulation_u} and we use a fixed-point argument relying on the Leray--Schauder principle to find an approximate solution to the whole system. \par
Exploiting the generalized gradient flow structure, we derive \textit{a priori} estimates and deduce first compactness results. Since these are not sufficient to pass to the limit, we use a lesser known version of the seminal Aubin--Lions--Simon theorem to deduce strong convergence of $p$ in $L^2(\OT)$, which in turn allows us to establish strong convergence of $\bm{u}$ in $L^2(0, T; \bm{H}^1_{\Gamma_D}(\Omega))$ and of $\theta$ in $L^2(\OT)$.\par
The main result of this section is the following theorem. 

\begin{theorem}[Existence of weak solutions to the regularized problem]\label{th:existence_reg} 
	Assume that the assumptions \linebreak \ref{A:domain}-\ref{A:initial} are fulfilled. Further, let $\vphi_0 \in H^2_{\bm{n}}(\Omega)$, $\bm{u}_0 \in \bm{H}^{1}_{\Gamma_D}(\Omega)$, $\theta_0 \in X(\Omega)$ and $\varrho > 0$. Finally, let $\phi \in C^\infty_c(\R^n)$ be a standard convolution kernel. 
	Then, there exist functions $(\vphi, \mu, \bm{u}, \theta, p)$ such that
	\begin{align*}
		\vphi &\in L^\infty(0, T; H_{\bm{n}}^2(\Omega)) \cap H^1(0, T; H^1(\Omega)'), \\ 
		\mu  &\in L^2(0, T; H^1(\Omega)), \\ 
		\bm{u} & \in H^1(0, T; \bm{X}(\Omega)), \\ 
		\theta &\in L^\infty(0, T; L^2(\Omega)) \cap H^1(0, T; X'(\Omega)), \\ 
		p &\in L^2(0, T; X(\Omega )) \cap L^\infty(0, T; L^2(\Omega))
	\end{align*}
	satisfying the following equations: 
	\begin{subequations}
		\begin{align}\label{eq:weak1_reg}
			\int_0^T  \prescript{}{(H^1)'}{}\langle \pt \vphi, \zeta \rangle_{H^1}\, \dt =   \int_{\OT}- m(\vphi) \nabla \mu \cdot \nabla \zeta + R(\vphi, \E(\bu), \theta) \zeta \; \dtx
		\end{align}
		for all $\zeta \in L^2(0, T; H^1(\Omega) )$;
		\begin{align*}\label{eq:weak2_reg}
			\int_{\OT} \mu\, \zeta \dtx 
			=   \int_{\OT} &\vepsilon \nabla \vphi \cdot \nabla \zeta +
			\varrho^{1/2} \Delta \vphi \Delta \zeta 
			+ \Big[ \frac{ 1}{\vepsilon} \psi'(\vphi)  + \Wp(\vphi, \E(\bm{u}))   \\ 
			&\begin{aligned}[t]
			&+ \frac{M'(\vphi)}{2} (\theta - \alpha(\vphi) \nabla \cdot \bm{u})^2  
			-M(\vphi) (\theta - \alpha(\vphi) \nabla \cdot \bm{u}) \alpha'(\vphi) \nabla \cdot \bm{u} \Big]\,   \zeta \; \dtx
			\end{aligned}
			\numberthis
		\end{align*}
		for all $\zeta \in L^\infty(\OT) \cap L^2(0, T; H_{\bm{n}}^2(\Omega))$;
		\begin{align*}\label{eq:weak3_reg}
			&\int_{\OT} \C_\nu (\vphi ) \E(\partial_t \bm{u}) :  \E(\bm{\eta}) + \WE(\vphi , \E(\bm{u})) : \E(\bm{\eta}) -  \alpha(\vphi) p\, (\nabla \cdot \bm{\eta})* \phi \dtx \\ 
			& \quad = \int_{\OT} \bm{f} \cdot \bm{\eta} \dtx + \int_0^T \int_{\bm{\Gamma_N}} \bm{g} \cdot \bm{\eta} \dH \dt
			- \vr \int_{\OT} \nabla \theta \cdot \nabla (\alpha(\vphi) (\nabla \cdot \bm{\eta}) * \phi ) \dtx 
			 \numberthis 
		\end{align*}
		for all $\bm{\eta} \in L^2(0, T; \bm{X}(\Omega))$;
		\begin{align}\label{eq:weak4_reg}
			\int_0^T  \prescript{}{X'}{}\langle \pt \theta, \xi \rangle_{X}\, \dt =   \int_{\OT} - \kappa(\vphi) \nabla p \cdot \nabla \xi + S_f(\vphi, \E(\bu), \theta) \xi \; \dtx
		\end{align}
		for all $\xi \in L^2(0, T; X(\Omega))$; 
		\begin{align}\label{eq:weak5_reg}
			\int_{\Omega_T} p \xi \dtx = \int_{\OT} \vr \nabla \theta \cdot \nabla \xi + M(\vphi) (\theta - \alpha(\vphi) \nabla \cdot \bm{u})\, \xi \dtx  
		\end{align}
		for all $\xi \in L^2(0, T; X(\Omega))$; 
		and 
		\begin{equation}\label{eq:weak_initial_reg}
			(\vphi, \theta, \bm{u})_{|t= 0} = (\vphi_0, \theta, \bm{u}_0) \quad \textrm{a.e. in}\ \  \Omega. 
		\end{equation}
	\end{subequations}
	Moreover, the weak solution satisfies 
	\begin{align*}
			\norm{\vphi_{\varrho}(t)}_{H^1}^2 
			&+ \norm{ \varrho^{1/4} \Delta \vphi_{\varrho}(t)}_{L^2}^2  
			+ \norm{\psi(\vphi_{\varrho})(t)}_{L^1} 
			+ \norm{\bm{u}_{\varrho}(t)}_{\bm{X}}^2
			+  \norm{\theta_{\varrho}(t)}_{L^2}^2 
			+ \norm{\vr^{1/2 }\nabla \theta_{\vr} (t)}_{L^2}^2\\ 
			&+\norm{\pt \varphi_{\varrho}}_{L^2((H^1)')}^2
			+\norm{\pt \theta_{\varrho}}_{L^2(X')}^2
			+\norm{ \mu_{\varrho}}_{L^2(H^1)}^2 
			+ \norm{ p_{\varrho}}_{L^2(X)}^2
			+ \norm{\pt \bm{u}_{\varrho}}_{L^2(\bm{X})}^2
			\leq C, 
	\end{align*} 
	for almost every $t \in (0, T)$ and some $C = C(T) > 0$ independent of $(\vphi, \mu, \bm{u}, \theta, p)$ and $\varrho \in (0, 1)$. 
\end{theorem}

\subsection{Definition of the discretized problem}
	As discussed before, we want to employ a semi-Galerkin ansatz, discretizing parts of the system in space while solving the other on the whole function space. First, we give a precise formulation of the approximate problems we investigate, for which we consider 
\begin{itemize}
	\item $\{z_i\}_{i \inN}$ a subset of eigenfunctions to the Neumann-Laplace operator with homogeneous boundary conditions. It is well known that these can be chosen as an orthonormal Schauder basis of $L^2(\Omega)$, which is simultaneously orthogonal in $H^1(\Omega)$ such that $z_1 \equiv \tfrac{1}{|\Omega|^{1/2}}$.  Moreover, it was shown in \cite[Sec.~3]{lam06} that these also form a Schauder basis of $H^2_{\bm{n}}(\Omega)$; 
	
	\item $\{y_i\}_{i \inN}$ a subset of eigenfunctions to the corresponding eigenvalue problem with mixed boundary conditions
		\begin{gather*}
			\begin{cases}
				- \Delta y= \lambda y \quad &\textrm{in } \Omega, \\ 
				y  = 0 \quad &\textrm{on } \Gamma_D, \\ 
				\nabla y \cdot \bm{n} = 0 & \textrm{on } \Gamma_N. 
			\end{cases}
		\end{gather*}
	 Spectral theory for self-adjoint, compact operators implies that these can be chosen to form an orthonormal basis of $L^2(\Omega)$, which is simultaneously a dense subset of $X(\Omega)$. 
\end{itemize}
Using the notation 
\begin{equation*}
	Z_k \coloneqq  \textnormal{span} \{ z_i : i \leq k \}, \quad \quad Y_k  \coloneqq \textnormal{span} \{ y_i : i \leq k \}
\end{equation*}
and defining $\Pi_k^z$, $\Pi_k^y$ as the orthogonal $L^2$-projections onto the spaces $Z_k$ and  $Y_k$, respectively, the goal is to find functions of the form 
\begin{subequations}
	\begin{align}
		&\vphi_k (t, \bm{x}) = \sum_{i = 1}^k a_i^k(t) z_i(\bm{x}), 
		\quad \mu_k (t, \bm{x}) = \sum_{i = 1}^k b_i^k(t) z_i(\bm{x}), \label{eq:galerkin_ansatz1}\\ 
		&\theta_k (t, \bm{x}) = \sum_{i = 1}^k c_i^k(t) y_i(\bm{x}), 
		\quad p_k (t, \bm{x}) = \sum_{i = 1}^k e_i^k(t) z_i(\bm{x})\label{eq:galerkin_ansatz2}, 
	\end{align}
\end{subequations}
together with a function $\bm{u}_k \in H^1(0, T; \bm{X}(\Omega))$ with $\bm{u}_k(0) = \bm{u}_0$  such that for all $t \in [0, T]$ and all $j \leq k$ the following equations hold: 
\begin{subequations}
	\begin{align}
		0 &= \langle \pt  \vphi_k, z_j \rangle + \langle m(\vphi) \nabla \mu_k, \nabla z_j \rangle - \langle R(\vphi_k,\E(\bu_k), \theta_k), z_j \rangle,  \label{eq:galerkin1}\\ 		
		0 &= \label{eq:galerkin2}
		\begin{aligned}[t]
			\langle \mu_k - \frac{1}{\vepsilon} \psi'(\vphi_k), z_j \rangle 
			&- \langle \vepsilon \nabla \vphi_k, \nabla z_j \rangle 
			- \varrho^{1/2} \langle \Delta \vphi_k, \Delta z_j \rangle  
			- \langle \Wp(\vphi_k , \E(\bm{u}_k)), z_j \rangle  \\
			& + \langle  M(\vphi_k) (\theta_k - \alpha(\vphi_k) \nabla \cdot \bm{u}_k) \alpha'(\vphi_k) \nabla \cdot \bm{u}_k, z_j \rangle \\
			& - \langle \tfrac{M'(\vphi_k)}{2} (\theta_k - \alpha(\vphi_k) \nabla \cdot \bm{u}_k)^2 , z_j \rangle ,
		\end{aligned}\\ 
		0 &= \langle \pt \theta_k , y_j\rangle + \langle \kappa(\vphi_k)\nabla p_k, \nabla y_j \rangle - \langle S_f(\vphi_k,\E(\bu_k), \theta_k), y_j \rangle,
		\label{eq:galerkin3} \\
		0 &= \langle p_k, y_j \rangle - \vr \langle \nabla \theta_k , \nabla y_j \rangle  - \langle M(\vphi_k) (\theta_k - \alpha(\vphi_k) \nabla \cdot \bm{u}_k), y_j \rangle,  \label{eq:galerkin4}\\ 
		\vphi_k(0) &= \vphi_{0, k} \coloneqq  \Pi_k^z (\vphi_0), \\
		\theta_k(0) &= \theta_{0, k} \coloneqq  \Pi_k^y (\theta_0), \label{eq:galerkin6}
	\end{align}
	where $\langle \cdot, \cdot \rangle$ denotes the $L^2$-inner product. 
	Moreover, for all $\bm{\eta} \in \bm{X}(\Omega) $ the following equation has to be satisfied for almost all $t \in (0, T)$
	\begin{align*}
		\int_{\Omega} \C_\nu (\vphi_k ) \E(\pt \bm{u}_k)  :  \E(\bm{\eta}) &+ \WE(\vphi_k , \E(\bm{u}_k) ) :  \E(\bm{\eta})  - \alpha(\vphi_k) p_k (\nabla \cdot \bm{\eta})* \phi \dx \\ 
		&= \int_{\Omega} \bm{f} \cdot \bm{\eta} \dx + \int_{\Gamma_D} \bm{g}\cdot \bm{\eta} \dH 
		- \vr \int_{\Omega} \nabla \theta_k \cdot \nabla (\alpha(\vphi_k) (\nabla \cdot \bm{\eta}) * \phi ) \dx .  \numberthis \label{eq:galerkin7}
	\end{align*} 
\end{subequations}

\begin{remark}\label{remark:projection_p}
	Owing to the properties of the projection, it follows from \eqref{eq:galerkin4} that for all $\xi \in H^1(\Omega)$ 
	\begin{align*}
		\int_\Omega p_k \xi \dx = \int_\Omega p_k (\Pi_k^y \xi) \dx 
		&= \int_\Omega \vr \nabla \theta_k \cdot \nabla \Pi_k^y \xi + M(\vphi_k) (\theta_k - \alpha(\vphi_k) \nabla \cdot \bm{u}_k) (\Pi_k^y \xi) \dx\\ 
		&= \int_\Omega \vr \nabla \theta _k\cdot \nabla  \xi + \Pi^y_k (M(\vphi_k) (\theta_k - \alpha(\vphi_k) \nabla \cdot \bm{u}_k) )\, \xi \dx \numberthis \label{eq:p_identity_new}
	\end{align*}
\end{remark}

\subsection{Energy estimates} \label{sec:energy_estimates}

Exploiting the gradient flow-like structure of the Cahn--Hilliard--Biot system, we will now derive the energy estimates on which all \textit{a priori} estimates in this work rely on. Note that we will frequently refer to this section instead of repeating these computations. 
\par 
Let us first point out that 
\begin{equation*}
	\frac{d}{dt} \int_\Omega W(\vphi_k , \E(\bm{u}_k)) \dx = \langle  \Wp(\vphi_k , \E(\bm{u}_k)), \pt \vphi_k  \rangle + \langle \WE (\vphi_k , \E(\bm{u}_k)), \E(\pt \bm{u}_k) \rangle, 
\end{equation*}
and 
\begin{equation*}
	\frac{d}{dt} \int \frac{\vepsilon}{2} \abs{\nabla \vphi_k}^2 + \frac{1}{\vepsilon} \psi(\vphi_k) \dx = \langle \vepsilon\nabla \vphi_k, \nabla \pt \vphi_k \rangle + \frac{1}{\vepsilon} \langle \psi'(\vphi_k) , \pt \vphi_k \rangle.  
\end{equation*}
We further find 
\begin{align*}
	\frac{d}{dt}& \int_\Omega \frac{M(\vphi_k)}{2} (\theta_k  - \alpha(\vphi_k) \nabla \cdot \bm{u}_k )^2
	\dx\\ 
	& = \langle  \frac{M'(\vphi_k)}{2} (\theta_k- \alpha(\vphi_k) \nabla \bm{u}_k)^2 ,  \pt \vphi \rangle 
	- \langle M(\vphi_k) (\theta_k - \alpha(\vphi_k) \nabla \cdot \bm{u}_k) \alpha'(\vphi_k)  \nabla \cdot \bm{u}_k , \pt \vphi_k \rangle   \\ 
	& \quad \quad - \langle M(\vphi_k) \alpha(\vphi_k)(\theta_k - \alpha(\vphi_k) \nabla \cdot \bm{u}_k),  \nabla \cdot \pt \bm{u}_k \rangle 
	+  \langle M(\vphi_k)(\theta_k - \alpha(\vphi_k) \nabla \cdot \bm{u}_k), \pt \theta_k \rangle  . 
\end{align*}

To obtain suitable \textit{a priori} estimates for solutions to the system of differential-algebraic equations, we test  \eqref{eq:galerkin1}-\eqref{eq:galerkin4} with suitable functions under the assumption that $\bm{u}_k \in H^1(0, T; \bm{X}(\Omega))$ is given. In particular, we multiply \eqref{eq:galerkin1} with $b_k^j$ and $\eqref{eq:galerkin2}$ with $(a_k^j)'$. Moreover, we multiply \eqref{eq:galerkin3} with $d_k^j$ and $\eqref{eq:galerkin4}$ with $(c_k^j)'$. Finally, summing from $j = 1$ to $k$ yields 
\begin{align*}
	0 &= \langle \pt \vphi_k, \mu_k \rangle + \langle m(\vphi) \nabla \mu_k, \nabla \mu_k \rangle - \langle R(\vphi_k,\E(\bu_k), \theta_k), \mu_k \rangle, \\ 		
	0 &= 
	\begin{aligned}[t]
		&\langle \mu_k - \frac{1}{\vepsilon} \psi'(\vphi_k), \pt \vphi_k \rangle 
		- \langle \vepsilon \nabla \vphi_k, \nabla \pt \vphi_k \rangle 
		- \varrho^{1/2} \langle \Delta \vphi_k, \Delta \pt \vphi_k\rangle  
		- \langle \Wp(\vphi_k , \E(\bm{u}_k)),\pt \vphi_k \rangle  \\
		&+\langle M(\vphi_k) (\theta_k - \alpha(\vphi_k) \nabla \cdot \bm{u}_k) \alpha'(\vphi_k)  \nabla \cdot \bm{u}_k , \pt \vphi_k\rangle 
		 - \langle\tfrac{M'(\vphi_k)}{2} (\theta_k- \alpha(\vphi_k) \nabla \bm{u}_k)^2 ,\pt \vphi_k \rangle ,
	\end{aligned}\\ 
	0 &= \langle \pt \theta_k, p_k \rangle + \langle \kappa(\vphi_k)\nabla p_k, \nabla p_k \rangle - \langle S_f(\vphi_k,\E(\bu_k), \theta_k), p_k \rangle,
 \\
	0 &= \langle p_k, \pt \theta_k \rangle - \vr \langle \nabla \theta_k , \nabla \pt \theta \rangle - \langle M(\vphi_k)(\theta_k - \alpha(\vphi_k) \nabla \cdot \bm{u}_k), \pt \theta_k \rangle. 
\end{align*}
Summing over these equations while adding and subtracting the term  
\begin{align*}
		\langle \WE(\vphi_k, \E(\bm{u}_k)), \E(\pt \bm{u}_k) \rangle - \langle M(\vphi_k)(\theta_k - \alpha(\vphi_k) \nabla \cdot \bm{u}_k) \alpha(\vphi_k) , \nabla \cdot \pt \bm{u}_k \rangle 
\end{align*}
leads to 
\begin{gather} \label{eq:galerkin_sum} 
	\begin{align*}
		&\norm{m(\vphi_k)^{1/2} \nabla \mu_k}_{L^2}^2
		+\norm{\kappa(\vphi_k)^{1/2} \nabla p_k}_{L^2}^2\\ 
		&\quad +\frac{d}{dt} \Big[
		\begin{aligned}[t]
			\int_\Omega \frac{\vepsilon}{2} \abs{\nabla \vphi_k}^2&+ \frac{1}{\vepsilon} \psi(\vphi_k) +  \frac{\varrho^{1/2}}{2} \abs{\Delta \vphi_k}^2 + \frac{\vr}{2} \abs{\nabla \theta_k}^2 \dx 
			+ \int_\Omega W(\vphi_k, \E(\bm{u}_k)) \dx \\ 
			&+ \int_\Omega \frac{M(\vphi_k)}{2}   (\theta_k - \alpha(\vphi_k) \nabla \cdot \bm{u}_k)^2 \dx
			\Big]
		\end{aligned}\\ 
		& = \begin{aligned}[t]
			\langle R(\vphi_k, \E(\bu_k), \theta_k), \mu_k \rangle &+ \langle S_f(\vphi_k,\E(\bu_k), \theta_k), p_k \rangle 
			+	\langle \WE(\vphi_k, \E(\bm{u}_k)), \E(\pt \bm{u}_k) \rangle \\ 
			& - \langle M(\vphi_k)(\theta_k - \alpha(\vphi_k) \nabla \cdot \bm{u}_k) \alpha(\vphi_k) , \nabla \cdot \pt \bm{u}_k \rangle . 
		\end{aligned}
	\end{align*}\numberthis 
\end{gather}

\subsubsection*{Estimates for the right-hand side }

Observe that the function which is constant with value $\tfrac{1}{|\Omega|^{1/2}}$ is an eigenfunction associated with the eigenvalue $0$ of the Neumann-Laplace operator. Hence, testing \eqref{eq:galerkin2} with this function and invoking Young's inequality yields, together with Poincaré's inequality and \ref{A:psi_1}, 
\begin{gather*}
	\begin{align*}
		| \Omega|^{1/2}  \Big| \dashint \mu_k  \Big|
		&= \frac{1}{|\Omega|^{1/2}}
		\begin{aligned}[t]
			  \Big|  \int_\Omega &\psi'(\vphi_k) 
			+ \Wp(\vphi_k , \E(\bm{u}_k))\\ 
			&+\tfrac{M'(\vphi_k)}{2} (\theta_k - \alpha(\vphi_k) \nabla \cdot \bm{u}_k)^2  
			-M(\vphi_k)  (\theta_k - \alpha(\vphi_k) \nabla \cdot \bm{u}_k) \alpha'(\vphi_k) \nabla \cdot \bm{u}\, \dx \Big|  
		\end{aligned}\\ 
		&\leq C(\norm{\psi'(\vphi_k)}_{L^1} + \norm{\vphi_k }_{L^2}^2 + \norm{\theta_k}_{L^2}^2 + \norm{\E(\bm{u}_k)}_{L^2}^2  +\norm{\nabla \cdot \bm{u}_k}_{L^2}^2 + 1) \\ 
		&\leq C(\norm{\psi(\vphi_k)}_{L^1} + \norm{\vphi_k}_{L^2}^2 + \norm{\theta_k}_{L^2}^2 + \norm{\bm{u}_k}_{\bm{X}}^2    + 1). 
	\end{align*}\numberthis \label{eq:mu_mean_value}
\end{gather*}
\par 
We rewrite the first term on the right-hand side of \eqref{eq:galerkin_sum} as 
\begin{equation*}
	\langle R(\vphi_k, \E(\bu_k), \theta_k), \mu_k \rangle = \langle R(\vphi_k, \E(\bu_k), \theta_k), \mu_k - \dashint \mu_k \rangle + \langle R(\vphi_k, \E(\bu_k), \theta_k), \dashint \mu_k \rangle
\end{equation*}
and deduce from the fact that $R$ is bounded
\begin{equation*}
	| \bigl< R(\vphi_k, \E(\bu_k), \theta_k), \dashint \mu_k \bigr> | \leq 
	C(\norm{\psi(\vphi_k)}_{L^1} + \norm{\vphi_k}_{L^2}^2 + \norm{\theta}_{L^2}^2 + \norm{\bm{u}_k}_{\bm{X}}^2   + 1) . 
\end{equation*}
The inequalities of Young and Poincaré further imply 
\begin{equation*}
	| \bigl<  R(\vphi_k,\E(\bu_k), \theta_k), \mu_k - \dashint \mu_k \bigr> |  \leq C + {\rho_\mu} \norm{\mu_k - \dashint \mu_k }_{L^2}^2
	\leq C + \rho_\mu C_p  \norm{\nabla \mu_k }_{L^2}^2, 
\end{equation*}
where $\rho_\mu > 0$ is a small parameter yet to be determined. 
\par 
Recalling the growth conditions for $S_f$, similar arguments as above show for all $\rho_p > 0$
\begin{align*}
	|\langle S_f(\vphi_k, \E(\bu_k), \theta_k), p_k \rangle | &\leq  C\norm{S_f(\vphi_k, \E(\bm{u}_k), \theta_k)}_{L^2}^2 + \rho_p \norm{p_k}_{L^2}^2\\ 
	& \leq 
	C(\norm{\vphi_k}_{L^2}^2 + \norm{\theta_k}_{L^2}^2 + \norm{\bm{u}_k}_{\bm{X}}^2) + \rho_p C_p\norm{ \nabla p_k}_{L^2}^2, 
\end{align*}
where the constant $C = C(\rho_p)$ may depend on $\rho_p$. 
\par 
It remains to estimate the terms we added artificially. Here, the growth conditions on the elastic energy along with Young's inequality imply that, for some small parameter $\rho_{\pt \bm{u}} > 0$, 
\begin{equation*}
	| \langle \WE(\vphi_k, \E(\bm{u}_k)), \E(\pt \bm{u}_k) \rangle |  \leq C(\norm{\vphi_k}_{L^2}^2 +\norm{\bm{u}_k}_{\bm{X}}^2 + 1) 
	+ \frac{ \rho_{\pt \bm{u}}}{2}  \norm{\pt \bm{u}_k}_{\bm{X}}^2. 
\end{equation*}
Invoking Young's inequality once again leads to 
\begin{align*}
 \langle M(\vphi_k)(\theta_k - \alpha(\vphi_k) \nabla \cdot \bm{u}_k) \alpha(\vphi_k) , \nabla \cdot \pt \bm{u}_k \rangle | 
	\leq C(   \norm{\bm{u}_k}_{\bm{X}}^2 + \norm{\theta_k}_{L^2}^2) + \frac{ \rho_{\pt \bm{u}}}{2}  \norm{\pt \bm{u}_k}_{\bm{X}}^2 . 
\end{align*}
\par 
Thus, the right-hand side can be estimated from above as 
\begin{align*}
	&| \langle R(\vphi_k,\E(\bu_k), \theta_k), \mu_k \rangle|  
	\begin{aligned}[t]
		&+ | \langle S_f(\vphi_k,\E(\bu_k), \theta_k), p_k \rangle | 
		+ |\langle \WE(\vphi_k, \E(\bm{u}_k)), \E(\pt \bm{u}) \rangle |\\ 
		& + |\langle M(\vphi_k)(\theta_k - \alpha(\vphi_k) \nabla \cdot \bm{u}_k) \alpha(\vphi_k) , \nabla \cdot \pt \bm{u}_k \rangle| 
	\end{aligned}\\ 
	&\quad \leq C(\norm{\psi(\vphi_k)}_{L^1} + \norm{\vphi_k}_{L^2}^2 + \norm{\theta_k}_{L^2}^2 + \norm{\bm{u}_k}_{\bm{X}}^2   + 1) +  \rho_p C_p \norm{ \nabla p_k}_{L^2}^2 + \rho_\mu C_p  \norm{\nabla \mu_k }_{L^2}^2 +  \rho_{\pt \bm{u}} \norm{\pt \bm{u}_k}_{\bm{X}}^2. 
\end{align*}

\begin{remark}
	We wish to point out that the term $C \norm{\pt \bm{u}_k}_{\bm{X}}^2$ only appears because we had to artificially add certain terms, which is not necessary when deriving \textit{a priori} estimates for the full system. 
	In the application of Gronwall's lemma later on, this term can be bounded by a constant. 
\end{remark}

\subsubsection*{Estimates for the left-hand side}

To establish the crucial estimates from below, we remind ourselves of the decomposition of the total energy into three components, as discussed in Section~\ref{sec:derivation}, and treat each contribution separately.  

\subparagraph*{Interface energy:}

\begin{subequations}
	Recalling the assumptions \ref{A:psi_1} and \ref{A:psi_2}, we calculate with the help of Young's inequality for some $\rho_{\psi_2} > 0$ and all $t \in [0, T]$
	\begin{align*}
		\int_{\Omega} \psi(\vphi_k) \dx 
		= 	\int_{\Omega} \psi_1(\vphi_k) + \psi_2(\vphi_k)\dx
		&\geq \int_{\Omega} \gamma_{\psi_1} \abs{\vphi_k}^p - c_{\psi_1}  - C_2 |\vphi_k|^2 - C_{\psi_2} \dx  \\ 
		&\geq (\gamma_{\psi_1} - \rho_{\psi_2}) \norm{\vphi_k}^p_{L^p} -  C . 
	\end{align*}
	Note that this is well-defined, as can be seen from $\vphi_k \in Z_k \hookrightarrow H^2(\Omega) \hookrightarrow L^\infty(\Omega)$. 
	Using this, we compute for all $t \in [0, T]$
	\begin{equation}\label{estimate_chemical_below}
		\int_\Omega\frac{\vepsilon}{2} \abs{\nabla \vphi_k}^2+ \frac{1}{\vepsilon} \psi(\vphi_k) \dx 
		\geq \frac{\vepsilon}{2} \norm{\nabla \vphi_k}_{L^2}^2 + \frac{1}{2 \vepsilon} \norm{\psi(\vphi_k)}_{L^1} +\frac{\gamma_{\psi_1} - \rho_{\psi_2}}{2 \vepsilon} \norm{\vphi_k}_{L^p}^p - C. 
	\end{equation}
	On the other hand, it holds 
	\begin{equation}\label{estimate_chemical_above}
		\int_\Omega\frac{\vepsilon}{2} \abs{\nabla \vphi_{0, k}}^2+ \frac{1}{\vepsilon} \psi(\vphi_{0, k}) \dx 
		\leq \frac{\vepsilon}{2} \norm{\vphi_{0, k}}_{H^1}^2 +  \frac{1}{\vepsilon}  \norm{\psi(\vphi_{0, k})}_{L^1}. 
	\end{equation}
\end{subequations}

\subparagraph{Elastic energy:}

\begin{subequations}
	Taking advantage of the estimate derived in \eqref{ineq:W} along with Korn's inequality and Young's inequality, we see
	\begin{align*}
		\int_\Omega W(\vphi_k , \E(\bm{u}_k)) \dx  
		&\geq \int_\Omega \frac{C_1}{4} \abs{\E(\bm{u}_k)}^2 - C'(\abs{\vphi_k}^2 + 1) \dx \\
		& \geq C_{\bm{u}} \norm{\bm{u}_k}_{\bm{X}}^2 - \rho_\vphi \norm{\vphi_k}_{L^p}^p - C, \numberthis \label{estimate_elastic_below}
	\end{align*}
	where the last step holds due to 
	\begin{align*}
		\norm{\vphi_k}^2_{L^2} &= \int_\Omega  \vphi_k^2 \dx 
		\leq  \int_\Omega  \dfrac{1}{q' ( \rho_\vphi q)^{\frac{q'}{q}} } + \rho_\vphi \vphi_k^{2 q} \dx  = \rho_\vphi \norm{\vphi_k}^p_{L^p} + C_{\rho_\vphi}
	\end{align*}
	with $q > 1$ such that $2q = p$. We remark that $\rho_\vphi > 0$ can be chosen as small as necessary. \par 
	\par 
	Moreover, the growth conditions (\hyperref[A:W_growth_conditions]{A3.2}) imply 
	\begin{equation}\label{eq:estimate_elastic_initial}
		\int_\Omega W(\vphi_{0,k}, \E(\bm{u}_0)) \dx 
		\leq C(\norm{\vphi_{0, k}}_{L^2}^2 + \norm{\bm{u}_0}_{\bm{X}}^2 +1 ). 
	\end{equation}
\end{subequations}

\subparagraph{Fluid energy:}

\begin{subequations}
	Using Young's inequality and the positivity of $M$, we calculate 
	\begin{align*}
		-  \theta_k (2 \left( \alpha(\vphi_k) \nabla \cdot \bm{u}_k) \right) 
		\geq - \rho_\theta \theta_k^2 + \frac{1}{\rho_\theta}  (\alpha(\vphi_k) \nabla \cdot \bm{u}_k)^2
	\end{align*}
	for all $\rho_\theta > 0$, from which we deduce 
	\begin{align*}
		 &\int_\Omega \tfrac{M(\vphi_k)}{2} (\theta_k - \alpha(\vphi_k) \nabla \cdot \bm{u}_k)^2 \dx
		 \geq \int_\Omega   \tfrac{\underline{M}}{2} \big[ \theta_k^2   
		- 2\theta_k  (\alpha(\vphi_k) \nabla \cdot \bm{u}_k) 
		+ (\alpha(\vphi_k) \nabla \cdot \bm{u}_k)^2  \big] \\ 
		&  \quad \geq \tfrac{\underline{M}}{2} \left((1-\rho_\theta) \norm{\theta_k}_{L^2}^2 -  ( \frac{1}{\rho_\theta} -1 )  \norm{\alpha(\vphi_k) \nabla \cdot \bm{u}_k}_{L^2}^2\right)
		 \geq \tfrac{\underline{M}} {2}(1-\rho_\theta) \norm{\theta_k}_{L^2}^2 +
		 \tfrac{\underline{M}}{2}(1- \frac{1}{\rho_\theta}) \overline{\alpha}^2  \norm{\bm{u}_k}_{\bm{X}}^2\numberthis \label{estimate_fluid_below}
	\end{align*}
	if we choose $\rho_\theta < 1$. 
	\par 
	Moreover, similar arguments show 
	\begin{align}\label{estimate_fluid_above}
		\int_\Omega \tfrac{M(\vphi_k)}{2}   (\theta_{k, 0} - \alpha(\vphi_{k, 0}) \nabla \cdot \bm{u}_0)^2 \dx 
		\leq C(\norm{\theta_{0, k}}_{L^2}^2 + \norm{\bm{u}_0}_{\bm{X}}^2). 
	\end{align}
\end{subequations}

\subsection*{Full energy estimate}

Integrating \eqref{eq:galerkin_sum} with respect to time and invoking the fundamental lemma of calculus along with the estimates \eqref{estimate_chemical_above}-\eqref{estimate_fluid_below} finally yields 
\begin{gather}\label{energy_inequality1}
	\begin{align*}
		&(\underline{m} - \rho_\mu C_p) \int_0^t \norm{\nabla \mu_k}^2 \dt  
		+(\underline{\kappa} - \rho_p C_p) \int_0^t \norm{ \nabla p_k}^2 \dt \\ 
		+&\frac{\vepsilon}{2} \norm{\nabla \vphi_k(t)}_{L^2}^2 + \frac{1}{2 \vepsilon} \norm{\psi(\vphi_k(t))}_{L^1}
		+ \frac{\varrho^{1/2}}{2} \norm{\Delta \vphi_k}_{L^2}^2
		 +(\frac{\gamma_{\psi_1} - \rho_{\psi_2}}{2 \vepsilon} -  \rho_\vphi ) \norm{\vphi_k(t)}_{L^p}^p  \\ 
		+ &(C_{\bm{u}} - \frac{\underline{M}} {2} \overline{\alpha}^2  (\frac{1}{\rho_\theta}- 1) )\norm{\bm{u}_k(t)}_{\bm{X}}^2
		+ \frac{\underline{M}} {2} (1-\rho_\theta) \norm{\theta(t)}_{L^2}^2 + \frac{\vr}{2} \norm{\nabla \theta_k(t)}_{L^2}^2 \\ 
		\leq& \begin{aligned}[t]
			&C \int_0^t \norm{\psi(\vphi_k)}_{L^1} + \norm{\vphi_k}_{L^2}^2 + \norm{\theta_k}_{L^2}^2+ \norm{\bm{u}_k}_{H^1}^2 +  \rho_{\pt \bm{u}} \norm{\pt \bm{u}_k}_{\bm{X}}^2  + 1\dt \\  
			&+ C( \norm{\vphi_{0, k}}_{H^1}^2 + \varrho^{1/2} \norm{\Delta \vphi_{0, k}}_{L^2}^2 + \vr \norm{\nabla \theta_{0, k}}_{L^2}
			^2 + \frac{1}{\vepsilon}  \norm{\psi(\vphi_{0, k})}_{L^1} +  \norm{\bm{u}_0}_{H^1}^2 + \norm{\theta_{0, k}}_{L^2}^2 + 1). 
		\end{aligned} 
	\end{align*}
	\numberthis 
\end{gather}

Choosing $\rho_\vphi, \rho_\mu, \rho_p, \rho_{\psi_2}, \rho_{\pt \bm{u}} > 0$ small enough and $0 < \rho_\theta <1$ sufficiently close to $1$ will allow us to find \textit{a priori} estimates by means of Gronwall's lemma. 

\subsection{Existence of approximate solutions}\label{sec:existence_semi_galerkin}

This section contains three major arguments. Firstly, we consider the system of differential-algebraic equations \eqref{eq:galerkin1}-\eqref{eq:galerkin6} and reduce it to a system of ordinary differential equations for the coefficient functions $\bm{a}$ and $\bm{d}$ with continuous right-hand side. Applying Peano's theorem, we obtain local in time solutions $\vphi, \mu, \theta, p$ on some interval, which can be extended to global solutions due to \textit{a priori} estimates relying on the energy estimates derived in the last section. In order to apply a fixed point theorem, the solution has to be unique and continuously depend on the data. We prove these properties by showing a Lipschitz-type estimate for the difference of the coefficient functions $\bm{a}_i, \bm{d_i}$, $i = 1, 2$, of two solutions and an application of Gronwall's lemma. \par 
Secondly, we define an abstract Cauchy-problem with unique solution which continuously depends on the data and also solves \eqref{eq:galerkin7}. Here, we apply maximal $L^p$-regularity theory using Theorem~\ref{th:max_reg_non_autonomous}, where regularity in space is a consequence of Theorem~\ref{thm:elliptic_sobolev}. \par
Building on these results, we can define an operator $\mathfrak{T}$ and apply the Leray--Schauder principle to find a fixed-point of $\mathfrak{T}$, which, by definition, gives rise to an approximate solution of the regularized system \eqref{eq:galerkin1}-\eqref{eq:galerkin7}. 

\begin{lemma}\label{lemma:existence_galerkin}
	For any given $\bm{u} \in H^1(0, T; \bm{X}(\Omega))$ there exist unique functions 
	\begin{equation*}
		(\vphi, \mu, \theta, p) \in C^1([0, T]; Z_k) \times C^0([0, T]; Z_k) \times C^1([0, T]; Y_k) \times C^0([0, T]; Y_k)
	\end{equation*}
	satisfying the system \eqref{eq:galerkin1}-\eqref{eq:galerkin7} and continuously depend on $\bm{u}$. Moreover, the weak derivative $\pt p_k$ exists and is in the space $L^2(0, T; L^2(\Omega))$. 
\end{lemma}

\begin{proof}
	\textbf{Existence:} We define the functions $\bm{a} \coloneqq (a_k^1, \ldots, a_k^k)^\intercal$, $\bm{b} \coloneqq (b_k^1, \ldots, b_k^k)^\intercal$, $\bm{c} \coloneqq (c_k^1, \ldots, c_k^k)^\intercal$  and $\bm{d} \coloneqq (d_k^1, \ldots, d_k^k)^\intercal$. Exploiting the orthonormality of the chosen basis, the system  \eqref{eq:galerkin1}-\eqref{eq:galerkin4} reduces to 
	\begin{subequations}
		\begin{align}
			\frac{d}{dt} a_k^j &= - \langle m(\vphi) \nabla \mu_k, \nabla z_j \rangle + \langle R(\vphi_k,\E(\bu), \theta_k), z_j \rangle, \label{eq:differential_algebraic1} \\ 		
			b_k^j &= 
			\begin{aligned}[t]
				&\langle \frac{1}{\vepsilon} \psi'(\vphi_k), z_j \rangle 
				 +\langle \vepsilon \nabla \vphi_k, \nabla z_j \rangle 
				+ \varrho^{1/2} \langle \Delta \vphi_k, \Delta z_j \rangle  
				+ \langle \Wp(\vphi_k , \E(\bm{u})), z_j \rangle  \\
				& - \langle M(\vphi_k)(\theta_k - \alpha(\vphi_k) \nabla \cdot \bm{u}) \alpha'(\vphi_k) \nabla \cdot \bm{u}, z_j \rangle 
				 + \langle \tfrac{M'(\vphi_k)}{2} (\theta_k - \alpha(\vphi_k) \nabla \cdot \bm{u})^2 , z_j \rangle ,
			\end{aligned} \label{eq:differential_algebraic2 }\\ 
			\frac{d}{dt} c_k^j &=  - \langle \kappa(\vphi_k)\nabla p_k, \nabla y_j \rangle + \langle S_f(\vphi_k, \E(\bu), \theta_k), y_j \rangle, \label{eq:differential_algebraic3}\\
			d_k^j &= \vr \langle \nabla \theta, \nabla y_j \rangle  + \langle M(\vphi_k) (\theta_k - \alpha(\vphi_k) \nabla \cdot \bm{u}), y_j \rangle \label{eq:differential_algebraic4}
		\end{align}
	\end{subequations}
	for all $1 \leq j \leq k$ and all fixed $k \inN$. 
	We observe that the right-hand side of the last equation \eqref{eq:differential_algebraic4} only depends on $\bm{a}, \bm{c}$ and due to the continuity of $M, \alpha$ and $\bm{u} \in H^1(0, T; \bm{X}(\Omega)) \hookrightarrow C^0([0, T]; \bm{X}(\Omega))$, the function $\bm{H}_p$ defined as 
	\begin{align*}
		\bm{H}_p : \R \times \R^k \times \R^k &\rightarrow \R^k, \\ (t, \bm{a},
		\bm{c}) &\mapsto \left(  \vr \langle \nabla \sum_{j = 1}^{k} c_k^j y_j, \nabla y_i \rangle +  \langle M(\sum_{j = 1}^k a_k^j z_j) ((\sum_{j = 1}^k c_k^j y_j)- \alpha(\sum_{j = 1}^k a_k^j z_j) \nabla \cdot \bm{u}(t)), y_j \rangle \right)_{j = 1}^k
	\end{align*}
	is continuous. Moreover, we define \new{$\bm{H}_{\mu}(t, \bm{a}, \bm{c})$ similarly} and note that due to the continuity of the functions $\psi', \Wp, M, M', \alpha, \alpha'$ and $\bm{H}_p$, this map is also continuous. Substituting the identities 
	\begin{align*}
		\bm{b}(t) = \bm{H}_\mu(t, \bm{a}(t), \bm{c}(t)), && \bm{d}(t) = \bm{H}_p(t, \bm{a}(t), \bm{c} (t))
	\end{align*}
	into \eqref{eq:differential_algebraic1} and \eqref{eq:differential_algebraic3}, we arrive at 
	\begin{align*}
		\frac{d}{dt} a_k^j (t) &= - \langle m(\sum_{j = 1}^k a_k^j(t) z_j) \nabla (\sum_{j = 1}^k \bm{H}_\mu^j(t, \bm{a}(t), \bm{c}(t)) z_j), \nabla z_j \rangle + \langle R(\sum_{j = 1}^k a_k^j(t) z_j, \E(\bu), \sum_{j = 1}^k c_k^j y_j), z_j \rangle, \\
		\frac{d}{dt} c_k^j (t) &=  - \langle \kappa(\sum_{j = 1}^k a_k^j(t) z_j)\nabla (\sum_{j = 1}^k \bm{H}_p^j(t, \bm{a}(t), \bm{c}(t)) y_j), \nabla y_j \rangle + \langle S_f(\sum_{j = 1}^k a_k^j(t) z_j, \E(\bu), \sum_{j = 1}^k c_k^j y_j), y_j \rangle, 
	\end{align*}
	i.e., the differential-algebraic system \eqref{eq:differential_algebraic1}-\eqref{eq:differential_algebraic4} reduces to a system of ordinary differential equations for $\bm{a}, \bm{c}$ with continuous right-hand side and initial conditions 
	\begin{align*}
		a_k^j(0) &= \langle \vphi_0, z_j \rangle \quad  \textrm{for all } j =1,\ldots,k ,  \\
		c_k^j(0) &= \langle \theta_0, y_j \rangle  \quad  \textrm{for all } j =1,\ldots,k. 
	\end{align*}
	Thus, we can apply Peano's theorem and obtain the existence of a, possibly small, $T^* \in (0, T)$ and local solutions $\bm{a}, \bm{c} \in C^1([0, T^*]; \R^k)$ giving rise to $\bm{b}, \bm{d} \in C^0([0, T^*]; \R^k)$. \par 
	Making use of the energy estimate \eqref{energy_inequality1} established above, we find for all applicable $t \in \R^+$
	\begin{align*}
		&\norm{\vphi_k(t)}_{H^1}^2 +\varrho^{1/2} \norm{\Delta \vphi_k(t)}_{L^2}^2 + \norm{\psi(\vphi_k(t))}_{L^1} +  \norm{\theta_k(t)}_{L^2}^2 + \vr \norm{\nabla \theta_k(t)}_{L^2}^2  \\ 
		&\quad \leq \begin{aligned}[t]
			& C \int_0^t \norm{\vphi_k}_{L^2}^2 + \norm{\psi(\vphi_k)}_{L^1} + \norm{\theta_k}_{L^2}^2  + \norm{\bm{u}}_{\bm{X}}^2 +  \norm{\pt \bm{u}}_{\bm{X}}^2 + 1 \dt \\  
			&+ C( \norm{\vphi_{0, k}}_{H^1}^2 + \varrho^{1/2} \norm{\Delta \vphi_{0, k}}_{L^2}^2 + \vr \norm{\nabla \theta_{0, k}}_{L^2}^2 +  \frac{1}{\vepsilon}  \norm{\psi(\vphi_{0, k})}_{L^1}  + \norm{\theta_{0, k}}_{L^2}^2 + \norm{\bm{u}_0}_{\bm{X}}^2  + 1). 
		\end{aligned} 
	\end{align*}
	By estimating the contributions from the fixed function $\bm{u}$ with its norm in the space $H^1(0, T; \bm{X}(\Omega))$, this expression simplifies to 
	\begin{align*}
		\norm{\vphi_k(t)}_{H^1}^2 +\varrho^{1/2} \norm{\Delta \vphi_k(t)}_{L^2}^2  + \norm{\psi(\vphi_k(t))}_{L^1} &+  \norm{\theta(t)}_{L^2}^2 + \vr \norm{\nabla \theta_k(t)}_{L^2}^2 \\ 
		&\leq  C \int_0^t \norm{\vphi_k}_{L^2}^2 + \norm{\psi(\vphi_k)}_{L^1} + \norm{\theta_k}_{L^2}^2  \dt + C
	\end{align*}
	and Gronwall's lemma yields for all applicable $t \in [0, T]$ the uniform estimate 
	\begin{equation}\label{eq:bound_galerkin}
		\norm{\vphi_k(t)}_{H^1}^2 + \norm{\psi(\vphi_k(t))}_{L^1} +  \norm{\theta_k(t)}_{X}^2 \leq C. 
	\end{equation}
	Since $\vphi_k, \theta_k$ lie in finite dimensional subspaces of $H^1(\Omega)$ and $L^2(\Omega)$, respectively, where all norms are equivalent, this implies the boundedness of $\bm{a}, \bm{c}$ on $[0, T^*]$. Hence, well known theorems for ordinary differential equations let us extend the local solution to the whole interval $[0, T]$ and we obtain functions 
	\begin{equation*}
		(\vphi_k, \mu_k, \theta_k, p_k) \in C^1([0, T]; Z_k) \times C^0([0, T]; Z_k) \times C^1([0, T]; Y_k) \times C^0([0, T]; Y_k)
	\end{equation*}
	satisfying the system  \eqref{eq:galerkin1}-\eqref{eq:galerkin7}. 
	\par 
	To establish $\pt p_k \in L^2(0, T; L^2(\Omega))$, we consider the derivative 
	\begin{align*}
		\frac{d}{dt} &\int_\Omega \vr \nabla \theta_k \cdot \nabla y_j + M(\vphi_k) (\theta_k - \alpha(\vphi_k) \nabla \cdot \bm{u})y_j \dx \\ 
		&= \int_\Omega \vr \nabla \pt \theta_k \cdot \nabla y_j \dx \\ 
		&\quad + \int_\Omega M'(\vphi) \pt \vphi (\theta_k - \alpha(\vphi_k) \nabla \cdot \bm{u})y_j
		+ M(\vphi) (\pt \theta - \alpha'(\vphi_k) \pt \vphi_k\nabla \cdot \bm{u} - \alpha(\vphi) \nabla \cdot \pt \bm{u})y_j\dx. 
	\end{align*}
	We take note of the fact that the basis functions satisfy $z_j \in H^2_{\bm{n}}(\Omega)$ for all $j \inN$ and recall the continuous embedding $H^2_{\bm{n}}(\Omega) \hookrightarrow C(\overline{\Omega}) \hookrightarrow L^\infty (\Omega)$ for $n \leq 3$. Therefore, $\pt \vphi_k \in C^0([0, T]; L^\infty(\Omega))$, implying that the expression above is well defined for all $t \in [0, T]$. In particular, an application of Hölders's inequality yields the estimate 
	\begin{equation*}
		| \frac{d}{dt} d_k^j | \leq C(\norm{\pt \vphi_k}_{L^\infty(L^2)}, k)\, ( \norm{\theta_k}_{L^2} + \norm{\pt \theta_k}_{L^2} +  \vr \norm{\nabla \pt \theta_k}_{L^2} + \norm{\bm{u}}_{\bm{X}} + \norm{ \pt \bm{u}}_{\bm{X}}).  
 	\end{equation*} 
 	Thus, we deduce from $\theta_k \in  C^1([0, T]; Y_k)$ and $\bm{u} \in H^1(0, T; \bm{X} (\Omega))$
 	\begin{align*}
 		\norm{\pt p_k}_{L^2(L^2)}^2 &= \int_0^T \norm{\sum_{i = 1}^k \frac{d}{dt} d_k^i y_i}_{L^2}^2 \dt 
 		= \int_0^T \sum_{i = 1}^{k} |\frac{d}{dt} d_k^i|^2  \norm{y_i}_{L^2}^2  \dt \\ 
 		&\leq    C(\norm{\pt \vphi_k}_{L^\infty(L^2)}, k) \int_0^T \norm{\theta_k}_{L^2}^2 + \norm{\pt \theta_k}_{L^2}^2  +  \vr \norm{\nabla \pt \theta_k}_{L^2}^2 + \norm{\bm{u}}_{\bm{X}}^2 + \norm{\pt \bm{u}}_{\bm{X}}^2 \dt 
 	\end{align*}
 	and therefore $\pt p_k \in L^2(0, T; L^2(\Omega))$. \par
 	\medskip 
 	\textbf{Uniqueness and continuous dependence:}
 	So far we have established the existence of at least one solution to the Galerkin system for any given $\bm{u} \in L^2(0, T; \bm{X}(\Omega))$. We proceed to show uniqueness of this solution and continuous dependence on the data. To this end, let $\bm{u}_1, \bm{u}_2 \in  L^2(0, T; \bm{X}(\Omega))$ and $(\vphi_i, \mu_i, \theta_i, p_i)$ $i = 1, 2$, be two solutions to the corresponding differential-algebraic system for some fixed $k \inN$. The aim is to find estimates for the differences $\sum |a_1^j - a_2^j| + |c_1^j - c_2^j|$, such that an application of Gronwall's lemma yields the desired result. Note that here, in general, the constants $C= C(k) > 0$ cannot be chosen uniformly in $k \inN$. 
 	\par 
 	\textit{Ad $\bm{d}$:} \begin{subequations}
 		Exploiting the orthogonality of our basis, the Lipschitz continuity and boundedness of $M$ along with the embedding $z_j \in H^2_{\bm{n}}(\Omega) \hookrightarrow L^\infty(\Omega)$, we find
 		\begin{align*}
 		&	| \langle M(\vphi_1) \theta_1 - M(\vphi_2) \theta_2, y_j \rangle | 
 		 \leq  | \langle M(\vphi_1) \sum_{i= 1}^k  (c_1^i - c_2^i) y_i , y_j \rangle| 
 		+ | \langle \theta_2 (M(\vphi_1) - M(\vphi_2)), y_j \rangle | \\ 
 		&\quad \leq  \overline{M} |c_1^j -c_2^j| \langle y_j, y_j \rangle 
 		+  L \sum_{i = 1}^k |a_1^i- a_2^i| \langle |z_j| |\theta_2|, |y_j| \rangle
 		\leq C (\sum_{i = 1}^k |a_1^i- a_2^i|  + |c_1^j -c_2^j|) .  \numberthis \label{eq:estimate_d_1}
 		\end{align*}
 		Exploiting the Lipschitz-continuity of $\alpha, M$ imposed in \ref{A:phase_coefficients} along with the fact that the product of two bounded Lipschitz functions is still Lipschitz continuous, we compute 
 		\begin{align*}
 			|M(\vphi_1)& \alpha(\vphi_1) \nabla \cdot \bm{u}_1 - M(\vphi_2) \alpha(\vphi_2) \nabla \cdot \bm{u}_2 |
 			 \leq C( |\nabla \cdot (\bm{u}_1 - \bm{u}_2) | +   \sum_{i = 1}^k |a_1^i- a_2^i| |z_i| |\nabla \cdot \bm{u}_2| ). 
 		\end{align*} 
 		Since $u_i \in H^1(0, T; \bm{X}(\Omega)) \hookrightarrow C([0, T]; \bm{X}(\Omega))$ implies an uniform bound on $\norm{\bm{u}_i(t)}_{\bm{X}}$ for all $t \in [0, T]$, $i = 1, 2$, we arrive at 
 		\begin{align*}
 			&| \langle M(\vphi_1) \alpha(\vphi_1) \nabla \cdot \bm{u}_1 - M(\vphi_2) \alpha(\vphi_2) \nabla \cdot \bm{u}_2, y_j \rangle |  \\ 
 			&\quad \leq C \langle |\nabla (\bm{u}_1 - \bm{u}_2) |, |y_j| \rangle + C \sum_{i = 1}^k |a_1^i - a_2^i| \langle |z_j| |\nabla \cdot  \bm{u}_2|, |y_j| \rangle 
 			\leq C(\sum_{i = 1}^k |a_1^i - a_2^i| + \norm{\bm{u}_1 - \bm{u}_2}_{\bm{X}} ), \numberthis \label{eq:eq_estimate_d_2}
 		\end{align*}
 		where we also used $z_j \in H^2_{\bm{n}}(\Omega) \hookrightarrow L^\infty (\Omega)$. \par 
 		We further obtain 
 		\begin{align*}
 			\vr \abs{\langle \nabla (\theta_1 - \theta_2) , \nabla y_j \rangle } \leq \vr \sum_{i = 1}^k \abs{c_1^i - c_2^i} \abs{\langle \nabla y_i, \nabla y_j \rangle } \leq C  \sum_{i = 1}^k \abs{c_1^i - c_2^i}
 			 \numberthis \label{eq:eq_estimate_d_3}
 		\end{align*}
 		Taking the difference of \eqref{eq:differential_algebraic4} for the two solutions, we conclude with the help of \eqref{eq:estimate_d_1} and \eqref{eq:eq_estimate_d_2} 
 		\begin{align*}
 			|d_1^j - d_2^j| 
 			&\leq  \vr \abs{\langle \nabla( \theta_1- \theta_2) , \nabla y_j \rangle }  
 			\begin{aligned}[t]
 				&+ |\langle M(\vphi_1) \theta_1 - M(\vphi_2) \theta_2, y_j \rangle| \\
 				&+ |	\langle M(\vphi_1) \alpha(\vphi_1) \nabla \cdot \bm{u}_1 - M(\vphi_2) \alpha(\vphi_2) \nabla \cdot \bm{u}_2, y_j \rangle| 
 			\end{aligned}\\ 
 			&\leq C(k)(\sum_{i = 1}^k |a_1^i - a_2^i|+ |c_1^j -c_2^j| + \norm{\bm{u}_1 - \bm{u}_2}_{\bm{X}} ).   \numberthis \label{eq:estimate_d}
 		\end{align*}
 	\end{subequations}
 	\par 
 	\textit{Ad $\bm{b}$:} To derive a similar estimate for the differences $|b_1^j - b_2^j|$, we consider the terms in \eqref{eq:differential_algebraic2 } separately. Taking advantage of orthogonality, it follows that 
 	\begin{subequations}
 	\begin{gather}
 		\vepsilon |\langle \nabla (\vphi_1  - \vphi_2), \nabla z_j \rangle| \leq \vepsilon \sum_{i = 1}^k  |a_1^i - a_2^i| |\langle \nabla z_i, \nabla z_j \rangle| \leq C |a_1^j - a_2^j|, \label{eq:estimate_b_1} \numberthis\\ 
 		\varrho^{1/2}  |\langle \Delta (\vphi_1  - \vphi_2), \Delta z_j \rangle| 
 		\leq \varrho^{1/2} \sum_{i = 1}^k  |a_1^i - a_2^i|  |\Delta z_i|^2_{L^2} \delta_{ij}  \leq C  \sum_{i = 1}^k |a_1^i - a_2^i | . 	\label{eq:estimate_b_1_b} \numberthis 
 	\end{gather}
 	\par 
 	Since the boundedness of $\bm{a}_1, \bm{a}_2$ in $[0, T]$ and the continuity of the basis functions $z_j$, $j = 1, \ldots, k$, in $\overline{\Omega}$ imply $|\vphi_i| \leq C$ in $\overline{\Omega}_T$, we find along with \ref{A:psi}, which stipulates that $\psi \in C^2(\R)$ and therefore implies local Lipschitz continuity, that $|\psi'(\vphi_1) - \psi'(\vphi_2)| \leq C |\vphi_1 - \vphi_2|$. Hence, 
 	\begin{equation}\label{eq:estimate_b_2}
 		\frac{1}{\vepsilon}  |\langle \psi'(\vphi_1) - \psi'(\vphi_2),  z_j \rangle| 
 		\leq \frac{L}{\vepsilon} \sum_{i = 1}^k  |a_1^i - a_2^i| |\langle  |z_j| ,  |z_i| \rangle| \leq C  \sum_{i = 1}^k  |a_1^i - a_2^i|. 
 	\end{equation}
 	\par 
 	Consider 
 	\begin{align*}
 		&M'(\vphi_1) (\theta_1 - \alpha(\vphi_1) \nabla \cdot \bm{u}_1 )^2 
 		-M'(\vphi_2) (\theta_2 - \alpha(\vphi_2) \nabla \cdot \bm{u}_2 )^2\\ 
 		 &  =  \begin{aligned}[t]
 		 	& \left(M'(\vphi_1) \theta_1^2 - M'(\vphi_2)\theta_2^2 \right)
 		 	- \left( M'(\vphi_1) \theta_1 \nabla \cdot \bm{u}_1 + M'(\vphi_2) \theta_2\nabla \cdot \bm{u}_2 \right) \\ 
 		 	&+ M'(\vphi_1) \alpha(\vphi_1)^2 (\nabla \cdot \bm{u}_1)^2 - M'(\vphi_2)\alpha(\vphi_2)^2 (\nabla \cdot \bm{u}_2)^2
 		 	 \eqqcolon I + II + III. 
 		 \end{aligned}
 	\end{align*}
 	For $I$, it holds that 
 	\begin{align*}
 		| M'(\vphi_1) \theta_1^2 - M'(\vphi_2)\theta_2^2  | 
 		& \leq \overline{M} \abs{\theta_1} \sum_{i = 1}^k | c_1^i - c_2^i| |y_i|  + \overline{M} \abs{\theta_2} \sum_{i = 1}^k | c_1^i - c_2^i| |y_i| + C |\theta_2^2|  \sum_{i = 1}^k | a_1^i - a_2^i| |z_i|, 
 	\end{align*}
 	which, along with our \textit{a priori} estimates, leads to
 	\begin{align*}
 		&| \langle M'(\vphi_1) \theta_1^2 - M'(\vphi_2)\theta_2^2  , z_j \rangle| \leq  C ( \sum_{i = 1}^k | a_1^i - a_2^i| +  \sum_{i = 1}^k | c_1^i - c_2^i|).  \numberthis 
 		\label{eq:estimate_b_3_1}
 	\end{align*}
	Since $II$ is of a similar structure, we proceed analogously, computing 
	\begin{align*}
		&| M'(\vphi_1) \theta_1 \nabla \cdot \bm{u}_1 - M'(\vphi_2) \theta_2\nabla \cdot \bm{u}_2 | \\ 
		& \quad \leq \overline{M} |\theta_1| |\nabla \cdot \bm{u}_1 - \nabla \cdot \bm{u}_2| 
		+ \overline{M} | \nabla \cdot \bm{u}_2|  \sum_{i = 1}^k | c_1^i - c_2^i| |y_i|  
		+  C  |\theta_2| |\nabla \cdot \bm{u}_2| \sum_{i = 1}^k | a_1^i - a_2^i| |z_i|
	\end{align*}
	and we arrive at 
	\begin{align*}
		| \langle M'(\vphi_1) \theta_1 \nabla \cdot \bm{u}_1 - M'(\vphi_2) \theta_2\nabla \cdot \bm{u}_2, z_j \rangle |
		  \leq C ( \sum_{i = 1}^k | a_1^i - a_2^i| +  \sum_{i = 1}^k | c_1^i - c_2^i| + \norm{\bm{u}_1 - \bm{u}_2}_{\bm{X}} ).  
		\numberthis \label{eq:estimate_b_3_2}
	\end{align*}
	Finally we observe that $III$ can be treated similarly to $I$, yielding the estimate 
	\begin{align*}
		|\langle M'(\vphi_1) \alpha(\vphi_1)^2 (\nabla \cdot \bm{u}_1)^2 - M'(\vphi_2)\alpha(\vphi_2)^2 (\nabla \cdot \bm{u}_2)^2, z_j \rangle | 
		\leq C( \sum_{i = 1}^k | a_1^i - a_2^i|  + \norm{\bm{u}_1 - \bm{u}_2}_{\bm{X}} ) \numberthis \label{eq:estimate_b_3_3}, 
	\end{align*}
	such that, along with \eqref{eq:estimate_b_3_1} and \eqref{eq:estimate_b_3_2}, we can conclude  
	\begin{align*}
		| \langle M'(\vphi_1) (\theta_1 - \alpha(\vphi_1) \nabla \cdot \bm{u}_1 )^2 
		&-M'(\vphi_2) (\theta_2 - \alpha(\vphi_2) \nabla \cdot \bm{u}_2 )^2   , z_j \rangle | \\ 
		&  \leq C ( \sum_{i = 1}^k | a_1^i - a_2^i| +  \sum_{i = 1}^k | c_1^i - c_2^i| + \norm{\bm{u}_1 - \bm{u}_2}_{\bm{X}} ).  
		\numberthis \label{eq:estimate_b_3}
	\end{align*}
	\par 
	Next, we look at the difference 
	\begin{equation*}
		 M(\vphi_1)(\theta_1 - \alpha(\vphi_1) \nabla \cdot \bm{u}_1) \alpha'(\vphi_1) \nabla \cdot \bm{u}_1
		 -  M(\vphi_2)(\theta_2 - \alpha(\vphi_2) \nabla \cdot \bm{u}_2) \alpha'(\vphi_2) \nabla \cdot \bm{u}_2
	\end{equation*}
	and observe that the structure of this difference is identical to the cases $II$ and $III$ from above. We therefore omit the relevant calculations and simply state the estimate
	\begin{align*}
		&| \langle M(\vphi_1)(\theta_1 - \alpha(\vphi_1) \nabla \cdot \bm{u}_1) \alpha'(\vphi_1) \nabla \cdot \bm{u}_1
		-  M(\vphi_2)(\theta_2 - \alpha(\vphi_2) \nabla \cdot \bm{u}_2) \alpha'(\vphi_2) \nabla \cdot \bm{u}_2, z_j \rangle |\\ 
		& \quad \leq C ( \sum_{i = 1}^k | a_1^i - a_2^i| +  \sum_{i = 1}^k | c_1^i - c_2^i| + \norm{\bm{u}_1 - \bm{u}_2}_{\bm{X}} ).  
		\numberthis \label{eq:estimate_b_4}
	\end{align*}  
 	At last, it remains to find suitable estimates for 
 	\begin{equation*}
 		\Wp(\vphi_1, \E(\bm{u}_1)) - \Wp(\vphi_2 , \E(\bm{u}_2)).  
 	\end{equation*}
 	Since
 	\begin{equation*}
 		2 \Wp(\vphi, \E(\bm{u})) = \begin{aligned}[t]
 			\C'(\vphi) \E(\bm{u}) : \E(\bm{u}) &- 2 \C'(\vphi) \E(\bm{u}) : \Tau(\vphi) + \C'(\vphi) \Tau(\vphi) : \Tau(\vphi)\\ 
 			& - 2 \C(\vphi) \E(\bm{u}): \Tau'(\vphi) + 2 \C(\vphi) \Tau'(\vphi) : \Tau(\vphi), 
 		\end{aligned}
 	\end{equation*}
 	we, once again, proceed termwise. Taking advantage of the boundedness  and Lipschitz continuity of $\C, \C'$ imposed in assumption \ref{A:W}, we see with the help of the Cauchy-Schwarz inequality 
 	\begin{align*}
 		& |\C'(\vphi_1) \E(\bm{u}_1) : \E(\bm{u}_1) - \C'(\vphi_2) \E(\bm{u}_2) : \E(\bm{u}_2) | \\ 
 		& \ \ = \begin{aligned}[t]
 			| \C'(\vphi_1) \E(\bm{u}_1) : (\E(\bm{u}_1) - \E(\bm{u}_2)) + \C'(\vphi_1) \E(\bm{u}_2) : (\E(\bm{u}_1) - \E(\bm{u}_2))
 			 + (\C'(\vphi_1) - \C'(\vphi_2)) \E(\bm{u}_2) : \E(\bm{u}_2) | 
 		\end{aligned}\\ 
 		& \ \  \leq C (\norm{\E(\bm{u}_1)} + \norm{\E(\bm{u}_2)}) \norm{\E(\bm{u}_1) - \E(\bm{u}_2) } + C \norm{\E(\bm{u}_2)}^2 \abs{\vphi_1 - \vphi_2}. 
 	\end{align*}
 	Moreover, assumption \ref{A:W} further requires $\Tau$ to be Lipschitz continuous, which along with \eqref{eq:bound_galerkin} yields the uniform bound $\norm{\Tau(\vphi_2)} \leq C$. Hence, 
 	\begin{align*}
 		&| \C'(\vphi_1) \E(\bm{u}_1) : \Tau(\vphi_1)  - \C'(\vphi_2) \E(\bm{u}_2) : \Tau(\vphi_2) | \\ 
 		& \ \leq C(\norm{ \E(\bm{u}_1) }  +  \norm{\E(\bm{u}_2)} ) \abs{\vphi_1 - \vphi_2} + C \norm{\E(\bm{u}_1) - \E(\bm{u}_2)}. 
 	\end{align*}
 	We note that $\C', \Tau$ are Lipschitz continuous functions and $\vphi_1, \vphi_2$ are bounded in $\OT$, which implies the following Lipschitz property for their product 
 	\begin{equation*}
 		| \C'(\vphi_1) \Tau(\vphi_1) : \Tau(\vphi_1) - \C'(\vphi_2) \Tau(\vphi_2) : \Tau(\vphi_2)|  \leq C \abs{\vphi_1 - \vphi_2}. 
 	\end{equation*}
 	Treating the fourth and fifth term analogously to the second and third, respectively, yields the necessary estimates to deduce 
 	\begin{align*}
 			&| \Wp(\vphi_1 , \E(\bm{u}_1)) - \Wp(\vphi_2, \E(\bm{u}_2))|  \\ 
 			& \quad \leq   C ( \abs{ \E(\bm{u}_1) }  +  \abs{\E(\bm{u}_2)} + \abs{\E(\bm{u}_2)}^2 + 1)\abs{\vphi_1 - \vphi_2}
 			+ C (\abs{ \E(\bm{u}_1) }  +  \abs{\E(\bm{u}_2)}  + 1) \abs{\E(\bm{u}_1) - \E(\bm{u}_2) }
 	\end{align*}
 	and we can conclude 
 	\begin{align*}
 		\abs{	\langle \Wp(\vphi_1, \E(\bm{u}_1)) - \Wp(\vphi_2 , \E(\bm{u}_2)), z_j  \rangle } 
 		 &\leq \begin{aligned}[t]
 			& C \sum_{i = 1}^k   \abs{a_1^i - a_2^i} \langle (\abs{ \E(\bm{u}_1) }  +  \abs{\E(\bm{u}_2)} + \abs{\E(\bm{u}_2)}^2 + 1) \abs{z_i}, \abs{z_j } \rangle \\ 
 			& \quad +C \langle (\abs{ \E(\bm{u}_1) }  +  \abs{\E(\bm{u}_2)}  + 1) \abs{\E(\bm{u}_1) - \E(\bm{u}_2)} ,  \abs{z_j } \rangle 
 		\end{aligned}\\ 
 		&  \leq C  (\sum_{i= 1}^k  \abs{a_1^i - a_2^i} + \norm{\bm{u}_1 - \bm{u}_2}_{\bm{X}}) .  \numberthis \label{eq:estimate_b_5}
 	\end{align*}
 	Finally, taking the difference of equation \eqref{eq:differential_algebraic2 } for the corresponding solutions $\bm{b}_1, \bm{b}_2$ and employing \eqref{eq:estimate_b_1}-\eqref{eq:estimate_b_5}, we obtain 
 	\begin{align*}
 		&\abs{b_1^j - b_2^j}  \leq C  (\sum_{i= 1}^k  \abs{a_1^i - a_2^i} +  \abs{c_1^i - c_2^i} +  \norm{ \bm{u}_1 - \bm{u}_2 }_{\bm{X}} ).  \numberthis \label{eq:estimate_b}
 	\end{align*}
 \end{subequations}
\par 
\medskip 
\textit{Ad $\bm{c}$:} Exploiting the Lipschitz continuity of $\kappa$ along with \eqref{eq:estimate_d} gives
\begin{align*}
	&| \kappa(\vphi_1) \nabla p_1 - \kappa(\vphi_2)\nabla p_2 | 
	 \leq \overline{\kappa} \abs{\nabla(p_1 - p_2)} + \abs{ \nabla p_2} \abs{\kappa(\vphi_1) - \kappa(\vphi_2)} \\ 
	& \quad \begin{aligned}
		\leq C \sum_{i = 1}^k \abs{\nabla y_i} \sum_{l = 1}^k \abs{a_1^l- a_2^l} +  \abs{c_1^l - c_2^l} + \norm{\bm{u}_1 - \bm{u}_2}_{\bm{X}}  
		+ C \abs{\nabla p_2} \sum_{i = 1}^k \abs{a_1^i - a_2^i} \abs{z_i}
	\end{aligned}
 \end{align*}
 and therefore 
 \begin{subequations}
 	\begin{equation}
 		\abs{\langle 	\kappa(\vphi_1) \nabla p_1 - \kappa(\vphi_2)\nabla p_2 , \nabla y_j  \rangle } 
 		\leq C  \Big( \sum_{i= 1}^k  \abs{a_1^i - a_2^i} +  \abs{c_1^i - c_2^i} +  \norm{\bm{u}_1 - \bm{u}_2}_{\bm{X}} \Big) .  \label{eq:estimate_c_1} 
 	\end{equation}
 Moreover, the Lipschitz continuity of $S_f$ leads to 
 \begin{align*}
 	&\abs{\langle S_f(\vphi_1, \E(\bm{u}_1), \theta_1) - S_f(\vphi_2, \E(\bm{u}_2), \theta_2), y_j  \rangle }
 	 \leq C  \Big( \sum_{i= 1}^k  \abs{a_1^i - a_2^i} +  \abs{c_1^i - c_2^i}  +  \norm{\bm{u}_1 - \bm{u}_2}_{\bm{X}} \Big) . 
 \end{align*}
 By taking the difference of \eqref{eq:differential_algebraic3} and integrating with respect to time, we obtain 
 \begin{align*}
 	\abs{c_1^j &-c_2^j} (t) \\ 
 	&\leq \int_0^t  \abs{\langle 	\kappa(\vphi_1) \nabla p_1 - \kappa(\vphi_2)\nabla p_2 , \nabla y_j  \rangle }  + 	\abs{\langle S_f(\vphi_1, \E(\bm{u}_1), \theta_1) - S_f(\vphi_2, \E(\bm{u}_2), \theta_2), y_j  \rangle }  \dt + \abs{c_1^j -c_2^j} (0) \\ 
 	& \leq C \int_0^t  \sum_{i= 1}^k  \abs{a_1^i - a_2^i} +  \abs{c_1^i - c_2^i} +  \norm{\bm{u}_1 - \bm{u}_2}_{\bm{X}}\dt  + \abs{c_1^j -c_2^j} (0). 
 	\numberthis \label{eq:estimate_c}
 \end{align*}
 \end{subequations}
\par 
\medskip 
\textit{Ad $\bm{a}$:} At last it remains to establish a similar estimate for the differences $ \abs{a_1^j - a_2^j}$. Here we use the Lipschitz continuity of $m$ together with \eqref{eq:estimate_b} to compute 
\begin{align*}
	| m(\vphi_1) \nabla \mu_1 &- m(\vphi_1)\nabla \mu_2 | 
	 \leq \overline{m} \abs{\nabla(\mu_1 - \mu_2)} + \abs{\nabla \mu_2} \abs{m(\vphi_1) - m(\vphi_2)} \\ 
	& \leq C \sum_{i = 1}^k \abs{\nabla z_i} \sum_{l = 1}^k \abs{a_1^l - a_2^l} +  \abs{c_1^l - c_2^l} + \norm{\bm{u}_1 - \bm{u}_2}_{\bm{X}}   + C \abs{\nabla \mu_2} \sum_{i = 1}^k \abs{a_1^i - a_2^i} \abs{z_i}, 
\end{align*}
yielding
\begin{equation*}
		\abs{\langle 	m(\vphi_1) \nabla \mu_1 - m(\vphi_1)\nabla \mu_2 , \nabla z_j  \rangle } 
	\leq C  \Big( \sum_{i= 1}^k  \abs{a_1^i - a_2^i} +  \abs{c_1^i - c_2^i} +  \norm{\bm{u}_1 - \bm{u}_2}_{\bm{X}} \Big).
\end{equation*}
Taking advantage of the Lipschitz continuity of $R$, cf. \ref{A:source_terms}, and arguing exactly as above, we arrive at 
 \begin{align*}
	\abs{a_1^j &-a_2^j} (t) \\ 
	&\leq \int_0^t  \abs{\langle 	m(\vphi_1) \nabla \mu_1 - m(\vphi_1)\nabla \mu_2 , \nabla z_j  \rangle } 
	+ 	\abs{\langle R(\vphi_1,\E(\bm{u}_1) , \theta_1) - R(\vphi_2,\E(\bm{u}_2), \theta_2), y_j  \rangle } \dt + \abs{a_1^j -a_2^j} (0) \\ 
	& \leq C \int_0^t  \sum_{i= 1}  \abs{a_1^i - a_2^i} +  \abs{c_1^i - c_2^i} +  \norm{\bm{u}_1 - \bm{u}_2}_{\bm{X}}\dt + \abs{a_1^j -a_2^j} (0). 
	\numberthis \label{eq:estimate_a}
\end{align*} 
\medskip
Summing over \eqref{eq:estimate_c} and \eqref{eq:estimate_a} for all $j \leq k$ leads to 
\begin{align*}
	\sum_{j = 1}^k 	&\abs{a_1^j -a_2^j} (t)  + 	\abs{c_1^j -c_2^j} (t) \\ 
	& \leq C \int_0 ^t 	\sum_{j = 1}^k  \abs{a_1^j -a_2^j} (\tau)  + 	\abs{c_1^j -c_2^j} (\tau) \dt + C \norm{\bm{u}_1 - \bm{u}_2}_{L^1(\bm{X})}
	+ \sum_{j = 1}^k  \abs{a_1^j -a_2^j} (0) +  \abs{c_1^j -c_2^j} (0).
\end{align*}
With the help of Gronwall's lemma we can thus conclude that for any $\bm{u} \in H^1(0, T; \bm{X}(\Omega))$ and matching initial conditions the solution $(\vphi, \mu, \theta, p)$ is unique. Moreover, if $\norm{\bm{u}_1 - \bm{u}_2}_{L^2(0, T; \bm{X})} \rightarrow 0$, then the corresponding solutions also converge in their respective spaces, i.e. the solution continuously depends on $\bm{u}$. 
\end{proof}

The lemma above gives rise to the continuous operator $\mathcal{G}^*$ defined by 
\begin{align*}
	H^1(0, T; \bm{X}(\Omega)) &\rightarrow C^1([0, T]; Z_k) \times C^0([0, T]; Z_k) \times C^1([0, T]; Y_k) \times C^0([0, T]; Y_k) \cap W^{1,2}(0, T; L^2(\Omega)), \\ 
	\bm{u} & \mapsto \mathcal{G}^*(\bm{u}) =  (\vphi, \mu, \theta, p), 
\end{align*}
mapping any given $\bm{u}$ to the unique solution of the corresponding system of differential-algebraic equations for some fixed $k \inN$. Observe that the embeddings 
\begin{gather*}
	Z_k \xhookrightarrow{cpt} H^2_{\bm{n}}(\Omega) \xhookrightarrow{c} H^1(\Omega)
	\quad \textrm{and} \quad 
	Y_k \xhookrightarrow{cpt} X(\Omega) \xhookrightarrow{c} L^2(\Omega)
\end{gather*}
satisfy the assumptions of the Aubin-Lions-Simon theorem, implying compactness for the following operator
\begin{gather}\label{def:G}
	\begin{aligned}
			\mathcal{G} : H^1(0, T; \bm{X}(\Omega))  &\xhookrightarrow{cpt} C^0([0, T]; H^2_{\bm{n}}(\Omega)) \times C^0([0, T]; X(\Omega)) ) \times C^0([0, T]; X(\Omega)), \\ 
		\bm{u}  &\mapsto (\mathcal{G}_1^*\bm{u}, \mathcal{G}^*_3 \bm{u}, \mathcal{G}^*_4\bm{u}) = (\vphi, \theta, p).
	\end{aligned}
\end{gather}

Before we can state the theorem that will allow us to derive solutions to the linear elasticity equation \eqref{eq:galerkin7} for any given $(\vphi, \mu, \theta, p)$, some preparations are necessary. \par 
For a fixed function $\vphi \in L^2(\Omega)$ and some $q \in \R$ close to $2$, we denote by $\mathcal{B}(\vphi)$and $\mathcal{C}(\vphi)$ the operators 
\begin{align*}
	\mathcal{B}(\vphi) &: \bm{W}^{1, q}_{\Gamma_D}(\Omega )  \rightarrow \bm{W}^{-1, q}_{\Gamma_D}(\Omega ) , 
	\quad \bm{v} \mapsto \int_\Omega \C_{\nu}(\vphi ) \E(\bm{v}): \E(\cdot) \dx,\\ 
	\mathcal{C}(\vphi) &: \bm{W}^{1, q}_{\Gamma_D} (\Omega ) \rightarrow \bm{W}^{-1, q}_{\Gamma_D}(\Omega ) , 
	\quad \bm{v} \mapsto \int_\Omega \C(\vphi ) \E(\bm{v}): \E(\cdot) \dx. 
\end{align*}
As observed in Remark \ref{rem:applicability_elliptic}, the assumptions for linear elasticity suffice to find Theorem~\ref{thm:elliptic_sobolev} to be applicable, i.e., for all $\vphi \in L^2(\Omega)$ the operators $\mathcal{B}(\vphi), \mathcal{C}(\vphi)$ are topological isomorphism between $\bm{W}^{1,q}_{\Gamma_D}(\Omega)$ and $\bm{W}^{-1, q}_{\Gamma_D}(\Omega)$ and there exists a common bound for the norm of the inverse. In particular, it holds
\begin{equation}\label{iq:operators_uniform}
	\sup_{\vphi \in L^2(\Omega)} 
	\begin{aligned}[t]
		&(\norm{\mathcal{B}(\vphi)}_{\mathcal{L}(\bm{W}^{1,q}_{\Gamma_D},\bm{W}^{-1, q}_{\Gamma_D} )}
		+  \norm{\mathcal{B}^{-1}(\vphi)}_{\mathcal{L}(\bm{W}^{-1,q}_{\Gamma_D},\bm{W}^{1, q}_{\Gamma_D} )}\\ 
		&+ \norm{\mathcal{C}(\vphi)}_{\mathcal{L}(\bm{W}^{1,q}_{\Gamma_D},\bm{W}^{-1, q}_{\Gamma_D} )}
		+  \norm{\mathcal{C}^{-1}(\vphi)}_{\mathcal{L}(\bm{W}^{-1,q}_{\Gamma_D},\bm{W}^{1, q}_{\Gamma_D} )}) 
		\leq C. 
	\end{aligned}	 
\end{equation} 
Hence, the operator 
\begin{align*}
	\mathcal{A}(\vphi) &: \bm{W}^{1, q}_{\Gamma_D}(\Omega) \rightarrow \bm{W}^{1, q}_{\Gamma_D}(\Omega), 
	\quad \bm{v} \mapsto \mathcal{B}^{-1}(\vphi) \mathcal{C}(\vphi) \bm{v}
\end{align*}
is a well-defined automorphism of $\bm{W}^{1,q}_{\Gamma_D} (\Omega)$. \par 
For the application of fixed-point methods, it is crucial that these operators are continuous, which we investigate in the following lemma.

\begin{lemma}\label{lemma:convergence_B_C}
	Suppose $(\vphi_n)_{n \inN} \subset L^2(\Omega)$ is a convergent sequence with limit $\vphi$ and assume that $(\vphi_n)_{n \inN}$ and $\vphi$ are bounded in $L^6(\Omega)$. Then it holds for all $\hat{\bm{f}} \in \bm{W}^{1,q}_{\Gamma_D}(\Omega)$ and all $\tilde{\bm{f}} \in \bm{W}^{-1,q}_{\Gamma_D}(\Omega)$, respectively,
	\begin{align*}
		\begin{aligned}
				\mathcal{B}(\vphi_n) \hat{\bm{f}} & \rightarrow 	\mathcal{B}(\vphi) \hat{\bm{f}} \\ 
				\mathcal{B}^{-1}(\vphi_n) \tilde{\bm{f}} & \rightarrow 	\mathcal{B}(\vphi)^{-1} \tilde{\bm{f}}				
		\end{aligned}
		\begin{aligned}
			 \quad \textrm{and} \quad \\ 
			  \quad \textrm{and} \quad 
		\end{aligned}
		\begin{aligned}			
			 \mathcal{C}(\vphi_n) \hat{\bm{f}} &\rightarrow 	\mathcal{C}(\vphi) \hat{\bm{f}}  \\ 
			 \mathcal{C}(\vphi_n)^{-1} \tilde{\bm{f}} &\rightarrow 	\mathcal{C}(\vphi)^{-1} \tilde{\bm{f}} 
		\end{aligned}		
		\begin{aligned}
			\quad 
			&\textrm{in} \quad \bm{W}^{-1, q}_{\Gamma_D}(\Omega ),\\ 
			\quad 
			&\textrm{in} \quad \bm{W}^{1, q}_{\Gamma_D}(\Omega ). 
		\end{aligned}
	\end{align*}
\end{lemma}
\begin{proof}
	As the proof for $\mathcal{C}$ is completely analogous, we restrict ourselves in the following to the operator $\mathcal{B}$ and its inverse. 
	By definition, we obtain for all $\bm{\eta} \in \bm{W}^{1,q'}_{\Gamma_D}(\Omega)$
	\begin{align*}
		|(\mathcal{B}(\vphi)\hat{\bm{f}} - \mathcal{B}(\vphi_n) \hat{\bm{f}}) \bm{\eta}| 
		 = \Big| {\int_\Omega [\C_\nu (\vphi_n ) - \C_\nu (\vphi )]\E(\hat{\bm{f}}) : \E(\bm{\eta}) \dx } \Big|
		 \leq \norm{[\C_\nu (\vphi_n ) - \C_\nu (\vphi )]\E(\hat{\bm{f}})}_{L^q} \norm{\E(\bm{\eta}) }_{L^{q'}}
	\end{align*}
	allowing us to deduce 
	\begin{align*}
		\norm{\mathcal{B}(\vphi)\hat{\bm{f}} - \mathcal{B}(\vphi_n) \hat{\bm{f}}}_{ \bm{W}^{-1, q}_{\Gamma_D} }
		&= \sup_{\norm{\bm{\eta}}_ { \bm{W}^{1,q'}_{\Gamma_D} } = 1}    \Big| {\int_\Omega [\C_\nu (\vphi_n ) - \C_\nu (\vphi )]\E(\hat{\bm{f}}) : \E(\bm{\eta}) \dx} \Big|  \\ 
		 &\leq  \norm{[\C_\nu (\vphi_n ) - \C_\nu (\vphi )]\E(\hat{\bm{f}})}_{\bm{L}^q}, 
	\end{align*}
	which tends to zero as $n \rightarrow \infty$. 
	To see this, recall that $\{ \bm{h}_{|\Omega} : \bm{h} \in C^\infty_c(\R^n, \R^n), \supp f \cap \Gamma_D = \emptyset\}$ is dense in $\bm{W}^{1, p}_{\Gamma_D}(\Omega)$. Therefore, for any fixed $\hat{\bm{f}} \in \bm{W}^{1, p}_{\Gamma_D}(\Omega)$ and any $\vepsilon > 0$, there exists some smooth $\hat{\bm{f}}_{\vepsilon}$ such that $\norm{\hat{\bm{f}} - \hat{\bm{f}}_{\vepsilon}}_{\bm{W}^{1, p}_{\Gamma_D}} < \vepsilon$ and we compute with the help of the Lipschitz continuity of $\C_{\nu}$ that 
	\begin{align*}
		\norm{[\C_\nu (\vphi_n ) - \C_\nu (\vphi )]\E(\hat{\bm{f}})}_{\bm{L}^q} 
		&\leq \norm{[\C_\nu (\vphi_n ) - \C_\nu (\vphi )]\E(\hat{\bm{f}} - \hat{\bm{f}}_{\vepsilon)}}_{\bm{L}^q} 
		+ \norm{[\C_\nu (\vphi_n ) - \C_\nu (\vphi )]\E(\hat{\bm{f}}_{\vepsilon})}_{\bm{L}^q}\\ 
		& \leq C_\nu \norm{\hat{\bm{f}} - \hat{\bm{f}}_{\vepsilon}}_{\bm{W}^{1, p}_{\Gamma_D}} 
		+ \norm{\E(\hat{\bm{f}}_{\vepsilon})}_{\bm{L}^\infty} \norm{[\C_\nu (\vphi_n ) - \C_\nu (\vphi )]}_{\bm{L}^q}\\ 
		& <  C_\nu \vepsilon + L \norm{\E(\hat{\bm{f}}_{\vepsilon})}_{\bm{L}^\infty} \norm{\vphi_n - \vphi}^{\vartheta}_{L^2}  \norm{\vphi_n - \vphi}^{1- \vartheta}_{L^6},  \numberthis \label{conv:no_subsq}
	\end{align*}
	where $C_\nu, L$ only depend on $\C_\nu$ and $\vartheta = \tfrac{6-p}{2p}$. Since $\vphi_n \rightarrow \vphi$ in $L^2(\Omega)$ and this sequence is also bounded in $L^6(\Omega)$, we obtain for sufficiently large $n \inN$
	\begin{equation*}
		\norm{[\C_\nu (\vphi_n ) - \C_\nu (\vphi )]\E(\hat{\bm{f}})}_{\bm{L}^q} \leq C_\nu \vepsilon + \varepsilon. 
	\end{equation*}
	As $\varepsilon > 0$ was chosen arbitrarily, this entails the convergence
		\begin{equation*}
				[\C_\nu (\vphi_n ) - \C_\nu (\vphi )]\E(\hat{\bm{f}}) \rightarrow 0 \quad \textrm{in } L^q(\Omega). 
			\end{equation*} 
	\par 
	It remains to show strong convergence for the inverse $\mathcal{B}^{-1}$. Recalling that by \eqref{iq:operators_uniform} the norms of $\mathcal{B}^{-1}$ are uniformly bounded, a standard argument shows for all $\tilde{\bm{f}} \in \bm{W}^{-1, q}_{\Gamma_D}$
	\begin{align*}
		&\norm{\mathcal{B}^{-1} (\vphi_n)\tilde{\bm{f}} - \mathcal{B}^{-1}(\vphi)\tilde{\bm{f}}}_{\bm{W}^{1, q}_{\Gamma_D}}
		= \norm{(\mathcal{B}^{-1} (\vphi_n) [\mathcal{B}(\vphi) - \mathcal{B}(\vphi_n)] \mathcal{B}^{-1}(\vphi)) \tilde{\bm{f}}}_{\bm{W}^{1, q}_{\Gamma_D}} \\ 
		& \quad \leq \norm{\mathcal{B}^{-1} (\vphi_n)}_{\mathcal{L}(\bm{W}^{-1,q}_{\Gamma_D},\bm{W}^{1, q}_{\Gamma_D} )} 
		\norm{ [\mathcal{B}(\vphi) - \mathcal{B}(\vphi_n)] (\mathcal{B}^{-1}(\vphi) \tilde{\bm{f}}) }_{\bm{W}^{- 1, q}_{\Gamma_D}}
		 \leq  C \norm{ [\mathcal{B}(\vphi) - \mathcal{B}(\vphi_n)] (\mathcal{B}^{-1}(\vphi) \tilde{\bm{f}}) }_{\bm{W}^{- 1, q}_{\Gamma_D}}. 
	\end{align*}
	Thus, the result follows from the strong convergence of $\mathcal{B}$. 
\end{proof}

Given these operators, we now turn to study an abstract Cauchy-problem and already note that for the appropriate right-hand side, the solution also solves \eqref{eq:galerkin7}. 

\begin{lemma}\label{lem:cauchy_problem}
	Let $\vphi \in C^0([0, T]; L^2(\Omega)), \hat{\bm{f}} \in L^2(0, T;  \bm{W}^{1, q}_{\Gamma_D}(\Omega))$\ and $u_0 \in \bm{W}^{1, q}_{\Gamma_D}(\Omega)$. Then the  non-autonomous, abstract Cauchy-problem 
	\begin{align*}
		\begin{cases}
			\pt \bm{u}(t) + \mathcal{A}(\vphi(t)) \bm{u}(t) = \hat{\bm{f}}(t) \quad  \textrm{a.e. on } (0, T),\\ 
			\bm{u}(0)  = \bm{u}_0 
		\end{cases}
	\end{align*} 
	has a unique solution $\bm{u} \in H^1(0, T;  \bm{W}^{1, q}_{\Gamma_D}(\Omega))$. 
\end{lemma}

\begin{proof}
	Once we have verified that the assumptions of Theorem~\ref{th:max_reg_non_autonomous} are indeed satisfied, the result will follow immediately. We start by examining $\mathcal{A}$ and note that the discussion above already implies $\mathcal{A}(t)~\in~\mathcal{L}(\bm{W}^{1, q}_{\Gamma_D}(\Omega))$ for all $t \in [0, T]$. Moreover, the continuity of $\vphi$ in time together with the uniform estimates \eqref{iq:operators_uniform} yields that the map $t \mapsto \mathcal{A}(t)$ in $[0, T]$ is strongly measurable. \par 
	Concerning relative continuity, it holds for all $\bm{w} \in \bm{W}^{1, q}_{\Gamma_D}(\Omega)$ and any $t, s \in [0, T]$
	\begin{align*}
		&\norm{ \mathcal{A}(\vphi(t))\bm{w} - \mathcal{A}(\vphi(s))\bm{w} }_{\bm{W}^{1, q}_{\Gamma_D}}
		\leq \norm{\mathcal{B}^{-1}(\vphi(t)) \mathcal{C}(\vphi(t)) - \mathcal{B}^{-1}(\vphi(s)) \mathcal{C}(\vphi(s)) }_{\mathcal{L}(\bm{W}^{1, q}_{\Gamma_D}) } \norm{\bm{w}}_{\bm{W}^{1, q}_{\Gamma_D}}\\ 
		&\quad \leq  \begin{aligned}[t]
			(&\norm{\mathcal{B}^{-1}(\vphi(t))}_{\mathcal{L}(\bm{W}^{-1,q}_{\Gamma_D},\bm{W}^{1, q}_{\Gamma_D} )}
			\norm{\mathcal{C} (\vphi(t))}_{\mathcal{L}(\bm{W}^{1,q}_{\Gamma_D},\bm{W}^{-1, q}_{\Gamma_D} )}\norm{\bm{w}}_{\bm{W}^{1, q}_{\Gamma_D}} \\ 
			&+ \norm{\mathcal{B}^{-1}(\vphi(s))}_{\mathcal{L}(\bm{W}^{-1,q}_{\Gamma_D},\bm{W}^{1, q}_{\Gamma_D} )}
			  \norm{\mathcal{C}(\vphi(s))}_{\mathcal{L}(\bm{W}^{1,q}_{\Gamma_D},\bm{W}^{-1, q}_{\Gamma_D} )})   \norm{\bm{w}}_{\bm{W}^{1, q}_{\Gamma_D}}
			  \leq C \norm{\bm{w}}_{\bm{W}^{1, q}_{\Gamma_D}}. 
		\end{aligned}
	\end{align*}
	Setting $D = \bm{W}^{1, q}_{\Gamma_D} = Y$ in Theorem \ref{th:max_reg_non_autonomous}, we can choose $\eta = C$.\par 
	Moreover, $\mathcal{A}(\vphi(t))$ is a bounded, linear operator which is defined on the whole space $\bm{W}^{1, q}_{\Gamma_D}(\Omega)$ and therefore closed for every $t \in [0, T]$. It is well known that under these conditions $\mathcal{A}(\vphi(t)) \in \mathcal{MR}_q$ on every bounded interval and for all $q \in (1, \infty)$, cf. \cite{Dore93}. Hence, $\mathcal{A}(\vphi (t)) \in \mathcal{MR}$ for all $t \in [0, T]$, cf. \cite{ARENDT20071}. In particular, we can choose $t^*$ as any $t \in [0, T]$. \par 
	Observing that $u_0 \in \bm{W}^{1, q}_{\Gamma_D}(\Omega) =( \bm{W}^{1, q}_{\Gamma_D}(\Omega), \bm{W}^{1, q}_{\Gamma_D}(\Omega))_{\frac{1}{q^*}, } = (D, X)_{\frac{1}{p'}, p}$ concludes the proof.
\end{proof}

In particular, Lemma \ref{lem:cauchy_problem} establishes that for all $\vphi \in C^0([0, T]; L^2(\Omega))$ the bounded linear operator 
\begin{align*}
	\mathfrak{L}(\vphi) : H^1(0, T; \bm{W}^{1, q}_{\Gamma_D}(\Omega)) \rightarrow L^2(0, T; \bm{W}^{1, q}_{\Gamma_D}(\Omega)) \times \bm{W}^{1, q}_{\Gamma_D}(\Omega), \quad 
	\bm{u}  \mapsto (\pt \bm{u} + \mathcal{A} (\vphi) \bm{u}, \bm{u}(0))
\end{align*}
is invertible. Moreover, we deduce from \eqref{eq:estimate_solution} and \eqref{iq:operators_uniform} that the operators 
\begin{align*}
	\mathfrak{L}^{-1}(\vphi, \bm{u}_0) :  L^2(0, T;\bm{W}^{1, q}_{\Gamma_D}(\Omega))  \rightarrow H^1(0, T;\bm{W}^{1, q}_{\Gamma_D}(\Omega)), \quad 
	\hat{\bm{f}} & \mapsto \bm{u}
\end{align*}
are uniformly bounded in $\vphi$, i.e., 
\begin{equation} \label{iq:uniform_bound_L}
	\norm{\mathfrak{L}^{-1}(\vphi, \bm{u}_0)}_{\mathcal{L}( L^2(0, T; \bm{W}^{1, q}_{\Gamma_D}(\Omega)) \times \bm{W}^{1, q}_{\Gamma_D}(\Omega), H^1(0, T;\bm{W}^{1, q}_{\Gamma_D}(\Omega)))} \leq C
\end{equation}
for some $C > 0$ independently of $\vphi$. To see this, observe that since $Y = \bm{W}^{1, q}_{\Gamma_D}(\Omega) = D$, it holds that $\mathcal{A}(\vphi_1) - \mathcal{A}(\vphi_2) \in \mathcal{L}(X)$ for any $\vphi_1, \vphi_2 \in L^2(\Omega)$ with $\norm{\mathcal{A}(\vphi_1) - \mathcal{A}(\vphi_2)}_{\mathcal{L}(\bm{W}^{1, q}_{\Gamma_D}(\Omega))} \leq C \eqqcolon C_{\mathcal{I}}$. \par
As we shall see in the following lemma, these estimates imply that the dependency of $\mathfrak{L}^{-1}$ on $\vphi$ is also continuous. 

\begin{lemma}\label{lemma:convergece_L^1}
	Let $(\vphi_n)_{n \inN} \subset C^0([0, T]; L^2(\Omega))$ be a convergent sequence with limit $\vphi$. Then it holds for all $(\hat{\bm{f}}, \bm{u}_0)  \in  L^2(0, T; \bm{W}^{1, q}_{\Gamma_D}(\Omega)) \times \bm{W}^{1, q}_{\Gamma_D}(\Omega)$ 
	\begin{equation*}
		\mathfrak{L}^{-1}(\vphi_n, \bm{u}_0) (\hat{\bm{f}}) \rightarrow \mathfrak{L}^{-1}(\vphi,  \bm{u}_0) (\hat{\bm{f}}, \bm{u}_0)
		\quad \textrm{in} \quad H^1(0, T; \bm{W}^{1, q}_{\Gamma_D}(\Omega)). 
	\end{equation*}
\end{lemma} 
\begin{proof}
	We start by observing that 
	\begin{align*}
		&\norm{\mathfrak{L}^{-1}(\vphi_n, \bm{u}_0) (\hat{\bm{f}})- \mathfrak{L}^{-1}(\vphi, \bm{u}_0) (\hat{\bm{f}}) }_{H^1(\bm{W}^{1, q}_{\Gamma_D})}
		= \norm{
			\mathfrak{L}^{-1}(\vphi_n, \bm{u}_0) [\mathfrak{L}(\vphi) - \mathfrak{L} (\vphi_n)] (\mathfrak{L}^{-1}(\vphi, \bm{u}_0) (\hat{\bm{f}})) }_{H^1(\bm{W}^{1, q}_{\Gamma_D})}\\ 
		& \quad  \leq \norm{\mathfrak{L}^{-1}(\vphi, \bm{u}_0)}_{\mathcal{L}( L^2(\bm{W}^{1, q}_{\Gamma_D}) \times\bm{W}^{1, q}_{\Gamma_D}, H^1(\bm{W}^{1, q}_{\Gamma_D}))}  \norm{
			 [\mathfrak{L}(\vphi) - \mathfrak{L} (\vphi_n)] (\mathfrak{L}^{-1}(\vphi, \bm{u}_0) (\hat{\bm{f}})) }_{L^2(\bm{W}^{1, q}_{\Gamma_D}) \times\bm{W}^{1, q}_{\Gamma_D}}\\ 
		& \quad  \leq C \norm{
			[\mathfrak{L}(\vphi) - \mathfrak{L} (\vphi_n)] (\mathfrak{L}^{-1}(\vphi, \bm{u}_0) (\hat{\bm{f}})) }_{L^2(\bm{W}^{1, q}_{\Gamma_D}) \times\bm{W}^{1, q}_{\Gamma_D} }, 
	\end{align*}
	where the last estimate is due to \eqref{iq:uniform_bound_L}. Therefore, it is sufficient to establish the convergence 
	\begin{equation}\label{eq:convergence_A}
		\mathfrak{L}(\vphi_n) \bm{v} \rightarrow  \mathfrak{L} (\vphi) \bm{v} \quad \textrm{in} \quad  L^2(0, T; \bm{W}^{1, q}_{\Gamma_D}(\Omega)) \times \bm{W}^{1, q}_{\Gamma_D}(\Omega)
	\end{equation}
	for all $\bm{v} \in H^1(0, T; \bm{W}^{1, q}_{\Gamma_D}(\Omega))$. Since by definition 
	\begin{equation*}
		(\mathfrak{L}(\vphi_n) \bm{v}) = (\pt \bm{v} + \mathcal{A} (\vphi_n) \bm{v}, \bm{v}(0)), 
	\end{equation*}
	the only term depending on $\vphi_n$ is $\mathcal{A} (\vphi_n) \bm{v}$. Hence,  the result will follow immediately from
	\begin{equation*}
		\mathcal{A}(\vphi_n) \bm{v} \rightarrow \mathcal{A}(\vphi) \bm{v}  \quad \textrm{in} \quad  L^2(0, T; \bm{W}^{1, q}_{\Gamma_D}(\Omega)). 
	\end{equation*} 
	As by definition $(\mathcal{A} (\vphi_n) \bm{v})(t) = \mathcal{B}^{-1}(\vphi_n(t)) \mathcal{C}(\vphi_n(t) ) \bm{v}(t)$, Lemma \ref{lemma:convergence_B_C} together with \eqref{iq:operators_uniform} imply for all $t \in [0 ,T]$ the convergence 
	\begin{align*}
		&\norm{ (\mathcal{A} (\vphi_n) \bm{v})(t) - (\mathcal{A} (\vphi) \bm{v})(t)}_{ \bm{W}^{1, q}_{\Gamma_D}} 
		= \norm{ \mathcal{B}^{-1}(\vphi_n(t)) \mathcal{C}(\vphi_n(t) ) \bm{v}(t) -  \mathcal{B}^{-1}(\vphi(t)) \mathcal{C}(\vphi(t) ) \bm{v}(t)}_{\bm{W}^{1, q}_{\Gamma_D}}\\ 
		& \quad \leq  \begin{aligned}[t]
		&\norm{ \mathcal{B}^{-1}(\vphi_n(t))}_{\mathcal{L}(\bm{W}^{-1,q}_{\Gamma_D},\bm{W}^{1, q}_{\Gamma_D} )} \norm{ \mathcal{C}(\vphi_n(t) ) \bm{v}(t) - \mathcal{C}(\vphi(t) ) \bm{v}(t)}_{ \bm{W}^{-1, q}_{\Gamma_D}}\\ 
		&\quad + \norm{[ \mathcal{B}^{-1}(\vphi_n(t)) -  \mathcal{B}^{-1}(\vphi(t))]  \mathcal{C}(\vphi(t) ) \bm{v}(t)}
		\rightarrow 0. 
		\end{aligned} 		 
	\end{align*}
	Observing that \eqref{iq:operators_uniform} yields a uniform bound for the norm $\norm{\mathcal{A}(\vphi_n)}_{\mathcal{L}(\bm{W}^{1, q}_{\Gamma_D} )}$, we estimate 
	\begin{equation}
		\norm{\mathcal{A}(\vphi_n(t)) \bm{v}(t)}_{\bm{W}^{1, q}_{\Gamma_D} }^2 \leq C \norm{\bm{v}(t) }_{\bm{W}^{1, q}_{\Gamma_D} }^2, 
	\end{equation}
	which allows us to deduce \eqref{eq:convergence_A} with the help of Lebesgue's convergence theorem, cf. \cite[Sec. 3.25]{alt2013lineare}. 
\end{proof}

Lastly we need to define the appropriate right-hand side for our abstract Cauchy-problem. Therefore, we set 
\begin{align*} 
	\mathfrak{F} : L^\infty(0, T; H^2_{\bm{n}}(\Omega)) \times L^2(0, T; X(\Omega)) \times L^2(0, T; X(\Omega)) \rightarrow L^2(0, T; \bm{W}^{1, q}_{\Gamma_D}(\Omega)), \quad 
	(\vphi, \theta, p) \mapsto   \mathfrak{F}(\vphi, \theta, p), 
\end{align*}
where 
\begin{gather*}			
	\mathfrak{F} (\vphi, \theta, p)(t)= \mathcal{B}^{-1} (\vphi(t) ) \bm{f}^*(\vphi, \theta, p)(t) , \\ 
	\begin{aligned}
		\bm{f}^*(\vphi, \theta, p)(t)\bm{\eta} \coloneqq \int_\Omega \C(\vphi ) \Tau(\vphi ) : \E(\bm{\eta}) 
		+ \alpha(\vphi) p (\nabla \cdot \bm{\eta}) * \phi
		&+ \bm{f} \cdot \bm{\eta} \dx + \int_{\Gamma_N} \bm{g} \cdot \bm{\eta} \dH\\ 
		&- \vr \int_{\Omega} \nabla \theta \cdot \nabla (\alpha(\vphi) (\nabla \cdot \bm{\eta}) * \phi ) \dx.
	\end{aligned}  
\end{gather*}

\begin{lemma}\label{lemma:F_continuity}
	The operator $\mathfrak{F}$ is well-defined and strongly continuous. 
\end{lemma}
\begin{proof}
	To verify that $\mathfrak{F}$ is well-defined, i.e., $\mathfrak{F}(\vphi, p) \in L^2(0, T;  \bm{W}^{1, q}_{\Gamma_D}(\Omega))$ for all admissible $\vphi, \theta, p$, we recall the uniform bound on $\mathcal{B}^{-1}(\vphi)$, which implies that it suffices to show $\bm{f}^*(\vphi, p) \in L^2(0, T;  \bm{W}^{-1, q}_{\Gamma_D}(\Omega))$. To this end, let $\bm{\eta} \in  \bm{W}^{1, q'}_{\Gamma_D}(\Omega)$ and compute
	\begin{align*}
		\Big| \int_\Omega &\C(\vphi ) \Tau(\vphi ) : \E(\bm{\eta}) + \alpha(\vphi) p (\nabla \cdot \bm{\eta}) * \phi + \bm{f} \cdot \bm{\eta} \dx + \int_{\Gamma_N} \bm{g} \cdot \bm{\eta}\dH \Big|  \\ 
		&\leq C(\norm{\vphi }_{L^q} + 1) \norm{\E(\bm{\eta})}_{\bm{L}^{q'}} + C \norm{p}_{\bm{L}^q} \norm{\nabla \cdot \bm{\eta}}_{\bm{L}^{q'}} + C \norm{\bm{\eta}}_{\bm{L}^{q'}} + C \norm{\bm{\eta}}_{\bm{L}^{q'}(\Gamma_D)}\\ 
		&\leq C(\norm{\vphi }_{H^1} + 1) \norm{\bm{\eta}}_{\bm{W}^{1, q'}_{\Gamma_D}} + C \norm{p}_{X} \norm{ \bm{\eta}}_{\bm{W}^{1, q'}_{\Gamma_D}}. 
	\end{align*}
	Using Hölder's inequality and Young's inequality for convolutions, we further obtain
	\begin{align*}
	\left| \int_{\Omega} \nabla \theta \cdot \nabla (\alpha(\vphi) (\nabla \cdot \bm{\eta}) * \phi ) \dx \right| 
	&= \left| \int_{\Omega} \nabla \theta \cdot \left(\alpha'(\vphi) \nabla \vphi (\nabla \cdot \bm{\eta}) * \phi \right) + 
								\nabla \theta \cdot \left( \alpha(\vphi) (\nabla \cdot \bm{\eta} ) * \nabla \phi \right) \dtx \right|\\ 
	& \leq C \left( \norm{\nabla \theta}_{\bm{L}^2} \norm{\nabla \vphi}_{\bm{L}^4} \norm{(\nabla \cdot \bm{\eta}) * \phi }_{\bm{L}^4}
							+\norm{\nabla \theta}_{\bm{L}^2} \norm{(\nabla \cdot \bm{\eta}) * \nabla \phi }_{L^2}\right)\\ 
	& \leq C\left(  \norm{\nabla \theta}_{\bm{L}^2} \norm{\nabla \vphi}_{\bm{L}^4} \norm{\nabla \cdot \bm{\eta}}_{L^{q'}} \norm{\phi}_{L^s}
							+\norm{\nabla \theta}_{\bm{L}^2} \norm{\nabla \cdot \bm{\eta}}_{L^{q'}} \norm{\nabla \phi}_{\bm{L}^{r}} \right)								
	\end{align*}
	where $s = \tfrac{4q'}{5q' - 4}$ and $r = \tfrac{2q'}{3q'-2}$. 
	Hence, 
	\begin{align*}\label{eq:estimate_f*}
		\norm{\bm{f}^*(\vphi, \theta, p)(t)}_{ \bm{W}^{-1, q}_{\Gamma_D}} 
		&= \sup_{ \norm{\bm{\eta}}_{ \bm{W}^{1, q'}_{\Gamma_D}} =1} |\bm{f}^* (\vphi, \theta,  p)(t) \bm{\eta}| \\ 
		&\leq  C( \norm{\vphi(t)}_{H^1} +\norm{p(t)}_{X} + \norm{\nabla \theta}_{\bm{L}^2} \norm{\nabla \vphi}_{\bm{L}^6} + \norm{\nabla \theta}_{\bm{L}^2} +1)  , \numberthis
	\end{align*}
	which in turn yields, due to $\vphi \in L^\infty(0, T; H^2_{\bm{n}}(\Omega))$, that 
	\begin{align*}
		\norm{\bm{f}^*(\vphi, \theta, p)}_{L^2(\bm{W}^{-1, q}_{\Gamma_D})}^2 &= \int_0^T \norm{\bm{f}^*(\vphi, \theta, p)(t)}_{ \bm{W}^{-1, q}_{\Gamma_D}}^2 \dt \\ 
		&\leq \int_0^T C( \norm{\vphi(t)}_{H^1} +\norm{p(t)}_{X} + \norm{\nabla \theta}_{L^2} \norm{\nabla \vphi}_{L^6} + \norm{\nabla \theta}_{L^2} +1)^2 \dt < \infty . 
	\end{align*}
	\par 
	\medskip 
	To see that $\mathfrak{F}$ is strongly continuous, we take a convergent sequence 
	\begin{equation*}
			(\vphi_n, \theta_n, p_n)_{n \inN} \subset L^\infty(0, T; H^1(\Omega))  \times L^2(0, T; X(\Omega)) \times  L^2(0, T; X(\Omega))
	\end{equation*}
 and start by considering the convergence 
	\begin{equation}\label{eq:convergence_f*}
		\bm{f}^*(\vphi_n, \theta_n,  p_n) (t) \rightarrow \bm{f}^*(\vphi, \theta, p)(t) \quad \textrm{in} \quad \bm{W}^{-1, q}_{\Gamma_D}(\Omega)
	\end{equation}
	 for any $t \in [0, T]$, where $\vphi, \theta, p$ are the respective limits. Similar to the calculation above, we obtain the estimate 
	\begin{align*}
		&\left| \int_{\Omega} \nabla \theta_n \cdot \nabla (\alpha(\vphi_n) (\nabla \cdot \bm{\eta}) * \phi )  - \nabla \theta \cdot \nabla (\alpha(\vphi) (\nabla \cdot \bm{\eta}) * \phi )  \dx \right| \\ 
		& \  \leq  \norm{ \alpha'(\vphi_n) \nabla \theta_n \cdot \nabla \vphi_n - \alpha'(\vphi) \nabla \theta \cdot \nabla \vphi}_{L^{4/3}} \norm{(\nabla \cdot \bm{\eta}) * \phi }_{L^4} 
									+ \norm{\alpha(\vphi_n) \nabla \theta_n - \alpha(\vphi) \nabla \theta}_{\bm{L}^2} \norm{(\nabla \cdot \bm{\eta} ) * \nabla \phi}_{\bm{L}^2} \\ 
		&  \  \leq C \Big(  \norm{ \alpha'(\vphi_n) \nabla \theta_n \cdot \nabla \vphi_n - \alpha'(\vphi) \nabla \theta \cdot \nabla \vphi}_{L^{4/3}} +  \norm{\alpha(\vphi_n) \nabla \theta_n - \alpha(\vphi) \nabla \theta}_{\bm{L}^2} \Big) 
		\norm{\bm{\eta}}_{\bm{W}^{1, q'}_{\Gamma_D}}. 
	\end{align*}
	This leads to 
	\begin{align*}
		&\abs{	\bm{f}^*(\vphi_n, \theta_n, p_n)(t) \bm{\eta} - \bm{f}^*(\vphi, \theta, p)(t) \bm{\eta}} \\ 
		& \quad \leq \norm{\C(\vphi_n ) \Tau(\vphi_n ) - \C(\vphi ) \Tau(\vphi )}_{\bm{L}^q} \norm{\bm{\eta}}_{\bm{W}^{1, q'}}
		+ \norm{\alpha(\vphi_n) p_n - \alpha(\vphi) p}_{\bm{L}^q} \norm{\bm{\eta}}_{\bm{W}^{1, q'}} \\ 
		& \quad \quad +C \Big(  \norm{ \alpha'(\vphi_n) \nabla \theta_n \cdot \nabla \vphi_n - \alpha'(\vphi) \nabla \theta \cdot \nabla \vphi}_{L^{4/3}} +  \norm{\alpha(\vphi_n) \nabla \theta_n - \alpha(\vphi) \nabla \theta}_{\bm{L}^2} \Big) 
		\norm{\bm{\eta}}_{\bm{W}^{1, q'}}
	\end{align*}
	and thus
	\begin{align*}
		&\norm{\bm{f}^*(\vphi_n, \theta_n,  p_n)(t) - \bm{f}^*(\vphi, \theta, p)(t)}_{\bm{W}^{-1,q }_{\Gamma_D}} = 
		\sup_{\norm{\bm{\eta}}_{\bm{W}^{1, q'}} = 1} \abs{	\bm{f}^*(\vphi_n, \theta_n,  p_n)(t) \bm{\eta} - \bm{f}^*(\vphi,\theta,  p)(t) \bm{\eta}} \\ 
		 & \quad \begin{aligned}[t]
			&\leq \norm{\C(\vphi_n ) \Tau(\vphi_n) - \C(\vphi ) \Tau(\vphi )}_{\bm{L}^q} 
			+ \norm{\alpha(\vphi_n) p_n - \alpha(\vphi) p}_{L^q} \\ 
			&\quad + C \Big(  \norm{ \alpha'(\vphi_n) \nabla \theta_n \cdot \nabla \vphi_n - \alpha'(\vphi) \nabla \theta \cdot \nabla \vphi}_{L^{4/3}} +  \norm{\alpha(\vphi_n) \nabla \theta_n - \alpha(\vphi) \nabla \theta}_{\bm{L}^2} \Big). 
		\end{aligned}
	\end{align*}
	From the assumptions, it follows that for every subsequence $(n_k)_{k \inN}$ there exists another subsequence $(n_{k_l})_{l \inN}$ such that for almost all $t \in (0, T)$
	\begin{gather}\label{eq:convergence_phi_p}
		\begin{aligned}
			\vphi_{n_{k_l}} (t)& \rightarrow \vphi(t) \quad \textrm{in} \quad H^2_{\bm{n}}(\Omega) \quad \textrm{and}  \quad \textrm{a.e. in } \quad \Omega  \quad \textrm{and}  \quad &\nabla \vphi_{n_{k_l}} (t)  \rightarrow \nabla \vphi(t) \quad  \textrm{a.e. in } \quad \Omega, \\ 
			\theta_{n_{k_l}} (t)& \rightarrow \theta(t) \quad \, \textrm{in} \quad X(\Omega) \quad \ \, \textrm{and}  \quad \textrm{a.e. in } \quad \Omega  \quad \textrm{and}  \quad &\nabla \theta_{n_{k_l}} (t) \rightarrow \nabla \theta(t) \quad  \textrm{a.e. in } \quad \Omega, \\ 
			p_{n_{k_l}}(t)& \rightarrow p (t) \quad\, \textrm{in} \quad X(\Omega)  \quad \ \, \textrm{and}  \quad \textrm{a.e. in } \quad \Omega. 
		\end{aligned}\numberthis 
	\end{gather}
	The continuity of $\C, \Tau, \alpha, \alpha'$ therefore implies that for almost every $t \in (0, T)$
	\begin{align*}
		\C(\vphi_{n_{k_l}} (t)) \Tau(\vphi_{n_{k_l}}(t))  - \C(\vphi (t)) \Tau(\vphi(t)) & \rightarrow 0
		  \quad \textrm{and} \quad 
		\alpha(\vphi_{n_{k_l}}(t)) p_{n_{k_l}} (t)- \alpha(\vphi(t)) p(t) \rightarrow 0, 
	\end{align*}
	pointwise a.e.~in $\Omega$, along with the pointwise a.e.~convergences 
	\begin{align*}
		\alpha'(\vphi_{n_{k_l}}(t)) \nabla \theta_{n_{k_l}}(t) \cdot \nabla \vphi_{n_{k_l}}(t) - \alpha'(\vphi(t)) \nabla \theta(t) \cdot \nabla \vphi(t) &\rightarrow 0, \\ 
		\alpha(\vphi_{n_{k_l}}(t)) \nabla \theta_{n_{k_l}}(t)- \alpha(\vphi(t)) \nabla \theta(t) &\rightarrow 0. 
	\end{align*} Hence, the growth conditions on $\C, \Tau, \alpha, \alpha'$ and Lebesgue's generalized convergence theorem, cf. \cite[Sec. 3.25]{alt2013lineare}, along with \eqref{eq:convergence_phi_p} yield the strong convergences
		\begin{align*}
		\norm{\C(\vphi_{n_{k_l}} (t)) \Tau(\vphi_{n_{k_l}}(t) ) - \C(\vphi (t)) \Tau(\vphi (t))}_{\bm{L}^q} &\rightarrow 0, \\ 
		\norm{\alpha(\vphi_{n_{k_l}}(t)) p_{n_{k_l}} (t)- \alpha(\vphi(t)) p(t)}_{L^q} &\rightarrow 0,\\ 
		\norm{ \alpha'(\vphi_{n_{k_l}}(t)) \nabla \theta_{n_{k_l}}(t)\cdot \nabla \vphi_{n_{k_l}}(t) - \alpha'(\vphi(t)) \nabla \theta (t) \cdot \nabla \vphi (t)}_{L^{4/3}} & \rightarrow 0, \\ 
		\norm{\alpha(\vphi_{n_{k_l}}(t)) \nabla \theta_{n_{k_l}}(t) - \alpha(\vphi(t)) \nabla \theta(t)}_{\bm{L}^2} & \rightarrow 0
	\end{align*}
	for almost all $t \in (0, T)$. Along with Lemma \ref{lemma:convergence_B_C} and the uniform bound \eqref{iq:operators_uniform}, we infer, again for almost all $t \in (0, T)$, 
	\begin{align*}
		&\norm{ \mathcal{B}^{-1} (\vphi_{n_{k_l}}) \bm{f}^*(\vphi_{n_{k_l}}, \theta_{n_{k_l}}, p_{n_{k_l}}) - \mathcal{B}^{-1} (\vphi ) \bm{f}^*(\vphi, \theta, p)}_{  \bm{W}^{1, q}_{\Gamma_D}}\\ 
		& \quad  \leq \norm{ \mathcal{B}^{-1}(\vphi_{n_{k_l}}) (\bm{f}^*(\vphi_{n_{k_l}}, \theta_{n_{k_l}}, p_{n_{k_l}}) - \bm{f}^*(\vphi, \theta, p))}_{\bm{W}^{1, q}_{\Gamma_D}}
		+ \norm{( \mathcal{B}^{-1}(\vphi_{n_{k_l}}) -  \mathcal{B}^{-1}(\vphi)) \bm{f}^*(\vphi, \theta, p)}_{\bm{W}^{1, q}_{\Gamma_D}}\\ 
		& \quad \leq C  \norm{ \bm{f}^*(\vphi_{n_{k_l}}, \theta_{n_{k_l}}, p_{n_{k_l}}) - \bm{f}^*(\vphi, \theta, p)}_{\bm{W}^{-1, q}_{\Gamma_D}} 
		+ \norm{( \mathcal{B}^{-1}(\vphi_{n_{k_l}}) -  \mathcal{B}^{-1}(\vphi)) \bm{f}^*(\vphi, \theta, p)}_{\bm{W}^{1, q}_{\Gamma_D}} \rightarrow 0, 
	\end{align*}
	i.e., the chosen subsequence converges pointwise almost everywhere in $(0, T)$. 
	\par 
	Utilizing the convergences of $(p_n)_{n \inN}$ in $L^2(0, T; X(\Omega))$ and $(\vphi_n)_{n \inN}$ in $L^\infty(0, T; H^2_{\bm{n}}(\Omega))$ and $(\theta_n)_{n \inN}$ in $L^2(0, T; X(\Omega))$, the estimate \eqref{eq:estimate_f*} and the uniform bound on $\mathcal{B}^{-1}(\vphi_n)$, we infer with the help of Lebesgue's generalized convergence theorem, cf. \cite[Sec. 3.25]{alt2013lineare}, 
	\begin{align*}
		\mathfrak{F}(\vphi_{n_{k_l}}, \theta_{n_{k_l}}, p_{n_{k_l}}) \rightarrow \mathfrak{F}(\vphi, \theta, p)  \quad \textrm{in} \quad L^2(0, T; \bm{W}^{1, q}_{\Gamma_D}(\Omega)). 
	\end{align*}
	In summary, we obtain that every subsequence $(n_k)_{k \inN}$ of contains yet another subsequence such that the desired convergence holds. In particular, $(\mathfrak{F}(\vphi_n, p_n))_{n \inN}$ is precompact with a unique accumulation point, which yields the assertion.  
\end{proof}
\par 
\medskip 
Suppose we have found $(\vphi, p, \bm{u})$ in the corresponding spaces such that 
\begin{equation}\label{eq:galerkin7_equivalent}
	\mathfrak{L}_1(\vphi) \bm{u} = \mathfrak{F}(\vphi, \theta, p) 
\end{equation}
almost everywhere, or equivalently, 
\begin{align*}
	\pt \bm{u} +  \mathcal{B}^{-1}(\vphi) \mathcal{C}(\vphi) \bm{u}  =  \mathcal{B}^{-1} (\vphi ) \bm{f}^*(\vphi, \theta, p). 
\end{align*}
Multiplying $\mathcal{B} (\vphi )$ from the left side then yields that the equation
\begin{align*}
	\int_{\Omega} \C_\nu (\vphi ) \E(\pt \bm{u})  :  \E(\bm{\eta}) &+ \WE(\vphi_k , \E(\bm{u}) ) :  \E(\bm{\eta})  - \alpha(\vphi) p (\nabla \cdot \bm{\eta}) * \phi  \dx \\ 
	&= \int_{\Omega} \bm{f} \cdot \bm{\eta} \dx + \int_{\Gamma_N} \bm{g}\cdot \bm{\eta} \dH 
		- \vr \int_{\Omega} \nabla \theta \cdot \nabla (\alpha(\vphi) (\nabla \cdot \bm{\eta}) * \phi ) \dx
\end{align*} 
holds for all $\bm{\eta} \in \bm{W}^{1, q'}_{\Gamma_D}(\Omega)$, i.e., the functions $(\vphi, p, \bm{u})$ satisfy \eqref{eq:galerkin7} for almost all $t \in (0, T)$. Rewriting \eqref{eq:galerkin7_equivalent} with the help of $\mathcal{G}$ and exploiting that $\mathfrak{L}$ is invertible leads to the fixed point equation
\begin{equation}\label{eq:fixed_point}
	\bm{u} = \mathfrak{L}^{-1}(\mathcal{G}_1\bm{u}, \bm{u}_0) \left(\mathfrak{F}\mathcal{G}\bm{u}  \right) \eqqcolon \mathfrak{T} \bm{u}
\end{equation}
where 
\begin{align*}
	\mathfrak{T}: H^1(0, T; \bm{X}(\Omega)) \rightarrow H^1(0, T; \bm{X}(\Omega)), \quad 
	\bm{u}  \mapsto  \mathfrak{L}^{-1}(\mathcal{G}_1\bm{u}, \bm{u}_0) \left(\mathfrak{F}\mathcal{G}\bm{u} \right). 
\end{align*}
Note that the initial condition for the displacement is incorporated in the definition of the operator $\mathfrak{T}$.  
It is apparent from the discussion above that for a solution of \eqref{eq:fixed_point}, the quintuple $(\vphi, \mu, \theta, p, \bm{u}) \coloneqq (\mathcal{G}^*\bm{u}, \bm{u})$ satisfies \eqref{eq:galerkin1}-\eqref{eq:galerkin7}, i.e., it is a solution to the semi-discretized system.  

\begin{lemma} \label{lemma:leray_schauder_regular}
	There exists at least one $\bm{u} \in H^1(0, T; \bm{X} (\Omega))$ such that $\bm{u} = \mathfrak{T} \bm{u}$.  
\end{lemma}

\begin{proof}
	Since this is a direct consequence of the Leray-Schauder principle, it remains to verify that 
	\begin{enumerate}[label*=(\roman*), nosep]
		\item $\mathfrak{T}$ is completely continuous; 
		\item there exists $r > 0$ such that for all $\bm{v}$ satisfying
			$\bm{v} = \lambda \mathfrak{T} \bm{v}  \textrm{ for any } \lambda \in [0, 1]$
		we have $\norm{\bm{v}}_{ H^1(\bm{X})} \leq R$. 
	\end{enumerate}
	\par 
	\medskip
	\noindent
	\textit{Ad (i):} Let $(\bm{u}_n)_{n \inN} \subset H^1(0, T; \bm{X}(\Omega))$ be a bounded sequence. Due to the discussion above \eqref{def:G}, we already know that the mapping 
\begin{equation*}
	\bm{u}_n \mapsto \mathcal{G}(\bm{u}_n) = (\vphi_n, \theta_n, p_n) \in C^0([0, T]; H^2_{\bm{n}}(\Omega)) \times C^0([0, T]; X(\Omega))  \times C^0([0, T]; X(\Omega))
\end{equation*}
is compact, allowing us to extract a convergent subsequence. Without relabeling, we can now apply Lemma \ref{lemma:F_continuity} and deduce 
\begin{equation}
	\mathfrak{F}\mathcal{G}\bm{u}_n = \mathfrak{F}(\vphi_n, \theta_n, p_n) \rightarrow \mathfrak{F}(\vphi, \theta, p) \quad \textrm{in} \quad L^2(0, T; \bm{X}(\Omega)). 
\end{equation}
Hence, Lemma \ref{lemma:convergece_L^1} along with the uniform estimate \eqref{iq:uniform_bound_L} yield
\begin{align*}
	&\norm{\mathfrak{L}^{-1}(\vphi_n, \bm{u}_0) (\mathfrak{F} (\vphi_n, \theta_n, p_n))- \mathfrak{L}^{-1}(\vphi, \bm{u}_0)(\mathfrak{F}(\vphi,\theta, p)) }_{H^1(\bm{X})}\\ 
	& \quad \leq
	\begin{aligned}[t]
		&\norm{\mathfrak{L}^{-1}(\vphi_n, \bm{u}_0)}_{\mathcal{L}( L^2(0, T; \bm{X}(\Omega)), H^1(0, T; \bm{X}(\Omega)))} \norm{\mathfrak{F} (\vphi_n,\theta_n, p_n) - \mathfrak{F}(\vphi, \theta, p)}_{L^2(\bm{X})} \\ 
		& \quad + \norm{(\mathfrak{L}^{-1}(\vphi_n, \bm{u}_0) - \mathfrak{L}^{-1}(\vphi, \bm{u}_0)) (\mathfrak{F}(\vphi, \theta,  p))}_{H^1(\bm{X})}  \rightarrow 0 
	\end{aligned}
\end{align*}
and we conclude that, along a suitable subsequence, 
\begin{equation*}
	\mathfrak{T}\bm{u}_n  \rightarrow \mathfrak{T}\bm{u} \quad \textrm{in} \quad H^1(0, T; \bm{X}(\Omega)). 
\end{equation*}
\par
\medskip
\textit{Ad (ii):} Suppose $\lambda \in [0, 1]$ and assume $\bm{v}$ satisfies the equation $\bm{v} = \lambda \mathfrak{T} \bm{v}$, or equivalently that $\mathfrak{L}_1(\mathcal{G} \bm{v}) \bm{v} = \lambda \mathfrak{F}\mathcal{G}\bm{v}$,
which follows from the definition of $\mathfrak{T}$ and the linearity of $\mathcal{L}^{-1}$. By definition, the functions $(\vphi, \mu, \theta, p) \coloneqq \mathcal{G}^*\bm{v}$ are a solution to the system \eqref{eq:differential_algebraic1}-\eqref{eq:differential_algebraic4} of differential-algebraic equations with the corresponding initial conditions.  Similar to above, we further obtain 
\begin{equation*}
		 \mathcal{B}(\vphi) \pt \bm{v} +  \mathcal{C}(\vphi) \bm{v}  = \lambda \bm{f}^*(\vphi, p)
\end{equation*}
for almost all $t \in [0, T]$ and therefore
\begin{align*}
		&\int_{\Omega} \C_\nu (\vphi ) \E(\pt \bm{v})  :  \E(\bm{\eta}) 
		+\C (\vphi ) \E( \bm{v})  :  \E(\bm{\eta})  \dx  \\ 
		 & \quad = 
		 \begin{aligned}[t]
		 	\lambda  \int_\Omega \C(\vphi )  \Tau(\vphi ) :  \E(\bm{\eta})
		 	+\alpha(\vphi) p \, (\nabla \cdot \bm{\eta}) * \phi
		 	&+\bm{f} \cdot \bm{\eta} \dx + \lambda  \int_{\Gamma_N} \bm{g}\cdot \bm{\eta} \dH \\ 
		 	&-  \lambda  \int_{\Omega} \vr \nabla \theta \cdot \nabla (\alpha(\vphi) (\nabla \cdot \bm{\eta}) * \phi ) \dx
		 \end{aligned}
\end{align*}
for all $\bm{\eta} \in \bm{X}(\Omega)$. Adding to and subtracting from the right-hand side the term
\begin{equation*}
	(1 - \lambda )  \int_\Omega \C(\vphi )  \Tau(\vphi ) :  \E(\bm{\eta}) \dx 
	+ \int_\Omega M(\vphi)(\theta - \alpha(\vphi) \nabla \cdot \bm{u}) \alpha(\vphi)  \nabla \cdot \bm{\eta} \dx 
\end{equation*}
leads to
\begin{align*}
		&\int_{\Omega} \C_\nu (\vphi ) \E(\pt \bm{v})  :  \E(\bm{\eta}) + \WE(\vphi_k , \E(\bm{v}) ) :  \E(\bm{\eta})  - M(\vphi)(\theta - \alpha(\vphi) \nabla \cdot \bm{u}) \alpha(\vphi)  \nabla \cdot \bm{\eta} \dx \\ 
		& \quad  =   \begin{aligned}[t]
			&\lambda  \int_\Omega  \bm{f} \cdot \bm{\eta} +\alpha(\vphi) p \,(\nabla \cdot \bm{\eta}) * \phi   \dx + \lambda  \int_{\Gamma_N} \bm{g}\cdot \bm{\eta} \dH \\ 
			&+ ( \lambda - 1)  \int_\Omega  \C(\vphi )  \Tau(\vphi ) :  \E(\bm{\eta})
			-M(\vphi)(\theta - \alpha(\vphi) \nabla \cdot \bm{u}) \alpha(\vphi)  \nabla \cdot \bm{\eta} \dx\\ 
			&-  \lambda  \int_{\Omega} \vr \nabla \theta \cdot \nabla (\alpha(\vphi) (\nabla \cdot \bm{\eta}) * \phi ) \dx.
		\end{aligned} \numberthis \label{eq:galerkin7_lambda}
\end{align*}
After testing the system of differential-algebraic equations \eqref{eq:differential_algebraic1}-\eqref{eq:differential_algebraic4} just like in Section \ref{sec:energy_estimates} and \eqref{eq:galerkin7_lambda} with $\pt \bm{v}$, we arrive at 
\begin{align*}
	&\norm{m(\vphi)^{1/2} \nabla \mu}_{L^2}^2 
	+\norm{\kappa(\vphi)^{1/2} \nabla p}_{L^2}^2
	+ \int_\Omega \C_{\nu} (\vphi ) \E(\pt \bm{v}) : \E(\pt \bm{v}) \dx \\ 
	&\quad +\frac{d}{dt} \Big[
	\begin{aligned}[t]
		\quad &\int_\Omega\frac{\vepsilon}{2} \abs{\nabla \vphi}^2
		+ \frac{\varrho^{1/2}}{2} \abs{\Delta \vphi}^2 
		+\frac{\vr}{2} \abs{\nabla \theta}^2 
		 + \frac{1}{\vepsilon} \psi(\vphi) \dx \\ 
		&+ \int_\Omega W(\vphi , \E(\bm{v})) - \lambda \bm{f} \cdot \bm{v}  \dx - \lambda \int_{\Gamma_N} \bm{g} \cdot \bm{v} \dH
		+ \int_\Omega  \frac{M(\vphi)}{2} (\theta - \alpha(\vphi) \nabla \cdot \bm{v})^2 \dx
		\Big]
	\end{aligned}\\ 
	&\quad =  \begin{aligned}[t]
		&\langle R(\vphi, \E(\bm{v}), \theta), \mu \rangle + \langle S_f(\vphi,\E(\bm{v}), \theta), p \rangle \\ 
		&+   \int_\Omega \begin{aligned}[t]
			&( \lambda - 1) \C(\vphi )  \Tau(\vphi ) :  \E( \pt \bm{v})
			-M(\vphi)(\theta - \alpha(\vphi) \nabla \cdot \bm{v}) \alpha(\vphi)  \nabla \cdot \pt \bm{v} \dx\\ 
			&+  \lambda  \int_{\Omega} \alpha(\vphi) p \, (\nabla  \cdot  \pt \bm{v}) * \phi -  \vr \nabla \theta \cdot \nabla (\alpha(\vphi) (\nabla \cdot \pt \bm{v}) * \phi ) \dx, 
		\end{aligned}
	\end{aligned}\numberthis \label{eq:leray_schauder_est_1}
\end{align*}
which only slightly differs from \eqref{eq:galerkin_sum}. Employing the Cauchy-Schwarz inequality and invoking a trace theorem yields for all $\rho_{\bm{v}} > 0$
\begin{align*}
	-\lambda  \int_\Omega \bm{f} \cdot \bm{v} \dx - \lambda \int_{\Gamma_N} \bm{g} \cdot \bm{v}\dH 
	\geq - \lambda \rho_{\bm{v}} \norm{\bm{v}}_{\bm{X}}^2 - \lambda C(\rho_{\bm{v}}, \bm{f}, \bm{g}) 
	\geq - \rho_{\bm{v}} \norm{\bm{v}}_{\bm{X}}^2 - C(\rho_{\bm{v}}, \bm{f}, \bm{g}), 
\end{align*}
where we used $\lambda \in [0, 1]$. Hence, instead of \eqref{estimate_elastic_below}, we obtain
for suitably small $\rho_{\bm{v}}$ and all $t \in [0, T]$, $\lambda \in [0, 1]$
	\begin{align*}
			\int_\Omega W(\vphi, \E(\bm{v})) - \lambda \bm{f} \cdot \bm{v} \dx - \lambda \int_{\Gamma_N} \bm{g} \cdot \bm{v} \dH
			\geq C_{\bm{v}} \norm{\bm{v}}_{\bm{X}}^2 - \rho_\vphi \norm{\vphi}_{L^p}^p - C. 
		\end{align*}
Moreover, instead of \eqref{eq:estimate_elastic_initial}, we estimate
\begin{equation*}
	\int_\Omega W(\vphi_{0, k} , \E(\bm{v}_0)) \dx 
	- \lambda \bm{f} \cdot \bm{v}_0 \dx - \lambda \int_{\Gamma_N} \bm{g} \cdot \bm{u}_0 \dH
	\leq C(\norm{\vphi_{0, k}}_{L^2}^2 + \norm{\bm{v}_0}_{\bm{X}}^2 +1 ). 
\end{equation*}
Finally, we exploit the properties of the tensor $\C_{\nu}$ and invoke Korn's inequality to obtain
\begin{equation*}\label{estimate_visco_below}
			\int_\Omega \C_{\nu} (\vphi) \E( \pt \bm{v}) : \E( \pt \bm{v}) \dx \geq C_\nu \norm{\pt \bm{v}}_{\bm{X}}^2.
		\end{equation*}
\par 
It remains to estimate the right-hand side uniformly in $\lambda$. Since $\abs{1- \lambda} \leq 1$, we find with the help of assumption \ref{A:W} and Young's inequality 
\begin{align*}
	- (1 - \lambda )  \int_\Omega \C(\vphi )  \Tau(\vphi ) :  \E( \pt \bm{v}) \dx 
	&\leq \abs{1-\lambda} (C(\norm{\vphi}_{L^2} + 1)  \norm{\pt \bm{v}}_{\bm{X}}) \\ 
	 &\leq C_{\bm{v}, 1} (\norm{\vphi}_{L^2}^2 + 1) +\frac{ \delta_{\bm{v}}}{3} \norm{\pt \bm{v}}_{\bm{X}}^2. 
\end{align*}
Taking note of the fact that $\alpha(\vphi) (\nabla \cdot \pt \bm{v}) * \phi \in H^1(\Omega)$, the identity \eqref{eq:p_identity_new} allows us to rewrite the last term in \eqref{eq:leray_schauder_est_1} as 
\begin{align*}
	 & \lambda \int_\Omega  p \alpha(\vphi)  (\nabla \cdot \pt \bm{v}) * \phi  \dx \\ 
	 \quad &= 
	 \lambda  \int_\Omega \vr \nabla \theta \cdot \nabla (\alpha(\vphi) (\nabla \cdot \pt \bm{v})* \phi) \dx + 
	\lambda  \int_\Omega \Pi_k^y(M(\vphi) (\theta - \alpha(\vphi) \nabla \cdot \bm{v})) \alpha(\vphi)  (\nabla \cdot \pt \bm{v}) * \phi  \dx
\end{align*}
and find that the the first integral cancels out against the last term in \eqref{eq:leray_schauder_est_1}. For the second term, we compute 
\begin{align*}
	\lambda  \int_\Omega \Pi_k^y(M(\vphi) (\theta - \alpha(\vphi) \nabla \cdot \bm{v})) \alpha(\vphi)  (\nabla \cdot \pt \bm{v} ) * \phi  \dx 
	 &\leq  \frac{\lambda}{4\delta_{\bm{v}}} \norm{\Pi_k^y(M (\theta - \alpha(\vphi) \nabla \cdot \bm{v}))}_{L^2}^2 + \frac{ \delta_{\bm{v}}}{3} \norm{\pt \bm{v}}_{H^1}^2\\ 
	&\leq C_{\bm{v}, 2}(\norm{\theta}_{L^2}^2 +  \norm{\bm{v}}_{\bm{X}}^2) + \frac{ \delta_{\bm{v}}}{3} \norm{\pt \bm{v}}_{\bm{X}}^2, 
\end{align*}
where we used \ref{A:phase_coefficients} and the stability estimate $\norm{\Pi_y^k(\cdot )}_{L^2} \leq \norm{\cdot}_{L^2}$. Similarly, we further obtain
\begin{align*}
	\Big( \int_\Omega M(\vphi)(\theta - \alpha(\vphi) \nabla \cdot \bm{v}) \alpha(\vphi)  \nabla \cdot \pt \bm{v} \dx \Big) 
	 \leq C_{\bm{v}, 3}(\norm{\theta}_{L^2}^2 +  \norm{\bm{v}}_{\bm{X}}^2) + \frac{ \delta_{\bm{v}}}{3} \norm{\pt \bm{v}}_{\bm{X}}^2. 
\end{align*}
\par 
We wish to point out that the constants $C_{\bm{v}, 1}, C_{\bm{v}, 2}, \delta_{\bm{v}}$ in the inequalities above are independent of $\lambda$. 
Hence, we can use the computation from Section \ref{sec:energy_estimates} to obtain 
\begin{gather}
	\begin{align*}
		&\begin{aligned}[t]
			&(\underline{m} - \rho_\mu C_p) \int_0^t \norm{\nabla \mu}^2 \dt  
			+(\underline{\kappa} - \rho_p C_p) \int_0^t \norm{ \nabla p}^2 \dt 
			+ (C_\nu -  \delta_{\bm{v}})  \int_0^t \norm{\pt \bm{v}}_{\bm{X}}^2 \dt  \\  
			&+\frac{\vepsilon}{2} \norm{\nabla \vphi(t)}_{L^2}^2 
			+ \varrho^{1/2} \norm{\Delta \vphi(t)}_{L^2}^2
			+ \frac{1}{2\vepsilon}  \norm{\psi(\vphi(t))}_{L^1} 
			+(\frac{\gamma_{\psi_1} - \rho_{\psi_2}}{2 \vepsilon} -  \rho_\vphi ) \norm{\vphi(t)}_{L^p}^p  \\ 
			&+ (C_{\bm{v}} - \frac{\underline{M}} {2}(1- \frac{1}{\rho_\theta}) \overline{\alpha}^2   )\norm{\bm{v}(t)}_{\bm{X}}^2
			+ \frac{\underline{M}} {2} (1-\rho_\theta) \norm{\theta(t)}_{L^2}^2 + \frac{\vr}{2} \norm{\nabla \theta (t)}_{L^2}^2
		\end{aligned}\\ 
		\leq \, & \begin{aligned}[t]
			&C \int_0^t \norm{\psi(\vphi)}_{L^1} + \norm{\vphi}_{L^2}^2 + \norm{\bm{v}}_{\bm{X}}^2   + 1\dt 
			 + \int_0^t ( C_{\bm{v}, 1} + C_{\bm{v}, 2} + C_{\bm{v}, 3}) (\norm{\theta}_{L^2}^2 +  \norm{\bm{v}}_{\bm{X}}^2 + 1) \dt \\ 
			&+ C( \norm{\vphi_{0, k}}_{H^1}^2 + \varrho^{1/2} \norm{\Delta \vphi_{0, k}}_{L^2}^2 +  \frac{1}{\vepsilon}  \norm{\psi(\vphi_{0, k})}_{L^1} +  \norm{\bm{v}_0}_{\bm{X}}^2 + \norm{\theta_{0,k}}_{L^2}^2 + \vr \norm{\nabla \theta_{k, 0}}_{L^2}^2 + 1). 
		\end{aligned} 
	\end{align*}
\end{gather}
As before, we can choose all parameters $\rho_\vphi, \rho_\mu,\rho_p,\rho_\theta, \rho_{\psi_2}$ and $\delta_{\bm{v}}$ suitably, such that this simplifies to  
\begin{align*}
	\norm{\vphi(t)}_{L^2}^2 &+ \norm{\psi(\vphi(t))}_{L^1} +  \norm{\theta(t)}_{L^2}^2 + \norm{\bm{v} (t)}_{\bm{X}}^2
	+ \int_0^T \norm{\pt \bm{v}}_{H^1}^2 \dt\\ 
	&\leq  C \int_0^t \norm{\vphi}_{L^2}^2 + \norm{\psi(\vphi)}_{L^1} +  \norm{\theta}_{L^2}^2   +  \norm{\bm{v}}_{\bm{X}}^2  \dt + C
\end{align*}
with some constant $C > 0$ independent of $\lambda \in [0, T]$. Hence, Gronwall's lemma in particular yields the existence of some $r > 0$ such that 
\begin{equation*}
	\norm{\bm{v}}_{H^1(\bm{X})} \leq r
\end{equation*}
independently of $\lambda$. Thus, the Leray-Schauder principle (Theorem \ref{leray_schauder}) is applicable and we deduce the existence of a fixed-point $\bm{u}$, as desired.  	
\end{proof}


\subsection{\textit{A priori} estimates and compactness results }

The following is concerned with the derivation of \textit{a priori} estimates and the extraction of (weakly) convergent subsequences of the approximate solutions whose existence we showed in the section above. These results heavily rely on the estimates from Section \ref{sec:energy_estimates}, the properties of the orthogonal projections $\Pi_k$,  as well as standard compactness theorems.

\begin{lemma}\label{lem:a_priori_regular}
	Suppose that $(\vphi_k, \mu_k, \bm{u}_k, \theta_k,p_k)$ is a solution to the system \eqref{eq:galerkin1}-\eqref{eq:galerkin7}. Then there exists a constant $C > 0$, independent of $k \inN$, such that 
	\begin{align*}
		\norm{\vphi_k}_{L^\infty(H^1)} &+ \norm{\varrho^{1/4} \Delta \vphi}_{L^\infty(L^2)}
		+ \norm{\vphi_k}_{H^1((H^1)')}
		+ \norm{\mu_k}_{L^2(H^1)}
		+ \norm{\psi(\vphi_k)}_{L^\infty(L^1)} 
		+ \norm{\bm{u}_k}_{H^1(\bm{X})}\\ 
		&+  \norm{\theta_k}_{L^\infty(L^2)} 
		+  \norm{ \vr^{1/2} \nabla \theta_k}_{L^\infty(L^2)} 
		+ \norm{\theta_k}_{H^1(X')}
		+ \norm{p_k}_{L^2(X)}
		+ \norm{p_k}_{L^\infty(L^2)}
		\leq C. 
	\end{align*} 
\end{lemma}

\begin{proof}
	We use the same testing procedure as in Section \ref{sec:energy_estimates} and Lemma \ref{lemma:leray_schauder_regular}, i.e., 
	we test \eqref{eq:galerkin1} with $\mu_k$, \eqref{eq:galerkin2} with $\pt \vphi_k$, \eqref{eq:galerkin3} with $p_k$ and \eqref{eq:galerkin4} with $\pt \theta_k$. Moreover, we test \eqref{eq:galerkin7} with $\pt \bm{u}$ and perform the same computation as before to arrive at 
	\begin{align*}
		&\norm{m(\vphi)^{1/2} \nabla \mu_k}_{L^2}^2 + \norm{\kappa(\vphi_k)^{1/2} \nabla p_k}_{L^2}^2
		+ \int_\Omega \C_{\nu} (\vphi_k ) \E(\pt \bm{u}_k) : \E(\pt \bm{u}_k) \dx \\ 
		&\quad +\frac{d}{dt} \Big[
		\begin{aligned}[t]
			\quad &\int_\Omega\frac{\vepsilon}{2} \abs{\nabla \vphi_k}^2
			+ \frac{\varrho^{1/2} }{2} \abs{\Delta \vphi_k(t)}^2 
			+ \frac{\vr}{2} \abs{\nabla \theta_k}^2 
			+\frac{1}{\vepsilon} \psi(\vphi_k) \dx \\ 
			&+ \int_\Omega W(\vphi_k , \E(\bm{u}_k)) -  \bm{f} \cdot \bm{u}_k \dx - \int_{\Gamma_N} \bm{g} \cdot \bm{u}_k \dH
			+ \int_\Omega \frac{M(\vphi_k)}{2}  (\theta_k - \alpha(\vphi_k) \nabla \cdot \bm{u}_k)^2 \dx
			\Big]
		\end{aligned}\\ 
		&=  \begin{aligned}[t]
			&\langle R(\vphi_k, \E(\bm{u}_k), \theta_k), \mu_k \rangle + \langle S_f(\vphi_k,\E(\bm{u}_k), \theta_k), p_k \rangle\\ 
			 & + 
			 \begin{aligned}[t]
				 \int_\Omega  \alpha(\vphi_k) p_k (\nabla  \cdot  \pt \bm{u}_k) * \phi
				 &- \vr \nabla \theta \cdot \nabla (\alpha(\vphi) (\nabla \cdot \pt \bm{v}) * \phi )\\ 
				 &-M(\vphi_k)(\theta_k - \alpha(\vphi_k) \nabla \cdot \bm{u}_k) \alpha(\vphi_k)  \nabla \cdot \pt \bm{u}_k \dx. 
			 \end{aligned}
		\end{aligned}
	\end{align*}
Similar estimates as in Section \ref{sec:energy_estimates} now lead to 
	\begin{gather*}
		\begin{align*}
			&(\underline{m} - \rho_\mu C_p) \int_0^t \norm{\nabla \mu_k}^2 \dt  
			+(\underline{\kappa} - \rho_p C_p) \int_0^t \norm{ \nabla p_k}^2 \dt 
			+ (C_\nu - \delta_{\pt \bm{u}}) \int_0^t \norm{\pt \bm{u}_k}_{\bm{X}}^2 \dt  \\  
			+&\frac{\vepsilon}{2} \norm{\nabla \vphi_k(t)}_{L^2}^2 
			+ \frac{\varrho^{1/2}}{2} \norm{ \Delta \varphi_k (t)}_{L^2}^2 
			+ \frac{1}{2 \vepsilon } \norm{\psi(\vphi_k(t))}_{L^1} +(\frac{\gamma_{\psi_1}- \rho_{\psi_2}}{2 \vepsilon} -  \rho_\vphi ) \norm{\vphi_k(t)}_{L^p}^p  \\ 
			+ &(C_{\bm{u}} - \frac{\underline{M}} {2}(1- \frac{1}{\rho_\theta}) \overline{\alpha}^2   )\norm{\bm{u}_k(t)}_{\bm{X}}^2
			+ \frac{\underline{M}} {2} (1-\rho_\theta) \norm{\theta(t)}_{L^2}^2
			+ \frac{\vr}{2} \norm{\nabla \theta_k (t)}_{L^2}^2 \\ 
			\leq& \begin{aligned}[t]
				&C \int_0^t \norm{\psi(\vphi_k)}_{L^1} + \norm{\vphi_k}_{L^2}^2+ \norm{\theta_k}_{L^2}^2 + \norm{\bm{u}_k}_{H^1}^2  + 1\dt \\  
				&+ C\Big( \norm{\vphi_{0, k}}_{H^1}^2 + \varrho \norm{\Delta \vphi_{0, k}}_{L^2}^2  +  \frac{1}{\vepsilon}  \norm{\psi(\vphi_{0, k})}_{L^1} +  \norm{\bm{u}_0}_{\bm{X}}^2 + \norm{\theta_{0, k}}_{L^2}^2 + \vr \norm{\nabla \theta_{0, k}}_{L^2}^2+ 1 \Big). 
			\end{aligned} 
		\end{align*}\numberthis \label{eq:enegy_estimate_reg}
	\end{gather*}
	We recall that the eigenfunctions $\{ z_i\}$, $\{ y_i \}$ form an orthogonal basis of $L^2(\Omega)$ and that the projections satisfy
	\begin{gather*} 
		\norm{\vphi_{0, k}}_{L^2}^2 = \norm{\Pi_k^z \vphi_{0}}_{L^2}^2 \leq \norm{\vphi_0}_{L^2}^2 
		\quad \textrm{and} \quad 
		\norm{\theta_{0, k}}_{L^2}^2 = \norm{\Pi_k^y \theta_{0}}_{L^2}^2 \leq \norm{\theta_0}_{L^2}^2.
	\end{gather*}
	Moreover, due to orthogonality, it holds that $\norm{\vphi_{0, k}}_{H^1} \leq C \norm{\nabla \vphi_0}_{H^1} $. 
	Since by assumption we have $\vphi_0 \in H^2_{\bm{n}}(\Omega)$, it follows from \cite[§3]{lam06} that 
	\begin{equation*}
		\Pi_k^z \vphi_0 \rightarrow \vphi_0 \quad \textrm{in} \quad H^2_{\bm{n}}(\Omega). 
	\end{equation*}
	The embedding $ H^2_{\bm{n}}(\Omega) \hookrightarrow L^\infty(\Omega)$ and the properties of our basis imply, cf. \cite[Sec.~3.2]{MR4126782}, \cite[§3]{lam06}, 
	\begin{gather*}
		\norm{\Delta \vphi_{0, k}}_{L^2}^2 = \norm{\Delta \Pi_k^z \vphi_{0}}_{L^2}^2 = \norm{ \Pi_k^z \Delta \vphi_{0}}_{L^2}^2 \leq C \norm{\Delta \vphi_0}_{L^2}^2, \\ 
		\norm{\vphi_{0, k}}_{L^\infty} \leq C \norm{\Pi_k^z \vphi_{0}}_{H^2} \leq C \norm{\vphi_0}_{H^2}, 
	\end{gather*}
	and due to the continuity of $\psi$, we assert $\norm{\psi(\vphi_{0, k})}_{L^\infty} \leq C$
	for a constant $C > 0$ independent of $k$. Now, the strong convergence of $\vphi_{0, k}$ in $H^2_{\bm{n}}(\Omega)$ implies uniform convergence of $\psi(\vphi_{0, k})$ to $\psi(\vphi_0)$ and we conclude with the help of Lebesgue's theorem that 
	\begin{equation*}
		\psi(\vphi_{0, k}) \rightarrow \psi(\vphi_{0}) \quad \textrm{in} \quad L^1(\Omega). 
	\end{equation*}
	Finally, note that with the help of spectral theory and due to the choice of our basis $\{ y_i\}$, we can deduce $\norm{\Pi_k^y \zeta}_{X} \leq C \norm{ \zeta}_{X}$ for all $\zeta \in X(\Omega)$, $k \inN$, which entails 
	\begin{equation*}
		\norm{ \theta_{0, k}}_X =\norm{ \Pi_k^y \theta_{0}}_X   \leq C \norm{\theta_0}_{X}. 
	\end{equation*}
	Hence, without loss of generality, it holds 
	\begin{align*}
		\norm{ \vphi_k(t)}_{H^1}^2 
		&+ \norm{\varrho^{1/4} \Delta \vphi_k(t)}_{L^2}^2
		+ \norm{\psi(\vphi_k(t))}_{L^1}  
		+  \norm{\theta_k(t)}_{L^2}^2 
		+  \norm{ \vr^{1/2} \nabla \theta_k(t)}_{L^2}^2 
		+ \norm{\bm{u}_k(t)}_{\bm{X}}^2\\ 
		&+\int_0^t \norm{\nabla \mu_k}_{L^2}^2 \dt  
		+ \int_0^t \norm{ \nabla p_k}_{L^2}^2 \dt 
		+  \int_0^t \norm{\pt \bm{u}_k}_{\bm{X}}^2 \dt  \\  
		\leq& \begin{aligned}[t]
			&C \int_0^t \norm{\vphi_k}_{L^2}^2 +  \norm{\psi(\vphi_k)}_{L^1} + \norm{\theta_k}_{L^2}^2 +  \norm{\bm{u}_k}_{\bm{X}}^2  + 1\dt \\  
			&+ C\Big( \norm{\vphi_{0}}_{H^1}^2 + \varrho^{1/2} \norm{\Delta \vphi_{0}}_{L^2}^2+  \frac{1}{\vepsilon}  \norm{\psi(\vphi_{0})}_{L^1} +  \norm{\bm{u}_0}_{\bm{X}}^2 + \norm{\theta_{0}}_{L^2}^2 +  \vr \norm{  \theta_{0}}_{X}^2 + 1 \Big)
		\end{aligned} 
	\end{align*}
	and Gronwall's lemma yields for almost all $t \in (0, T)$
	\begin{align*}
		\norm{\vphi_k(t)}_{H^1}^2 
		&+\norm{\varrho^{1/4} \Delta \vphi_k(t)}_{L^2}^2
		+ \norm{\psi(\vphi_k)(t)}_{L^1} 
		+ \norm{\bm{u}(t)}_{\bm{H}^1}^2
		+  \norm{\theta_k(t)}_{L^2}^2 
		+  \norm{ \vr^{1/2} \nabla \theta_k(t)}_{L^2}^2 \\ 
		&+\norm{\nabla \mu_k}_{L^2(L^2)}^2 
		+ \norm{\nabla p_k}_{L^2(L^2)}^2
		+ \norm{\pt \bm{u}_k}_{L^2(\bm{X})}^2\\ 
		&\leq C \Big( \norm{\vphi_{0}}_{H^1}^2 +  \varrho^{1/2} \norm{\Delta \vphi_{0}}_{L^2}^2+ \frac{1}{\vepsilon}  \norm{\psi(\vphi_{0})}_{L^1} +  \norm{\bm{u}_0}_{\bm{X}}^2 + \norm{\theta_{0}}_{L^2}^2 +  \vr \norm{  \theta_{0}}_{X}^2  + 1 \Big). \numberthis \label{eq:energy_estimate_reg_1}
	\end{align*}
	With the help of the Poincaré inequality, we can deduce 
	\begin{equation}\label{eq:p_H^1}
		\norm{ p_k}_{L^2(X)}^2  
		\leq C( \norm{\vphi_{0}}_{H^1}^2 +  \varrho^{1/2} \norm{\Delta \vphi_{0}}_{L^2}^2+ \frac{1}{\vepsilon}  \norm{\psi(\vphi_{0})}_{L^1} +  \norm{\bm{u}_0}_{\bm{X}}^2  +  \vr \norm{  \theta_{0}}_{X}^2  + 1)  . 
	\end{equation}
	Similarly to \eqref{eq:mu_mean_value}, it follows that 
	\begin{align*}
		\Big| \dashint \mu_k  \Big| 
		\leq C\Big(\norm{\psi(\vphi_k)}_{L^1} + \norm{\vphi_k}_{L^2}^2 + \norm{\theta_k}_{L^2}^2 + \norm{\bm{u}_k}_{\bm{X}}^2 + 1 \Big) 
	\end{align*}
	 and with the help of the Poincaré--Wirtinger inequality, we get 
	 \begin{equation*}
	 	\norm{\mu_k}_{L^2(L^2)}^2 \leq C( \Big\| {\mu_k - \dashint \mu_k} \Big\|_{L^2(L^2)}^2 + \Big\|{\dashint \mu_k}\Big\|_{L^2(L^2)}^2)
	 	\leq C(\norm{\nabla \mu_k}_{L^2(L^2)}^2 + \Big\|{\dashint \mu_k}\Big\|_{L^2(L^2)}^2). 
	 \end{equation*}
	 Using the $L^\infty$-estimates in \eqref{eq:energy_estimate_reg_1}, we infer 
	 \begin{align*}
	 	 \norm{\dashint \mu_k}_{L^2(L^2)}^2
	 	 &\leq  \int_0^T \Big(C(\norm{\psi(\vphi_k)(t)}_{L^1} + \norm{\vphi_k(t)}_{L^2}^2 + \norm{\theta_k}_{L^2}^2 + \norm{\bm{u}_k(t)}_{\bm{X}}^2  + 1)\Big)^2 \dt \\ 
	 	 &\leq \int_0^T C \Big(  \norm{\vphi_{0}}_{H^1}^4 + \varrho \norm{\Delta \vphi_{0}}_{L^2}^4+  \frac{1}{\vepsilon}  \norm{\psi(\vphi_{0})}_{L^1}^2 +  \norm{\bm{u}_0}_{\bm{X}}^2 + \norm{\theta_{0}}_{L^2}^4 +   \vr^2 \norm{  \theta_{0}}_{X}^4   + 1 \Big) \dt \\ 
	 	 &\leq C\Big(  \norm{\vphi_{0}}_{H^1}^4 + \varrho \norm{\Delta \vphi_{0}}_{L^2}^4+  \frac{1}{\vepsilon}  \norm{\psi(\vphi_{0})}_{L^1}^2 + \norm{\theta_{0}}_{L^2}^4 +  \norm{\bm{u}_0}_{\bm{X}}^2 +  \vr^2 \norm{  \theta_{0}}_{X}^4  + 1 \Big) 
	 \end{align*}
	 and hence 
	 \begin{equation}\label{eq:mu_H^1}
	 	\norm{\mu_k}_{L^2(H^1)}^2
	 	 \leq 
	 	 C( \norm{\vphi_{0}}_{H^1}^4 + \varrho \norm{\Delta \vphi_{0}}_{L^2}^4+  \frac{1}{\vepsilon^2}  \norm{\psi(\vphi_{0})}_{L^1}^2 + \norm{\theta_{0}}_{L^2}^4+  \vr^2 \norm{  \theta_{0}}_{X}^4  + \norm{\bm{u}_0}_{\bm{X}}^2  + 1). 
	 \end{equation}
	 \par
	 Finally, it remains to find estimates for the time derivatives $\pt \vphi_k$ and $\pt \theta_k$. Since $\{ z_i \}$ is an orthogonal system in $L^2(\Omega)$, we derive from \eqref{eq:galerkin1} that for all $\zeta \in H^1(\Omega)$
	 \begin{equation*}
	 	\langle \pt \vphi_k, \zeta \rangle 
	 	= 	\langle \pt \vphi_k,   \Pi_k^z \zeta \rangle
	 	= - \langle m(\vphi_k)\nabla  \mu_k,  \nabla \Pi_k^z \zeta \rangle + \langle R(\vphi_k, \E(\bm{u}_k), \theta_k),  \Pi_k^z \zeta \rangle
	 \end{equation*}
	and exploiting \ref{A:m} along with $\norm{\Pi_k^z \zeta}_{H^1} \leq C \norm{ \zeta}_{H^1}$ for all $\zeta \in H^1(\Omega)$, we find
	\begin{equation}\label{eq:phi_H^1'}
		\norm{\pt \vphi_k}_{(H^1)'} \leq \overline{m} \norm{\mu_k}_{H^1} + C. 
	\end{equation}
	Recalling the orthogonality of $\{ y_i \}$ in $L^2(\Omega)$, the identity \eqref{eq:galerkin3} leads to 
	\begin{equation*}
			\langle \pt  \theta_k, \zeta \rangle 
		= 	\langle \pt \theta_k, \Pi_k^y \zeta \rangle
		= - \langle \kappa(\vphi_k)\nabla  p_k,  \Pi_k^y \zeta \rangle + \langle S_f(\vphi_k, \E(\bm{u}_k), \theta_k),  \Pi_k^y \zeta \rangle
	\end{equation*}
	for all $\zeta \in X(\Omega)$. With the help of spectral theory and due to the choice of our basis $\{ y_i\}$, we deduce $\norm{\Pi_k^y \zeta}_{X} \leq C \norm{ \zeta}_{X}$ for all $\zeta \in X(\Omega)$, $k \inN$ and therefore 
	\begin{equation}\label{eq:theta_H^1'}
			\norm{\pt \theta_k}_{X'} \leq C(\overline{\kappa} \norm{p_k}_{X} + \norm{\vphi_k}_{L^2} + \norm{\theta_k}_{L^2} + \norm{\bm{u}}_{\bm{X}} + 1). 
	\end{equation}
	Together with the estimates from above, we arrive at 
	\begin{equation*}
		\norm{\vphi_k}_{H^1((H^1)')}
		+ \norm{\theta_k}_{H^1(X')}
		\leq C. 
	\end{equation*}
\end{proof}

\begin{lemma}\label{lem:elliptic_reg_theta}
	The volumetric fluid content satisfies the regularity $\theta_k \in L^2(0, T; H^{1+\gamma}_{\Gamma_D}(\Omega))$ with 
	\begin{equation*}
		\norm{\theta_k}_{L^2(H^{1+\gamma}_{\Gamma_D})} \leq C
	\end{equation*}
	for some $C(\vr) > 0$, $\gamma > 0$ independent of $k \inN$. 
\end{lemma}

\begin{proof}
	Elliptic regularity theory, cf. \cite[Thm.~1, Cor.~1]{haller2019higher}, tells us that there exists some $\gamma > 0$ such that the operator 
	\begin{equation*}
		 -\vr \Delta : H^{1+ \gamma}_{\Gamma_D}(\Omega) \rightarrow H^{\gamma - 1}_{\Gamma_D}(\Omega), 
		 \quad 
		v \mapsto  \left( w \mapsto \int_\Omega \vr \nabla v \cdot \nabla \overline{w} \dx \right) 
	\end{equation*}
	is a topological isomorphism between $H^{1+ \gamma}_{\Gamma_D}(\Omega)$ and $H^{\gamma- 1}_{\Gamma_D}(\Omega)$. Since $L^2(\Omega) \hookrightarrow H^{\gamma- 1}_{\Gamma_D}(\Omega)$, cf. \cite[Rem.~2]{haller2019higher}, this entails that for any $f \in L^2(\Omega)$ the weak solution $v$ to the mixed boundary-value problem with variational formulation 
	\begin{equation*}
		 \int_\Omega \vr \nabla v \cdot \nabla \overline{w} \dx = \int_\Omega f \overline{w} \dx \quad \textrm{for all} \quad w \in H^{1- \gamma}_{\Gamma_D}(\Omega)
	\end{equation*} 
	is in the space $H^{1+ \gamma}_{\Gamma_D}(\Omega)$ and that there exists some $C > 0$ independent of $f \in L^2(\Omega)$ such that 
	\begin{equation}\label{eq:elliptic_H_eps}
		\norm{v}_{H^{1+ \gamma}_{\Gamma_D}} \leq C \norm{f}_{L^2}. 
	\end{equation}
	For $\theta_k^* = (-\vr \Delta)^{-1}\big( p_k - \Pi_k^y(M(\vphi_k) (\theta_k - \alpha(\vphi_k) \nabla \cdot \bm{u}_k)) \big)$ it therefore holds that 
	\begin{equation*}
		 \int_\Omega \vr \nabla v \cdot \nabla \overline{w} \dx = \int_\Omega  p_k \overline{w} - \Pi_k^y(M(\vphi_k) (\theta_k - \alpha(\vphi_k) \nabla \cdot \bm{u}_k)) \overline{w} \dx \quad \textrm{for all} \quad w \in H^{1- \gamma}_{\Gamma_D}(\Omega)
	\end{equation*}
	and since $X(\Omega) \subset H^{1- \gamma}_{\Gamma_D}(\Omega)$, \eqref{eq:p_identity_new} implies that $\theta_k$ and $\theta_k^*$ both solve the mixed boundary value problem 
	\begin{align*}
		\begin{cases}
				- \vr \Delta v =  p_k - \Pi_k^y(M(\vphi_k) (\theta_k - \alpha(\vphi_k) \nabla \cdot \bm{u}_k)) \quad &\textrm{in} \quad \Omega,\\
			v = 0 \quad &\textrm{on } \quad \Gamma_D, \\ 
			\nabla v \cdot \bm{n} = 0 \quad &\textrm{on} \quad \Gamma_N.
		\end{cases}
	\end{align*} 
	As the solution is unique, we have $\theta_k = \theta_K^*$ and \eqref{eq:elliptic_H_eps} along with Lemma~\ref{lem:a_priori_regular} yields 
	\begin{align*}
		\norm{\theta_k}_{L^2(H^{1+ \gamma}_{\Gamma_D})}^2 
		&\leq C \norm{ p_k - \Pi_k^y(M(\vphi_k) (\theta_k - \alpha(\vphi_k) \nabla \cdot \bm{u}_k)) }_{L^2(L^2)}^2\\ 
		&\leq C \left(  \norm{p_k}_{L^2(L^2)}^2 + \norm{\theta_k}_{L^2(L^2)}^2 + \norm{\bm{u}}_{L^2(\bm{X})}^2 \right) 
		\leq C . 
	\end{align*}
\end{proof}
Moreover, elliptic regularity theory yields the existence of some $C > 0$ such that 
\begin{equation*}
	\norm{\vphi_k}_{H^2} \leq C (\norm{\Delta \vphi_k}_{L^2} + \norm{\vphi_k}_{H^1} ) 
\end{equation*}
and together with Lemma \ref{lem:a_priori_regular}, we find 
\begin{equation}\label{est:phi_h^2}
	\norm{\vphi_k}_{L^\infty(H^2)} \leq C (\norm{\Delta \vphi_k}_{L^\infty(L^2)} + \norm{\vphi_k}_{L^\infty(H^1)}) \leq C(\varrho).   
\end{equation}
With these estimates, we can invoke the Aubin--Lions--Simon theorem to obtain
\begin{align}
	\vphi_k &\rightarrow  \vphi \quad  \textrm{in} \quad  C^0([0, T]; W^{1, r}(\Omega)),  \label{conv:phi_strong_regular}  \\
	\theta_k &\rightarrow \theta \quad \, \textrm{in} \quad  C^0([0, T]; L^r(\Omega)) \cap L^2(0, T; X(\Omega)),  \label{conv:theta_strong_regular}
\end{align}
where $r < \infty$ arbitrarily if $n \leq 2$ and $r < 6$ if $n = 3$. By passing to an appropriate subsequence, we can also assume that $\vphi_k$ converges to $\vphi$ pointwise almost everywhere. 
Hence, we can deduce the existence of functions $(\vphi, \mu, \theta, p)$ such that, along a not relabeled subsequence, 
\begin{equation*}\label{conv:comapactness_regular}
	\begin{aligned}
		\vphi_k \rightharpoonup \vphi \quad & \textrm{in} \quad  L^2(0, T; H^2(\Omega)) \cap H^1(0, T; H^1(\Omega)'),  \\ 
 		\mu_k \rightharpoonup \mu \quad & \textrm{in} \quad  L^2(0, T; H^1(\Omega)), \\ 
		\bm{u}_k  \rightharpoonup \bm{u} \quad & \textrm{in} \quad  H^1(0, T; \bm{H}^1_{\Gamma_D}(\Omega)),  \\ 
		\theta_k \rightharpoonup \theta \quad & \textrm{in} \quad   L^2(0, T; X(\Omega)) \cap H^1(0, T; X'(\Omega)),  \\ 
		p_k \rightharpoonup p \quad & \textrm{in} \quad  L^2(0, T; X(\Omega)). 
	\end{aligned} \numberthis 
\end{equation*}

Pointwise convergence of $\vphi_k$ almost everywhere along with the continuity of $\psi'$ further implies
\begin{equation*}
	\psi'(\vphi_k) \rightarrow \psi'(\vphi) \quad \textrm{pointwise a.e. in } \Omega_T. 
\end{equation*}
To prove convergence for this term in $L^1(\OT)$, we use the decomposition $\psi = \psi_1 + \psi_2$ from \ref{A:psi} and treat the two cases separately. First of all, we note that due to \ref{A:psi_1}, the family $\{ \psi'_1(\vphi_k)\}$ is uniformly integrable over $\OT$. Indeed, due to Lemma \ref{lem:a_priori_regular}, it holds for all subsets $E \subset \OT$ that 
\begin{align*}
	\int_E | \psi'_1(\vphi_k)| \dtx 
	\leq  \rho_{\psi_1} \int_E |\psi_1(\vphi_k) | \dtx + C_{\rho_{\psi_1}} |E|
	 &\leq  \rho_{\psi_1}\int_{\OT} |\psi_1(\vphi_k) | \dtx + C_{\rho_{\psi_1}} |E| \\
	 &\leq \rho_{\psi_1} C  + C_{\rho_{\psi_1}} |E|, 
\end{align*}
which converges to $\rho_{\psi_1} C$ uniformly in $k \inN$ as $|E| \rightarrow 0$. Since $\rho_{\psi_1}$ can be chosen arbitrarily small and the constant $C$ does not depend on $k \inN$ or the set $E \subset \Omega_T$, we obtain uniform integrability. 
After applying Vitali's convergence theorem, we find 
\begin{equation*}
		\psi_1'(\vphi_k) \rightarrow \psi'_1(\vphi) \quad \textrm{in} \quad L^1(\Omega_T). 
\end{equation*}
Moreover, the growth condition \ref{A:psi_2}, the pointwise convergence and Lebesgue's generalized convergence theorem, cf. \cite[Sec. 3.25]{alt2013lineare}, yield 
\begin{equation*}
		\psi_2'(\vphi_k) \rightarrow \psi'_2(\vphi) \quad \textrm{in} \quad L^2(\Omega_T)
\end{equation*}
so that we can conclude  
\begin{equation}\label{conv:psi_reg}
		\psi'(\vphi_k) \rightarrow \psi'(\vphi) \quad \textrm{in} \quad L^1(\Omega_T). 
\end{equation}

\subsection{Additional compactness results}
Taking advantage of \eqref{eq:weak5}, we can show a uniform estimate for the differences $\tau_h p_k - p_k$ such that an application of a version of the Aubin--Lions--Simon theorem yields strong convergence in the space $L^2(0, T; L^2(\Omega))$. \par 
Testing \eqref{eq:weak3} with both $\bm{u}_k - \bm{u}$ and the time derivatives $\pt \bm{u}_k - \pt \bm{u}$ further yields an estimate for the difference $\norm{\bm{u}_{k}(t) - \bm{u}(t)}_{\bm{X}}$ such that an application of Gronwall's lemma  allows us to derive pointwise a.e. convergence  in $\bm{X}$. Here, we crucially rely on the strong convergence of $p_k$, which was shown before. Finally, an $L^p$-$L^q$ compactness property delivers the strong convergence of $\bm{u}_k$ in $L^2(0, T; \bm{X})$. 

\begin{lemma}\label{lem:p_strong}
	There exists some subsequence of ${k} \rightarrow \infty$ such that, along this not relabeled subsequence, 
	\begin{equation*}
		p_{k} \rightarrow p \quad \textrm{in} \quad  L^2(0, T; L^2(\Omega)). 
	\end{equation*}
\end{lemma}
\begin{proof}
	 Due to the \textit{a priori} estimates we already know a uniform bound for all  $p_k$, $k \inN$ in $L^2(0, T; H^1(\Omega))$. For any suitable function $f$, let $\tau_h f(t) \coloneqq f(t+h)$ and define 
	 \begin{equation*}
	 	W^{3, 2}_{0}(\Omega) \coloneqq \overline{C^\infty_c(\Omega)}^{ \norm{\cdot }_{W^{3, 2}}},  
	 \end{equation*}
	 with $W^{-3, 2}(\Omega)$ denoting the dual space $(W^{3, 2}_{0}(\Omega))'$. 
	If we can show that 
	\begin{equation}\label{eq:translation_estimate}
		\norm{\tau_h {p_{k}} - {p_{k}}}_{L^1(0, T-h; W^{-3, 2}(\Omega))} \dt \rightarrow 0 
		\quad \textrm{as } h \rightarrow 0 \quad \textrm{uniformly in } k, 
	\end{equation}
	for some subsequence of ${k} \rightarrow \infty$, then \cite[Thm.~5]{simon1986compact} already yields the assertion, since \begin{equation*}
		X(\Omega) \xhookrightarrow{cpt} L^2(\Omega) \xhookrightarrow{c} W^{-3, 2}_{0}(\Omega). 
	\end{equation*} 
	Recalling the isomorphism $L^1(0, T-h; W^{-3, 2}(\Omega)) \cong (L^\infty(0, T-h; W^{3, 2}_{0}(\Omega)))'$, it suffices to show that 
	\begin{equation*}
		\sup_{\norm{\xi}_{L^\infty(0, T-h; W^{3, 2}_{0}(\Omega))} = 1} \left|\int_{\Omega_{T-h}} (\tau_h p_k - p_k) \xi \dtx\right| 
		 \rightarrow 0 
		\quad \textrm{as } h \rightarrow 0 \quad \textrm{uniformly in } k. 
	\end{equation*}
	With the help of \eqref{eq:p_identity_new} and the boundary conditions in the space $W^{3, 2}_{0}(\Omega)$, we obtain 
	\begin{align*}
		\int_{\Omega_{T-h}}p \xi \dtx &= 
		\int_{\Omega_{T-h}} \vr \nabla \theta \cdot \nabla  \xi + M(\vphi_k) (\theta_k - \alpha(\vphi_k) \nabla \cdot \bm{u}_k) \,  \Pi^y_k\xi \dtx\\ 
		&= \int_{\Omega_{T-h}} -\vr \theta\, \Delta \xi +M(\vphi_k) (\theta_k - \alpha(\vphi_k) \nabla \cdot \bm{u}_k) \,  \Pi^y_k \xi \dtx
	\end{align*}
	for all $\xi \in L^\infty(0, T-h; W^{3, 2}_{0}(\Omega))$. 
	Setting $\Delta_h f \coloneqq \tau_h f - f$, it follows that 
	\begin{align*}
		 &\int_{\Omega_{T-h}} \Delta_h p_k\, \xi \dtx \\ 
		&\quad = \begin{aligned}[t]
		 	&\int_{\Omega_{T-h}} - \vr \Delta_h \theta_k \, \Delta \xi \dtx 
		 	+\int_{\Omega_{T-h}}  \tau_h M(\vphi_k)\, \Delta_h \theta_k   \Pi^y_k \xi \dtx
		 	+ \int_{\Omega_{T-h}}  \theta_k \Delta_h M(\vphi_k)   \Pi^y_k \xi \dtx \\ 
		 	&+  \int_{\Omega_{T-h}}   \tau_h(M(\vphi_k) \alpha(\vphi_k)) \Delta_h(\nabla \cdot \bm{u}_k)    \Pi^y_k \xi \dtx
		 	+  \int_{\Omega_{T-h}}   \nabla \cdot \bm{u}_k \Delta_h(M(\vphi_k) \alpha(\vphi_k))   \Pi^y_k \xi \dtx
		 \end{aligned} \\ 
		 &\quad = I + II + III + IV + V, 
	\end{align*}
	which we need to treat separately. 
	\par  
	\medskip
	\textit{Ad $(IV)$}: Due to the well-known stability property $\norm{\Pi_k^y (\cdot) }_{L^2} \leq \norm{ \cdot }_{L^2} $ of the projection operator and since $M, \alpha$ are bounded functions, we find 
	\begin{align*}
		&\left|  \int_{\Omega_{T-h}}   \tau_h(M(\vphi_k) \alpha(\vphi_k)) \Delta_h(\nabla \cdot \bm{u}_k)    \Pi^y_k \xi \dtx \right|\\ 
		& \quad \leq \norm{ \tau_h(M(\vphi_k) \alpha(\vphi_k)) \Delta_h(\nabla \cdot \bm{u}_k)}_{L^1(0, T-h; L^2)} \norm{ \Pi^y_k \xi }_{L^\infty(0, T-h; L^2)}\\ 
		&\quad \leq  \norm{\Delta_h(\nabla \cdot \bm{u}_k)}_{L^1(0, T-h; L^2)} \norm{ \xi }_{L^\infty(0, T-h; W^{3, 2}_{0})}. 
		 \numberthis \label{eq:conv_p_1}
	\end{align*}
	\par 
	\medskip 
	\textit{Ad $(III)$ \& $(V)$}:
	Recalling the continuous embedding $X(\Omega) \hookrightarrow L^6(\Omega)$, we deduce with the help of Hölder's inequality and a duality argument that $L^{6/5}(\Omega) \hookrightarrow X'(\Omega). $
	As mentioned before, $\{ y_i \}$ is a basis of $X(\Omega)$ satisfying the stability property $\norm{\Pi_k^y (\cdot) }_{X} \leq C \norm{ \cdot }_{X} $ and $W^{3, 2}_{0} (\Omega) \xhookrightarrow{c} X(\Omega)$. Thus, it holds that 
	\begin{align*}
		\left| \int_{\Omega_{T-h}}  \theta_k \Delta_h M(\vphi_k)   \Pi^y_k \xi \dtx  \right|
		&\leq  \norm{\theta_k \Delta_h M(\vphi_k)}_{L^1(0, T-h; X')} \norm{ \Pi_y^k \xi}_{L^\infty(0, T-h; X)}\\  
		&\leq  C \norm{\theta_k \Delta_h M(\vphi_k)}_{L^1(0, T-h; L^{6/5})}\, \norm{ \xi}_{L^\infty(0, T-h; X)}\\  
		& \leq C \norm{\theta_k}_{L^\infty(0, T-h; L^2)}\, \norm{\Delta_h M(\vphi_k)}_{L^1(0, T-h; L^3)}  \, \norm{  \xi}_{L^\infty(0, T-h; W^{3, 2}_{ 0})}
		 \numberthis \label{eq:conv_p_2}
	\end{align*}
	Similarly, we find for $(V)$
	\begin{align*}
		&\left| \int_{\Omega_{T-h}}   \nabla \cdot \bm{u}_k \Delta_h(M(\vphi_k) \alpha(\vphi_k))   \Pi^y_k \xi \dtx \right| 
		\leq  \norm{ \nabla \cdot \bm{u}_k \Delta_h(M(\vphi_k) \alpha(\vphi_k)) }_{L^1(0, T-h; X')} \norm{ \Pi_y^k \xi}_{L^\infty(0, T-h; X)}\\  
		&\quad \leq C \norm{ \nabla \cdot \bm{u}_k \Delta_h(M(\vphi_k) \alpha(\vphi_k))}_{L^1(0, T-h; L^{6/5})}\, \norm{ \xi}_{L^\infty(0, T-h; X)}\\  
		&\quad  \leq C \norm{ \nabla \cdot \bm{u}_k}_{L^\infty(0, T-h; L^2)}\, \norm{\Delta_h(M(\vphi_k) \alpha(\vphi_k))}_{L^1(0, T-h; L^3)}  \, \norm{  \xi}_{L^\infty(0, T-h; W^{3, 2}_{ 0})}
		\numberthis \label{eq:conv_p_3}
	\end{align*}
	\par
	\medskip 
	\textit{Ad $(II)$:} Due to the uniform bound on $\vphi_k$ in $L^\infty(0, T; H^2_{\bm{n}}(\Omega))$, see \eqref{est:phi_h^2}, and the continuous embedding $H^2(\Omega) \hookrightarrow W^{1, 6}(\Omega)$, we deduce 
	\begin{align*}
		\norm{\tau_h  M(\vphi_k) \Pi_k^y\xi}_{X} 
		&\leq  \norm{\tau_h  M(\vphi_k) \Pi_k^y \xi}_{L^2} + \norm{\tau_h  (M'(\vphi_k)\nabla \vphi_k) \Pi_k^y \xi}_{\bm{L}^2} + \norm{\tau_h  M(\vphi_k) \nabla  \Pi_k^y \xi}_{\bm{L}^2}\\ 
		& \leq \overline{M} \norm{\xi}_{\bm{L}^2} + \overline{M} \norm{\tau_h  \nabla \vphi_k}_{\bm{L}^4} \norm{\Pi_k^y \xi}_{L^4} + \overline{M} \norm{ \Pi_k^y \xi}_{X}\\ 
		& \leq  \overline{M} \norm{\xi}_{L^2} + \overline{M} \norm{\tau_h  \vphi_k}_{H^2_{\bm{n}}} \norm{\xi}_{X} + C \overline{M} \norm{\xi}_{X}
		\leq C  \norm{\xi}_{W^{3, 2}_{0}} 
	\end{align*}
	for all $\xi \in W^{3, 2}_{0}(\Omega)$. Exploiting this estimate, we infer  
	\begin{align*}
		\left| \int_{\Omega_{T-h}}  \tau_h M(\vphi_k)\, \Delta_h \theta_k   \Pi^y_k \xi \dtx \right|
		&\leq \norm{\Delta_h \theta_k }_{L^1(0, T-h; X')} \norm{ \tau_h M(\vphi_k)\,   \Pi^y_k \xi}_{L^\infty(0, T-h; X)}\\ 
		& \leq  C \norm{\Delta_h \theta_k }_{L^1(0, T-h; X')} \norm{\xi}_{L^\infty(0, T-h; W^{3, 2}_{0})}. 
		  \numberthis \label{eq:conv_p_4}
	\end{align*}
	\par 
	\textit{Ad $(I)$:} Since $\Delta \xi \in X(\Omega)$ for all $\xi \in W^{2, 3}_{0}(\Omega)$, it holds 
	\begin{align*} 
		\left| \int_{\Omega_{T-h}} - \vr \Delta_h \theta_k \, \Delta \xi \dtx  \right| 
		&\leq \vr \norm{\Delta_h \theta_k}_{L^1(0, T-h; X')} \norm{ \Delta \xi}_{L^\infty(0, T-h; X)}\\ 
		&\leq  C \norm{\Delta_h \theta_k}_{L^1(0, T-h; X')} \norm{ \xi}_{L^\infty(0, T-h; W^{3, 2}_{0})}. 
		\numberthis \label{eq:translations_I}
	\end{align*}
	Here, we want to emphasize that this estimate is independent of all $\vr < 1$.
	\par
	Along with \eqref{eq:conv_p_1}, \eqref{eq:conv_p_2}, \eqref{eq:conv_p_3},  \eqref{eq:conv_p_4} and the Lipschitz continuity of $\alpha$ and $M$, see assumptions \ref{A:phase_coefficients}, as well as the uniform bounds on $\bm{u}_k$ in $L^\infty(0, T; H^1_{\Gamma_D}(\Omega))$ and $\theta_k$ in $L^\infty(0, T; L^2(\Omega))$, we conclude 
	\begin{align*}
			&\norm{\tau_h {p_{k}} - {p_{k}}}_{L^1(0, T-h; W^{-3, 2}_{00}(\Omega))} = \sup_{\substack{ \xi \in {L^\infty(0, T-h; W^{3, 2}_{0}(\Omega))} \\ \norm{\xi} = 1 }} \left|\int_{\Omega_{T-h}} (\tau_h p_k - p_k) \xi \dtx\right| \\ 
			&\quad \leq\sup_{\substack{ \xi \in {L^\infty(0, T-h; W^{3, 2}_{0}(\Omega))} \\ \norm{\xi} = 1 }} \begin{aligned}[t]
				C\Big( &\norm{\Delta_h \theta_k}_{L^1(0, T-h; X')} 
				+  \norm{\Delta_h M(\vphi_k)}_{L^1(0, T-h; L^3)} 
				+ \norm{\Delta_h(\nabla \cdot \bm{u}_k)}_{L^1(0, T-h; L^2)} \\ 
				&+ \norm{\Delta_h(M(\vphi_k) \alpha(\vphi_k))}_{L^1(0, T-h; L^3)}
				\Big)  \norm{ \xi}_{L^\infty(0, T-h; W^{3, 2}_{0})}
			\end{aligned}\\ 
			& \quad \leq C \Big( \norm{\Delta_h \theta_k}_{L^1(0, T-h; X')}  
			+ \norm{\Delta_h(\nabla \cdot \bm{u}_k)}_{L^1(0, T-h; L^2)} 
			+\norm{\Delta_h \vphi_k}_{L^1(0, T-h; L^3)} 
			\Big) 
	\end{align*}
	Recalling the estimates
	\begin{gather*}
		\norm{\Delta_h \theta_k}_{L^1(0, T-h; X') } 
		= \norm{\tau_h \theta_k - \theta_k}_{L^1(0, T-h; X')} 
		\leq hC \norm{\theta_k}_ {H^1(0, T; X')}, \\ 
		\norm{\Delta_h (\nabla \cdot \bm{u}_k) }_{L^1(0, T-h; L^2)} 
		= \norm{\nabla \cdot \tau_h \bm{u}_k - \nabla \cdot \bm{u}_k}_{L^1(0, T-h ; L^2)} 
		\leq h C \norm{\nabla \cdot \bm{u}_k}_ {H^1(0, T; L^2)}
	\end{gather*}
	from \cite[Lem.~4]{simon1986compact}, together with the \textit{a priori} estimate $\norm{\theta_k}_ {H^1(0, T; X')} + \norm{\bm{u}_k}_{H^1(0, T; H^1)} \leq C$,
	we obtain 
	\begin{align}\label{eq:translation_estimate_u_theta}
		\sup_{k \inN}  \Big( \norm{\Delta_h \theta_k}_{L^1(0, T-h; X')} 
		+  \norm{ \Delta_h (\nabla \cdot  \bm{u}_k)}_{L^1(0, T-h ; L^2)}  \Big)  \rightarrow 0
		\quad \textrm{as } h \rightarrow 0. 
	\end{align}
	Moreover, due to the compact embedding 
	\begin{equation*}
		L^\infty(0, T; H^1(\Omega)) \cap H^1(0, T; (H^1(\Omega)')) \xhookrightarrow{cpt} C^0([0, T]; L^3(\Omega)) 
	\end{equation*}
	it holds that, along a subsequence, 
	\begin{equation*}
		\vphi_{k} \rightarrow \vphi \quad \textrm{in } C^0([0, T]; L^3(\Omega)) \quad \textrm{as } k \rightarrow \infty. 
	\end{equation*}
	The Azelá-Ascoli theorem for Banach space valued functions asserts that this implies uniform equi-continuity, i.e., 
	\begin{equation*}
		\sup_{k \inN } \Big( \norm{\vphi_{k}(t) -  \vphi_{k}(s)}_{L^3} \Big) \rightarrow 0 \quad \textrm{as } \abs{t-s} \rightarrow 0.  
	\end{equation*}
	Thus, we conclude 
	\begin{equation*}
		\norm{ \Delta_h \vphi_{k}}_{L^1(0, T-h; L^{3})} \rightarrow 0 , 
	\end{equation*}
	which, along with \eqref{eq:translation_estimate_u_theta}, implies \eqref{eq:translation_estimate}. 
\end{proof}

With this convergence result, we finally turn to show the strong convergence of $\bm{u}_{k}$. 

\begin{lemma}\label{lem:convergence_u_reg}
	There exists a subsequence of ${k} \rightarrow \infty$ such that, along this not relabeled subsequence, 
	\begin{equation*}
		\bm{u}_{k} \rightarrow \bm{u} \quad \textrm{in} \quad  L^2(0, T; \bm{X}(\Omega)). 
	\end{equation*}
\end{lemma}
\begin{proof}

	Recall that the weak formulation for the approximate problems \eqref{eq:galerkin7} asserts that for all test functions $\bm{\eta} \in L^2(0, T;\bm{X}(\Omega))$ it holds 
	\begin{align*}
			\int_{\Omega_T} \C_\nu (\vphi_k ) \E(\pt \bm{u}_k)&  :  \E(\bm{\eta}) + \WE(\vphi_k , \E(\bm{u}_k) ) :  \E(\bm{\eta})  - \alpha(\vphi_k) p_k (\nabla \cdot \bm{\eta})* \phi \dtx \\ 
		&= \int_{\Omega_T} \bm{f} \cdot \bm{\eta} \dtx + \int_0^T\int_{\Gamma_D} \bm{g}\cdot \bm{\eta} \dH \dt 
		- \vr \int_{\OT} \nabla \theta_k \cdot \nabla (\alpha(\vphi_k) (\nabla \cdot \bm{\eta}) * \phi ) \dtx. 
	\end{align*}
	Due to our \textit{a priori} estimates, we can pass to the limit and obtain the analogous equation for the weak limit $\bm{u}$, cf. Section \ref{sec:limit_reg}. Here we need to use that $ \WE(\vphi_k, \E(\bm{u}_k)) =  \C(\vphi_k  ) (\E(\bm{u}_k) - \Tau(\vphi_k ))$. Testing with the difference $\bm{u}_k \chi_{[0, t]} -  \bm{u}\chi_{[0, t]}  \in L^2(0, T; \bm{X}(\Omega))$ and subtracting the equations leads to  
	\begin{align*}
		\int_{\Ot} &[\C_\nu (\vphi_k ) \E(\partial_t \bm{u}_k) - \C_\nu (\vphi) \E(\partial_t \bm{u})] :  \E(\bm{u}_k - \bm{u})  
		+ [\C(\vphi_k ) \E( \bm{u}_k) - \C(\vphi) \E( \bm{u})] :  \E(\bm{u}_k - \bm{u})  \\ 
		&+ [\C (\vphi_k) \Tau(\vphi_k) - \C(\vphi) \Tau(\vphi )] :  \E(\bm{u}_k - \bm{u}) 
		-  (\alpha(\vphi_k) p_k - \alpha(\vphi) p)\nabla \cdot(\bm{u}_k - \bm{u}) * \phi \dtx \\ 
		& =-  \vr \int_{\Omega_t} \nabla \theta_k \cdot \nabla (\alpha(\vphi_k) (\nabla \cdot(\bm{u}_k - \bm{u})) * \phi )  - \nabla \theta \cdot \nabla (\alpha(\vphi) (\nabla \cdot(\bm{u}_k - \bm{u})) * \phi )  \dtx
	\end{align*}
	For the first term, we calculate 
	\begin{align*}
		\int_{\Ot} &[\C_\nu (\vphi_k ) \E(\partial_t \bm{u}_k) - \C_\nu (\vphi) \E(\partial_t \bm{u})] :  \E(\bm{u}_k - \bm{u}) \dtx \\ 
		&=\begin{aligned}[t]
			\int_0^t \frac{d}{dt} \frac{1}{2} \norm{\E(\bm{u}_k - \bm{u})}_{L^2}^2 \dt 
			&+ \int_{\Ot}  \partial_t  \E(\bm{u}_k - \bm{u}) : (\C_\nu(\vphi_k ) \E(\bm{u}_k - \bm{u})   - \E(\bm{u}_k - \bm{u}))\dtx \\ 
			&- \int_{\Ot}   (\C_\nu(\vphi) -  \C_\nu(\vphi_k )) \E(\pt \bm{u}) : \E(\bm{u}_k - \bm{u})\dtx. 
		\end{aligned}
	\end{align*}
	And for the second we have 
	\begin{align*}
		&\int_{\Ot}  [\C(\vphi_k ) \E( \bm{u}_k) - \C(\vphi) \E( \bm{u})] :  \E(\bm{u}_k - \bm{u}) \dtx \\ 
		& \quad =  	\int_{\Ot} \C(\vphi_k ) \E( \bm{u}_k - \bm{u}) :  \E(\bm{u}_k - \bm{u}) 
		- (\C(\vphi) - \C(\vphi_k ))  \E( \bm{u}) :  \E(\bm{u}_k - \bm{u})\dtx. 
	\end{align*}
	\par 
	We leave the remaining terms for now and test the weak formulations with $\pt\bm{u}_k \chi_{[0, t]} - \pt \bm{u}\chi_{[0, t]}$ to obtain
	\begin{align*}
		\int_{\Ot} &[\C_\nu (\vphi_k ) \E(\partial_t \bm{u}_k) - \C_\nu (\vphi) \E(\partial_t \bm{u})] :   \pt \E(\bm{u}_k - \bm{u})  
		+ [\C(\vphi_k) \E( \bm{u}_k) - \C(\vphi) \E( \bm{u})] : \pt  \E(\bm{u}_k - \bm{u})  \\ 
		&+ [\C (\vphi_k) \Tau(\vphi_k) - \C(\vphi) \Tau(\vphi)] :  \pt \E(\bm{u}_k - \bm{u}) \dtx
		-  (\alpha(\vphi_k) p_k - \alpha(\vphi) p) \nabla \cdot(\pt \bm{u}_k - \pt \bm{u}) * \phi \dtx  \\ 
		& =- \vr \int_{\Omega_t} \nabla \theta_k \cdot \nabla (\alpha(\vphi_k) (\nabla \cdot(\pt \bm{u}_k - \pt \bm{u})) * \phi )  - \nabla \theta \cdot \nabla (\alpha(\vphi) (\nabla \cdot(\pt \bm{u}_k - \pt \bm{u})) * \phi )  \dtx
	\end{align*}
	Again, we want to rewrite the first two terms as 
	\begin{align*}
		\int_{\Ot} &[\C_\nu (\vphi_k ) \E(\partial_t \bm{u}_k) - \C_\nu (\vphi) \E(\partial_t \bm{u})] : \pt \E(\bm{u}_k - \bm{u}) \dtx \\ 
		&= \int_{\Ot} \C_\nu(\vphi_k )   \partial_t  \E(\bm{u}_k - \bm{u}) :\pt  \E(\bm{u}_k - \bm{u})  
		- (\C_\nu(\vphi) -  \C_\nu(\vphi_k )) \E(\pt \bm{u}) :  \pt \E(\bm{u}_k - \bm{u}) \dtx 
	\end{align*}
	and 
	\begin{align*}
		\int_{\Ot}  &[\C(\vphi_k ) \E( \bm{u}_k) - \C(\vphi) \E( \bm{u})] : \pt  \E(\bm{u}_k - \bm{u}) \dtx \\ 
		&= \begin{aligned}[t]
			\int_0^t \frac{d}{dt} \frac{1}{2} \norm{\E(\bm{u}_k - \bm{u})}_{L^2}^2   \dt 
			&+ (\C(\vphi_k) \pt  \E(\bm{u}_k - \bm{u})  - \pt  \E(\bm{u}_k - \bm{u})  ) :\E(\bm{u}_k - \bm{u}) \\ 
			& - (\C(\vphi_k) - \C(\vphi))  \E(\bm{u}) : \pt  \E(\bm{u}_k - \bm{u})  \dtx. 
		\end{aligned}
	\end{align*}
	Using that the tensors $\C_\nu, \C$ are uniformly positive definite and applying Korn's inequality, we find 
	\begin{align*}
		\int_{\Omega_t}   \C(\vphi_k) \E( \bm{u}_k - \bm{u}) :  \E(\bm{u}_k - \bm{u})  \dx
		&+ \int_{\Ot} \C_\nu(\vphi_k)   \partial_t  \E(\bm{u}_k - \bm{u}) :\pt  \E(\bm{u}_k - \bm{u})   \dt \\
		&\geq C( \norm{\bm{u}_k - \bm{u}}_{L^2(\bm{X})}^2 +\norm{\pt (\bm{u}_k - \bm{u}) }_{L^2(\bm{X})}^2 ) . 
	\end{align*}
	With the help of Young's inequality, we further estimate 
	\begin{align*}
		- \int_{\Omega} & \partial_t  \E(\bm{u}_k - \bm{u}) : (\C_\nu(\vphi_k) \E(\bm{u}_k - \bm{u})   - \E(\bm{u}_k - \bm{u}))\\ 
		&+ \E(\bm{u}_k - \bm{u}) :(\C(\vphi_k) \pt  \E(\bm{u}_k - \bm{u})  - \pt  \E(\bm{u}_k - \bm{u})  ) \dx 
		\leq \rho_{\bm{u}} \norm{ \partial_t  (\bm{u}_k - \bm{u})}_{\bm{X}}^2 + C \norm{ \bm{u}_k - \bm{u}}_{\bm{X}}^2
	\end{align*}
	and remark that for sufficiently small $\rho_{\bm{u}}$ the first term can be absorbed on the right-hand side. Moreover, we observe that the last two integrals can be dealt with similarly to Lemma \ref{lemma:F_continuity}. 
	 Invoking the fundamental theorem of calculus and Korn's inequality therefore yields
	\begin{align*}
		&\norm{ \bm{u}_k - \bm{u}}_{\bm{X}}^2 (t) + \norm{\bm{u}_k - \bm{u}}_{L^2(0, t; \bm{X})}^2 + (1- \rho_{\bm{u}})\norm{\pt (\bm{u}_k - \bm{u}) }_{L^2(0, t; \bm{X})}^2 \\
		&\quad \leq C \int_0^t \norm{ \bm{u}_k - \bm{u}}_{\bm{X}}^2 (\tau)\, d\tau  + \norm{ \bm{u}_k - \bm{u}}_{\bm{X}}^2 (0) \\ 
		&\quad \  +C  \Big| \int_{\Omega_t} 
		\begin{aligned}[t]
			& [\C_\nu(\vphi) -  \C_\nu(\vphi_k)] \E(\pt \bm{u}) : \E(\bm{u}_k - \bm{u})
			 + [\C(\vphi) - \C(\vphi_k )]  \E( \bm{u}) :  \E(\bm{u}_k - \bm{u}) \\ 
			& + [\C_\nu(\vphi) -  \C_\nu(\vphi_k)] \E(\pt \bm{u}) :  \pt \E(\bm{u}_k - \bm{u}) 
			+ [\C(\vphi_k) - \C(\vphi) ] \E(\bm{u}) : \pt  \E(\bm{u}_k - \bm{u}) \dtx \Big| 
		\end{aligned}\\ 
		& \quad \ + \Big| \int_{\Omega_t} \begin{aligned}[t]
			&[\C (\vphi_k) \Tau(\vphi_k) - \C(\vphi) \Tau(\vphi)] :  \E(\bm{u}_k - \bm{u}) 
			+ [\alpha(\vphi_k) p_k - \alpha(\vphi) p] (\nabla \cdot(\bm{u}_k - \bm{u})) * \phi \\ 
			&\! + [\C (\vphi_k) \Tau(\vphi_k) - \C(\vphi) \Tau(\vphi)] :  \pt \E(\bm{u}_k - \bm{u}) 
			+  [\alpha(\vphi_k) p_k - \alpha(\vphi) p] (\nabla \cdot(\pt \bm{u}_k - \pt \bm{u})) * \phi \dtx \Big| 
		\end{aligned}\\ 
	& \quad \  +  \vr 
		 	\norm{\alpha'(\vphi_k) \nabla \theta_k \cdot \nabla \vphi_k - \alpha'(\vphi) \nabla \theta \cdot \nabla \vphi}_{L^2(0, t; L^{4/3})}
		 	\norm{\phi}_{L^\infty(0, t; L^{4/3})}\norm{\bm{u}_k - \bm{u}}_{H^1(0, t; \bm{X})}  \\ 
	& \quad \  +\vr  \norm{\alpha(\vphi_k) \nabla \theta_k - \alpha(\vphi) \nabla \theta}_{L^2(0, t; \bm{L}^2)} \norm{\nabla \phi}_{L^\infty(0, t; \bm{L}^1)}  \norm{\bm{u}_k - \bm{u}}_{H^1(0, t; \bm{X})}  .
	\end{align*}
	Observe that we can again employ Young's inequality on the last three integrals and obtain after absorbing some of the resulting terms on the left hand side that
		\begin{align*}
		&\norm{ \bm{u}_k - \bm{u}}_{\bm{X}}^2 (t) \\
		&\quad \leq C \int_0^t \norm{ \bm{u}_k - \bm{u}}_{\bm{X}}^2 (\tau)\, d\tau  + \norm{ \bm{u}_k - \bm{u}}_{\bm{X}}^2 (0) \\ 
		&\quad \quad +C \Big(
		\begin{aligned}[t]
			& \norm{[\C_\nu(\vphi) -  \C_\nu(\vphi_k)] \E(\pt \bm{u})}_{L^2(\bm{L}^2)}^2
			+ \norm{[\C(\vphi) - \C(\vphi_k )]  \E( \bm{u})}_{L^2(\bm{L}^2)}^2 \Big) 
		\end{aligned}\\ 
		& \quad \quad+ C \Big( \begin{aligned}[t]
			&\norm{[\C (\vphi_k) \Tau(\vphi_k) - \C(\vphi) \Tau(\vphi)]}_{L^2(\bm{L}^2)}^2 
			+ \norm{\alpha(\vphi_k) p_k - \alpha(\vphi) p}_{L^2(L^2)}^2 \Big) 
		\end{aligned}\\ 
		& \quad \quad+ C \Big( \begin{aligned}[t]
			\norm{\alpha'(\vphi_k) \nabla \theta_k \cdot \nabla \vphi_k - \alpha'(\vphi) \nabla \theta \cdot \nabla \vphi}_{L^2(L^{4/3})}^2
			+  \norm{\alpha(\vphi_k) \nabla \theta_k - \alpha(\vphi) \nabla \theta}_{L^2(\bm{L}^2)}^2 \Big) . 
		\end{aligned}
		\numberthis \label{eq:u_gronwall}
	\end{align*}
	On account of the strong convergence $\theta_k \rightarrow \theta$ in $L^2(0, T; X(\Omega))$, see \eqref{conv:theta_strong_regular}, and the strong convergence of the phase-field variable $\vphi_k \rightarrow \vphi$ in $C^0([0, T]; W^{1, 4}(\Omega))$, see \eqref{conv:phi_strong_regular} , along with the embedding $W^{1, 4}(\Omega)) \hookrightarrow L^\infty(\Omega)$ and the continuity of $\alpha'$, we deduce that 
	\begin{align*}
		(\alpha'(\vphi_k) \nabla \theta_k \cdot \nabla \vphi_k )(t) \rightarrow  (\alpha'(\vphi) \nabla \theta \cdot \nabla \vphi) (t) \quad &\textrm{in} \quad L^2(0, T; L^{4/3}(\Omega)), \\ 
		(\alpha(\vphi_k) \nabla \theta_k)(t) \rightarrow (\alpha(\vphi) \nabla \theta) (t) \quad &\textrm{in} \quad L^2(0, T; \bm{L}^2 (\Omega)). 
	\end{align*}
	Since $\bm{u}_ k (0) \rightarrow \bm{u}(0)$, cf. Section \ref{sec:limit_reg}, the strong convergences of $\vphi_k \rightarrow \vphi$ and $p_k \rightarrow p$ in $L^2(\OT)$ and Gronwall's lemma yield that along a suitable subsequence 
	\begin{equation*}
		\norm{ \bm{u}_k - \bm{u}}_{\bm{X}} (t) \rightarrow 0 \quad \textrm{a.e. in } [0, T]. 
	\end{equation*}
	Reminding ourselves of the embedding $H^1(0, T; \bm{X}(\Omega)) \rightarrow C^0([0, T]; \bm{X}(\Omega))$, we can employ $L^p$-$L^q$-compactness to obtain 
	\begin{equation*}
		\bm{u}_k \rightarrow \bm{u} \quad \textrm{in } L^2(0, T; \bm{X}(\Omega)).  
	\end{equation*}
	
\end{proof}

\subsection{Limit passage}\label{sec:limit_reg}

After all the preparation above, we are finally in the position to pass to the limit in the semi-Galerkin system \eqref{eq:galerkin1}-\eqref{eq:galerkin7}.
\par 
\textit{Ad \eqref{eq:weak1_reg}:} Starting with \eqref{eq:galerkin1}, we choose some $z_j$ and an arbitrary function $\vartheta \in C^\infty_c(0, T)$ to obtain $\vartheta z_j \in L^2(0, T; Z_k)$ for all $k \geq j$. Hence, integration with respect to time leads to 
\begin{align*}
	\int_0^T  \prescript{}{(H^1)'}{}\langle \pt \vphi_k, \vartheta z_j \rangle_{H^1}\, \dt =   \int_{\OT} - m(\vphi_k) \nabla \mu_k \cdot \vartheta \nabla z_j+ R(\vphi_k, \E(\bu_k), \theta_k) \vartheta z_j \; \dtx
\end{align*}
for all $k \geq j$. The weak convergence $\vphi_k \rightharpoonup \vphi$ in $L^2(0, T; (H^1(\Omega))')$ and $\vartheta z_j \in L^2(0, T; H^1(\Omega))$ allow us to pass to the limit on the left-hand side. Since $m$ is uniformly bounded and $(\vphi_k)_{k \inN}$ converges pointwise, we obtain with the help of dominated convergence that 
\begin{equation*}
	m(\vphi_k) \vartheta \nabla z_j \rightarrow m(\vphi) \vartheta \nabla z_j  \quad \textrm{in} \quad L^2(0, T; L^2(\Omega)). 
\end{equation*}  
Together with $\nabla \mu_k \rightharpoonup \nabla \mu$ in $L^2(0, T; \bm{L}^2(\Omega))$ and the weak-strong convergence principle, we can pass to the limit in the first term on the right-hand side. \\ 
At last, we deduce the strong convergence $R(\vphi_k, \E(\bu_k), \theta_k) \rightarrow R(\vphi, \E(\bu), \theta)$ in $L^2(\OT)$ from the bound on $R$ and the strong convergences of $\vphi_k$, $ \nabla \bm{u}_k$ and $\theta_k$ in $L^2(\OT)$, allowing us to pass to the limit in the last term.  
\par 
\medskip
\textit{Ad \eqref{eq:weak2_reg}:} Similarly, testing \eqref{eq:galerkin2} with $\vartheta z_j$ and integrating with respect to time yields 
\begin{align*}
	\int_{\OT} \mu_k  \vartheta z_j \dtx 
	=   \int_{\OT} \vepsilon \nabla \vphi_k \cdot \vartheta \nabla   z_j  &+
	\varrho^{1/2}  \Delta \vphi_k \cdot \vartheta \Delta z_j
	+ \Big[ \frac{ 1}{\vepsilon} \psi'(\vphi)  + \Wp(\vphi, \E(\bm{u}))   \\ 
	&\begin{aligned}[t]
		&+\frac{M'(\vphi_k)}{2} (\theta_k - \alpha(\vphi_k) \nabla \cdot \bm{u}_k)^2  \\ 
		&-M(\vphi_k) (\theta_k - \alpha(\vphi_k) \nabla \cdot \bm{u}_k) \alpha'(\vphi_k) \nabla \cdot \bm{u}_k \Big]\,   \vartheta z_j  \; \dtx.
	\end{aligned}
\end{align*} 
Since $\mu_k \rightharpoonup \mu$ in $L^2(0, T; H^1(\Omega))$ and $\vphi_k \rightharpoonup \vphi$ in $L^2(0, T; H^2(\Omega))$ and $\vartheta z_j \in L^2(0, T; H^2(\Omega))$, limit passage on the left-hand side and in the first two terms on the right-hand side is possible. \\
Exploiting $z_j \in H^2_{\bm{n}}(\Omega) \hookrightarrow L^\infty(\Omega)$ along with the strong convergence \eqref{conv:psi_reg}, we can invoke the dominated convergence theorem to find $ \psi'(\vphi_k)  \vartheta z_j  \rightarrow \psi'(\vphi)  \vartheta z_j $ in $L^1(\OT)$ and hence, 
\begin{equation*}
	 \int_{\OT}\frac{ 1}{\vepsilon} \psi'(\vphi_k)  \vartheta z_j \dtx \rightarrow  \int_{\OT} \frac{ 1}{\vepsilon} \psi'(\vphi)  \vartheta z_j  \dtx. 
\end{equation*}
Along with the strong convergence of $\bm{u}_k$ in $L^2(0, T; \bm{X}(\Omega))$ and the growth condition (\hyperref[A:W_growth_conditions]{A3.2}) imposed on $\WE$, generalized dominated convergence implies $\Wp(\vphi_k , \E(\bm{u}_k)) \rightarrow \Wp(\vphi , \E(\bm{u}))$ in $L^1(\OT)$
along a not relabeled subsequence. Since $\vartheta z_j  \in L^\infty(\OT)$, we can also pass to the limit in this term. 
Similarly, the strong convergences of $\bm{u}_k$ in $L^2(0, T; \bm{X})$ and $p_k, \theta_k$ in $L^2(\Omega_T)$ imply 
\begin{align*}
	\frac{M'(\vphi_k)}{2} &(\theta_k - \alpha(\vphi_k) \nabla \cdot \bm{u}_k)^2  
	-M(\vphi_k) (\theta_k - \alpha(\vphi_k) \nabla \cdot \bm{u}_k) \alpha'(\vphi_k) \nabla \cdot \bm{u}_k\\ 
	&\rightarrow 
	\frac{M'(\vphi)}{2} (\theta - \alpha(\vphi) \nabla \cdot \bm{u})^2  
	-M(\vphi) (\theta - \alpha(\vphi) \nabla \cdot \bm{u}) \alpha'(\vphi) \nabla \cdot \bm{u} \quad \textrm{in }\quad  L^1(\OT), 
\end{align*}
along some subsequence, and we can pass to the limit in the last two terms just as before. 
\par 
\medskip 
\textit{Ad \eqref{eq:weak3_reg}:} For an arbitrary $\eta \in \bm{X}(\Omega)$, we obtain from \eqref{eq:galerkin7} along with the identity $ \WE(\vphi_k, \E(\bm{u}_k)) =  \C(\vphi_k) (\E(\bm{u}_k) - \Tau(\vphi_k ))$ that 
\begin{align*}
	\int_{\Omega_T} &\C_\nu (\vphi_k ) \E(\pt \bm{u}_k)  :  \vartheta \E(\bm{\eta}) 
	+  \C(\vphi_k  ) (\E(\bm{u}_k) - \Tau(\vphi_k )) : \vartheta \E(\bm{\eta})
	- \alpha(\vphi_k) p_k (\nabla \cdot \vartheta \bm{\eta}) * \phi \dtx \\ 
	&= \int_{\Omega_T} \bm{f} \cdot \vartheta  \bm{\eta} \dtx + \int_0^T\int_{\Gamma_N} \bm{g}\cdot \vartheta \bm{\eta} \dH \dt  
		- \vr \int_{\Omega_T} \nabla \theta_k \cdot \nabla (\alpha(\vphi_k) (\nabla \cdot \bm{\eta}) * \phi ) \dtx. 
\end{align*} 
Using the properties of the tensors $\C, \C_\nu$, we can rewrite the first two terms as 
\begin{align*}
	&\int_{\Omega_T} \C_\nu (\vphi_k ) \E(\pt \bm{u}_k)  :  \vartheta \E(\bm{\eta}) 
	+  \C(\vphi_k ) \E(\bm{u}_k) : \vartheta \E(\bm{\eta}) \dtx \\ 
	& \quad = 	\int_{\Omega_T}  \E(\pt \bm{u}_k)  :  \C_\nu (\vphi_k ) \vartheta \E(\bm{\eta}) 
	+   \E(\bm{u}_k) : \C(\vphi_k ) \vartheta \E(\bm{\eta}) \dtx.
\end{align*}
Once again we deduce from the strong convergence $\vphi_k \rightarrow \vphi$ in $L^2(\OT)$  together with the boundedness of the tensors that, along a suitable subsequence, $\C_\nu (\vphi_k ) \vartheta \E(\bm{\eta})  \rightarrow \C_\nu (\vphi ) \vartheta \E(\bm{\eta}) $ and $\C(\vphi_k ) \vartheta \E(\bm{\eta}) \rightarrow \C(\vphi  ) \vartheta \E(\bm{\eta})$ in $\bm{L}^2(\Omega_T)$. By the same arguments, it follows that  $ \C(\vphi_k  ) \Tau(\vphi_k ) \rightarrow \C(\vphi) \Tau(\vphi )$ in $\bm{L}^2(\OT)$. \\ 
Taking advantage of the weak convergence $\bm{u}_k \rightharpoonup \bm{u}$ in $H^1(0, T; \bm{X}(\Omega))$, we obtain the desired limit.\\ 
Moreover, we deduce that $\alpha(\vphi_k) \vartheta(\nabla \cdot \bm{\eta}) * \phi \rightarrow \alpha(\vphi) \vartheta (\nabla \cdot \bm{\eta}) * \phi$ in $L^2(\OT)$ and exploit the weak convergence of $p_k$ to get 
\begin{equation*}
	\int_{\OT} p_k \alpha(\vphi_k) \vartheta (\nabla \cdot \bm{\eta}) * \phi \dtx 
	\rightarrow 
		\int_{\OT} p\alpha(\vphi) \vartheta ( \nabla \cdot \bm{\eta}) * \phi \dtx
		\quad \textrm{as } k \rightarrow \infty . 
\end{equation*}
Lastly, the strong convergence $\vphi_k \rightarrow \vphi$ in $C^0([0, T]; W^{1, 4}(\Omega))$ implies 
\begin{equation*}
	\nabla (\alpha(\vphi_k) (\nabla \cdot \bm{\eta}) * \phi ) \rightarrow \nabla (\alpha(\vphi) (\nabla \cdot \bm{\eta}) * \phi )
	\quad \textrm{in} \quad L^2(0, T; \bm{L}^2(\Omega)). 
\end{equation*} Along with $\theta_k \rightarrow \theta$ in $L^2(0, T; X(\Omega))$, this entails
\begin{equation*}
	\vr \int_{\Omega_T} \nabla \theta_k \cdot \nabla (\alpha(\vphi_k) (\nabla \cdot \bm{\eta}) * \phi ) \dtx 
	\rightarrow 
	\vr \int_{\Omega_T} \nabla \theta \cdot \nabla (\alpha(\vphi) (\nabla \cdot \bm{\eta}) * \phi ) \dtx. 
\end{equation*}
\par
\medskip 
\textit{Ad \eqref{eq:weak4_reg}:} For any fixed $y_j$, we obtain from \eqref{eq:galerkin3} that 
\begin{equation*}
	\int_0^T  \prescript{}{X'}{}\langle \pt \theta_k, \vartheta y_j \rangle_{X}\, \dt 
	=   \int_{\OT} - \kappa(\vphi_k) \nabla p_k \cdot \vartheta \nabla y_j + S_f(\vphi_k, \E(\bu_k), \theta_k) \vartheta y_j \; \dtx
\end{equation*} 
 for all $k \geq j$. The assumptions on $\kappa$ and $S_f$ along with the compactness results for $\theta_k$ and $p_k$ allow us to argue analogously to the first equation \eqref{eq:weak1_reg}.
 \par 
 \medskip
 \textit{Ad \eqref{eq:weak5_reg}:} At last, we test \eqref{eq:galerkin4} with $\vartheta y_j$ and find the postulated identity by similar arguments as above. 
\par 
\medskip 
Recall that $\spn \{ z_j : j \inN \}$ is a dense subset of $H^2_{\bm{n}}(\Omega)$ and $\spn \{y_j : j \inN \}$ is dense in $X(\Omega)$. Since $C^\infty_0(0, T)$ is a dense subspace of $L^2(0, T)$, the discussion above already suffices to conclude that the limit $(\vphi, \mu, \bm{u}, \theta, p)$ satisfies \eqref{eq:weak1_reg}-\eqref{eq:weak5_reg} for all applicable test functions $\zeta, \bm{\eta}$ and $\xi$, as postulated in Theorem \ref{th:existence_reg}. 

\subsubsection*{Recovery of initial conditions}

Due to the \textit{a priori} estimates from Lemma \ref{lem:a_priori_regular}, the Aubin--Lions--Simon theorem yields 
\begin{gather*}\label{conv:initial_strong_reg}
		\begin{aligned}
		\vphi_k \rightarrow \vphi \quad &\textrm{in} \quad C^0([0, T]; L^2(\Omega)),\\ 
		\bm{u}_{k} \rightarrow \bm{u} \quad &\textrm{in} \quad C^0([0, T]; L^2(\Omega)),\\ 
		\theta_k \rightarrow \theta \quad &\textrm{in} \quad C^0([0, T]; L^2(\Omega)). 
	\end{aligned}\numberthis 
\end{gather*}

In particular, this implies 
\begin{equation*}
	\vphi_k(0) \rightarrow \vphi(0) \quad \textrm{and} \quad \bm{u_k}(0) \rightarrow \bm{u}(0) \quad \textrm{and} \quad 
	\theta_k(0) \rightarrow \theta(0) \quad  \textrm{in} \quad L^2(\Omega).
\end{equation*}
Recalling that  as $k \rightarrow \infty$
\begin{equation*}
	\vphi_k(0) = \vphi_{0, k} = \Pi_k^z(\vphi_0) \rightarrow \vphi_0 \quad \textrm{and} \quad 
	\theta_k(0) = \theta_{0, k} = \Pi_k^z(\theta_0) \rightarrow \theta_0 \quad \textrm{in} \quad L^2(\Omega), 
\end{equation*}
the uniqueness of limits yield 
\begin{equation*}
	\vphi(0) = \vphi_0 \quad  \textrm{and} \quad 
	\theta(0) = \theta_0  
	\quad \textrm{a.e. in } \Omega. 
\end{equation*}
Lastly, we observe $\bm{u}_k(0) = \bm{u}_0$ for all $k \inN$ and deduce $\bm{u}(0) = \bm{u}_0 \textrm{ a.e. in } \Omega$. 

\subsubsection*{\textit{A priori} estimates} 

We start by collecting some results from Lemma \ref{lem:a_priori_regular}. Inserting \eqref{eq:mu_H^1} into \eqref{eq:phi_H^1'}, we obtain
\begin{equation}
	\norm{\pt \varphi_k}_{L^2((H^1)')}^2 \leq C \Big(  \norm{\vphi_{0}}_{H^1}^4 +  \varrho \norm{\Delta \vphi_{0}}_{L^2}^4+  \frac{1}{\vepsilon^2}  \norm{\psi(\vphi_{0})}_{L^1}^2 +  \norm{\bm{u}_0}_{\bm{X}}^2 + \norm{\theta_{0}}_{L^2}^4 + \vr^2 \norm{\theta_0}_{X}^4+ 1 \Big) . 
\end{equation}
Moreover, we insert \eqref{eq:p_H^1} into \eqref{eq:theta_H^1'} and deduce together with the estimates  \eqref{eq:energy_estimate_reg_1} that 
\begin{equation*}
	\norm{\pt \theta_k}_{L^2((H^1)')}^2 \leq 
	 C \Big(  \norm{\vphi_{0}}_{H^1}^2 +  \varrho^{1/2} \norm{\Delta \vphi_{0}}_{L^2}^2+  \frac{1}{\vepsilon}  \norm{\psi(\vphi_{0})}_{L^1} +  \norm{\bm{u}_0}_{\bm{X}}^2 + \norm{\theta_{0}}_{L^2}^2 + \vr \norm{\theta_{0}}_{X}^2 + 1 \Big) .
\end{equation*}
Thus, these two estimates together with \eqref{eq:energy_estimate_reg_1} establish 
\begin{align*}
		\norm{\vphi_k(t)}_{H^1}^2 
	&+ \norm{ \varrho^{1/4} \Delta \vphi_k(t)}_{L^2}^2  
	+ \norm{\psi(\vphi_k)(t)}_{L^1} 
	+ \norm{\bm{u}(t)}_{\bm{X}}^2
	+  \norm{\theta_k(t)}_{L^2}^2 
	+\norm{\vr^{1/2} \nabla \theta_k(t)}_{L^2}^2\\ 
	&+\norm{\pt \varphi_k}_{L^2((H^1)')}^2
	+\norm{\pt \theta_k}_{L^2(X')}^2
	+\norm{ \mu_k}_{L^2(H^1)}^2 
	+ \norm{ p_k}_{L^2(H^1)}^2
	+ \norm{\pt \bm{u}_k}_{L^2(\bm{X})}^2\\ 
	&\leq C( \norm{\vphi_{0}}_{H^1}^4 +  \varrho \norm{\Delta \vphi_{0}}_{L^2}^4+  \frac{1}{\vepsilon^2}  \norm{\psi(\vphi_{0})}_{L^1}^2 +  \norm{\bm{u}_0}_{\bm{X}}^4 + \norm{\theta_{0}}_{L^2}^4 +  \vr^2\norm{\theta_{0}}_{X}^4+ 1)
\end{align*} 
for almost all $t \in (0, T)$ and all $\vr < 1$. \\
Recall that \ref{A:psi_0} stipulates $\psi(\vphi') \geq 0$ for all $\vphi' \in \R$. Hence, $\psi(\vphi_k)$ is a nonnegative, continuous function and since $\vphi_k \rightarrow \vphi$ pointwise a.e. in $\OT$, applying Fatou's lemma yields 
\begin{equation*}
	\int_\Omega \psi(\vphi(t))  \dx \leq \liminf_{k \rightarrow \infty} \int_\Omega \psi(\vphi_k(t))\dx 
\end{equation*}
for almost all $t \in (0, T)$. By taking the limit inferior on both sides above and exploiting weak/weak* lower semi-continuity, we conclude 
\begin{align*}
	\norm{\vphi(t)}_{H^1}^2 
	&+ \norm{ \varrho^{1/2} \Delta \vphi(t)}_{L^2}^2  
	+ \norm{\psi(\vphi)(t)}_{L^1} 
	+ \norm{\bm{u}(t)}_{\bm{X}}^2
	+  \norm{\theta(t)}_{L^2}^2 
	+\norm{\vr^{1/2} \nabla \theta (t)}_{L^2}^2\\ 
	&+\norm{\pt \varphi}_{L^2((H^1)')}^2
	+\norm{\pt \theta}_{L^2(X')}^2
	+\norm{ \mu}_{L^2(H^1)}^2 
	+ \norm{ p}_{L^2(H^1)}^2
	+ \norm{\pt \bm{u}}_{L^2(\bm{X})}^2\\ 
	&\leq C( \norm{\vphi_{0}}_{H^1}^4 +  \varrho \norm{\Delta \vphi_{0}}_{L^2}^4+  \frac{1}{\vepsilon^2}  \norm{\psi(\vphi_{0})}_{L^1}^2 +  \norm{\bm{u}_0}_{\bm{X}}^4 + \norm{\theta_{0}}_{L^2}^4 +  \vr^2\norm{\theta_{0}}_{X}^4+ 1). 
	\numberthis \label{est:energy_explicit}
\end{align*}


\section{Existence of weak solutions} \label{sec:existence}

We now aim to establish the existence of weak solutions to the Cahn--Hilliard--Biot system without regularizations. To this end, we consider a family of weak solutions $(\vphi_\varrho, \mu_\varrho, \bm{u}_\varrho, \theta_\varrho, p_\varrho)$ to the regularized system and pass the limit $\varrho \rightarrow 0$. \par 
To recover our initial conditions, these need to be approximated in the corresponding spaces. In particular, we need to find a family $\{ \vphi_{\varrho, 0}\}_{\varrho}$ with $\vphi_{\varrho, 0} \rightarrow \vphi_0$ as $\varrho \searrow 0$ such that $\norm{\psi(\vphi_{\varrho, 0} ) }_{L^1}$ can be bounded by $\norm{\psi(\vphi_{ 0} ) }_{L^1}$. We note that this is not necessary for $\theta_0$ but refer to Section \ref{sec:vanishing_convolution} for a more detailed discussion. \\ 
Moreover, since we can no longer rely on the additional regularity for $\vphi_\varrho$, we need to slightly change the arguments which allowed us to deduce the strong convergences for $p$ and $\bm{u}$ in the previous section.  \par
Most importantly, we will deduce the strong convergence of $\theta$ in $L^2(\OT)$ even for the case that $\vr \searrow 0$.

\subsection{More general initial conditions} \label{sec:general_initial}

To obtain weak solutions for the regularized problem, we had to assume that the initial conditions satisfy $\vphi_0 \in H^2_{\bm{n}}(\Omega)$. We would like to weaken these assumptions and allow for more general initial conditions better fitted to the problem. 
\par 
\medskip
Suppose $\vphi_0 \in H^1(\Omega)$ such that $\psi(\vphi_0) \in L^1(\Omega)$. We want to employ a strategy by Colli, Frigeri and Grasselli \cite{COLLI2012428}, which was also used by Garcke, Lam and Signiori \cite{lam06}, to find a family of functions $\{ \vphi_{\varrho, 0}\}  \subset H^2_{\bm{n}}(\Omega)$ that converges to $\vphi_0$ in a suitable manner, while simultaneously the integral over $\psi(\vphi_{\varrho, 0})$ remains bounded. To this end, consider the elliptic problem 
\begin{gather*}
	\begin{cases}
		- \varrho^{1/2} \Delta \vphi_{\varrho, 0} + \vphi_{\varrho, 0} = \vphi_0 \quad & \textrm{in } \Omega, \\ 
		\nabla \vphi_{\varrho, 0} \cdot \bm{n} = 0  \quad & \textrm{on } \Gamma. 
	\end{cases}\numberthis\label{elliptic_phi}
\end{gather*}
Employing the Lax--Milgram theorem, we find that these problems admit unique weak solutions and elliptic regularity theory further implies $\vphi_{0, \varrho} \in H^2_{\bm{n}}(\Omega)$. We test the weak formulation with $\vphi_{\varrho, 0}$ and obtain with the help of Young's inequality 
\begin{equation}\label{eq:initial_1}
	2 \varrho^{1/2} \norm{\nabla \vphi_{\varrho, 0}}_{L^2}^2 + \norm{ \vphi_{\varrho, 0}}_{L^2}^2
	\leq \norm{ \vphi_{0}}_{L^2}^2. 
\end{equation}
On the other hand, integration by parts and the fundamental lemma of the calculus of variations imply that \eqref{elliptic_phi} already holds almost everywhere in $\Omega$. Multiplication with $\Delta \vphi_{\varrho,0}$ and integration by parts further yield
\begin{equation}\label{eq:initial_2}
	2 \varrho^{1/2} \norm{\Delta \vphi_{\varrho, 0}}_{L^2}^2 + \norm{ \nabla \vphi_{\varrho, 0}}_{L^2}^2
	\leq \norm{\nabla \vphi_{0}}_{L^2}^2. 
\end{equation} 
Together with \eqref{eq:initial_1}, this implies 
\begin{equation} \label{est:initial_4}
	\norm{\vphi_{\varrho, 0}}_{H^1}^2 + \varrho^{1/2} \norm{\Delta \vphi_{\varrho, 0}}_{L^2}^2 
	\leq C\norm{\vphi_{0}}_{H^1}^2 . 
\end{equation}
In summary, this tells us that the family $\{ \vphi_{\varrho, 0}\}$ is uniformly bounded in $H^1(\Omega)$ and we deduce the existence of some function $\vphi^* \in H^1(\Omega)$ such that 
\begin{equation*}
	\vphi_{\varrho, 0} \rightharpoonup \vphi^* \quad \textrm{in} \quad H^1(\Omega)
\end{equation*}
along a not relabeled subsequence. Exploiting $\varrho \nabla \vphi_{\epsilon, 0} \rightharpoonup 0$ when passing to the limit in the weak formulation gives 
\begin{equation*}
	\int_\Omega \vphi^* \zeta \dx = \int_\Omega \vphi_{0} \zeta \dx 
	\quad \textrm{for all} \quad \zeta \in H^1(\Omega)
\end{equation*}
and invoking the fundamental lemma of the calculus of variations, we find $\vphi^* = \vphi_0$. Moreover, the Rellich--Kondrachov theorem and the uniqueness of weak limits imply
\begin{equation}\label{conv:initial_strong}
		\vphi_{\varrho, 0} \rightarrow \vphi_0 \quad \textrm{in} \quad L^2(\Omega). 
\end{equation}
We proceed by defining the function $	G(s) \coloneqq \psi(\vphi) + \tfrac{1}{2}C_2 \vphi^2$,
which is nonnegative due to \ref{A:psi_0}. Moreover, using \ref{A:psi_1} and \ref{A:psi_2}, we compute 
\begin{equation*}
	G''(\vphi) = \psi_1''(\vphi) + \psi_2''(\vphi) + C_2 \geq 0
\end{equation*}
and find that $G$ is convex. Lastly, we would like to remark that due to \ref{A:initial}, we have $G(\vphi_0) \in L^1(\Omega)$. \\
From the embedding $H^2(\Omega) \hookrightarrow L^\infty(\Omega)$, we deduce $G'(\vphi_{\varrho, 0}) \in H^1(\Omega)$, which is therefore a valid test function in \eqref{elliptic_phi}. Hence, 
\begin{equation*}
	\int_\Omega (\vphi_{\varrho, 0} - \vphi_0) G'(\vphi_{\varrho, 0}) = - \int_\Omega \varrho G''(\vphi_{\varrho, 0}) |\nabla \vphi_{\varrho, 0}|^2 \dx \leq 0
\end{equation*} 
and together with the convexity of $G$ and the characterization \cite[E4.7(1)]{alt2013lineare}, we find 
\begin{equation*}
	\int_\Omega G(\vphi_{\varrho, 0}) \dx \leq \int_\Omega G(\vphi_0) + G'(\vphi_{\varrho, 0})(\vphi_{\varrho} - \vphi_0) \dx 
	\leq \int_\Omega G(\vphi_0) \dx  < \infty. 
\end{equation*}
Thus, by the strong convergence \eqref{conv:initial_strong}, we infer 
\begin{align*}
	\limsup_{\varrho \rightarrow 0} \int_\Omega \psi(\vphi_{\varrho, 0}) \dx  
	&\leq \limsup_{\varrho \rightarrow 0} \int_\Omega G(\vphi_{\varrho, 0})  \dx 
	-   \lim_{\varrho \rightarrow 0} \int_\Omega  \frac{C_2}{2} \vphi_{\varrho, 0}^2 \dx\\
	 &\leq  \int_\Omega G(\vphi_0) -  \frac{C_2}{2} \vphi_{0}^2 \dx 
	= \int_\Omega \psi(\vphi_0) \dx.  \numberthis \label{est:initial_3}
\end{align*}

\subsection{\textit{A priori} estimates and compactness results}

Suppose $\vphi_0  \in H^1(\Omega)$ such that $\psi(\vphi_0) \in L^1(\Omega)$ and assume $\bm{u}_0 \in H^1(\Omega)$. Moreover, we choose a family $\{ \vphi_{\varrho, 0}\}  \subset H^2_{\bm{n}} (\Omega)$ as in Section \ref{sec:general_initial} and obtain from Theorem \ref{th:existence_reg} the existence of weak solutions $(\vphi_\varrho, \mu_\varrho, \bm{u}_\varrho, \theta_\varrho, p_\varrho)$ to the regularized Cahn--Hilliard--Biot system with corresponding initial conditions $(\vphi_{\varrho, 0}, \bm{u}_0, \theta_0)$. Moreover, the energy estimate \eqref{est:energy_explicit} along with \eqref{est:initial_3} and \eqref{est:initial_4} yield the uniform estimate 
\begin{align*}
	\norm{\vphi_{\varrho}(t)}_{H^1}^2 
	&+ \norm{ \varrho^{1/4} \Delta \vphi_{\varrho}(t)}_{L^2}^2  
	+ \norm{\psi(\vphi_{\varrho})(t)}_{L^1} 
	+ \norm{\bm{u}_{\varrho}(t)}_{\bm{X}}^2
	+  \norm{\theta_{\varrho}(t)}_{L^2}^2 
	+ \norm{\vr^{1/2 }\nabla \theta_{\vr} (t)}_{L^2}^2\\ 
	&+\norm{\pt \varphi_{\varrho}}_{L^2((H^1)')}^2
	+\norm{\pt \theta_{\varrho}}_{L^2(X')}^2
	+\norm{ \mu_{\varrho}}_{L^2(H^1)}^2 
	+ \norm{ p_{\varrho}}_{L^2(X)}^2
	+ \norm{\pt \bm{u}_{\varrho}}_{L^2(\bm{X})}^2\\ 
	&\leq C( \norm{\vphi_{0}}_{H^1}^4 +  \frac{1}{\vepsilon^2}  \norm{\psi(\vphi_{0})}_{L^1}^2 +  \norm{\bm{u}_0}_{\bm{X}}^4 + \norm{\theta_{0}}_{L^2}^4 +  \vr^2\norm{\theta_{0}}_{X}^4+ 1), 
\end{align*} 
or more concisely, 
	\begin{align*}
	\norm{\vphi_\varrho}_{L^\infty(H^1)} &+ \norm{ \varrho^{1/2} \Delta \vphi_\varrho}_{L^\infty(L^2)}  
	+ \norm{\vphi_\varrho}_{H^1((H^1)')}
	+ \norm{\mu_\varrho}_{L^2(H^1)}
	+ \norm{\psi(\vphi_\varrho)}_{L^\infty(L^1)} \\ 
	&+ \norm{\bm{u}_\varrho}_{H^1(\bm{X})}
	+  \norm{\theta_\varrho}_{L^\infty(L^2)} 
	+  \norm{ \vr^{1/2} \nabla  \theta_\varrho}_{L^\infty(L^2)} 
	+ \norm{\theta_\varrho}_{H^1(X')}
	+ \norm{p_\varrho}_{L^2(X)}\\ 
	&\leq C( \norm{\vphi_{0}}_{H^1}^4 +  \frac{1}{\vepsilon^2}  \norm{\psi(\vphi_{0})}_{L^1}^2 +  \norm{\bm{u}_0}_{\bm{X}}^4 + \norm{\theta_{0}}_{L^2}^4 +  \vr^2\norm{\theta_{0}}_{X}^4+ 1),  \label{eq:a_priori_not_regularized} \numberthis
\end{align*}
where $C > 0$ does not depend on $\vr \in (0, 1)$. \\ 
Hence, we deduce the existence of functions $(\vphi, \mu, \bm{u}, \theta, p)$ such that, along suitable subsequences as $\varrho \rightarrow 0$, 
\begin{gather}\label{conv:comapactness}
	\begin{align*}
		\vphi_\varrho \rightharpoonup \vphi \quad & \textrm{in} \quad  L^2(0, T; H^1(\Omega)) \cap H^1(0, T; H^1(\Omega)'),  \\ 
		\varrho^{1/2} \Delta \vphi_\varrho \rightharpoonup 0\quad & \textrm{in} \quad   L^2(0, T; L^2(\Omega)), \\ 
		\mu_\varrho \rightharpoonup \mu \quad & \textrm{in} \quad  L^2(0, T; H^1(\Omega)), \\ 
		\bm{u}_\varrho  \rightharpoonup \bm{u} \quad & \textrm{in} \quad  H^1(0, T; \bm{H}^1_{\Gamma_D}(\Omega)),  \\ 
		\theta_\varrho \rightharpoonup \theta \quad & \textrm{in} \quad   L^2(0, T; L^2(\Omega)) \cap H^1(0, T; X'(\Omega)),  \\ 
		\vr \nabla \theta_\varrho \rightharpoonup 0 \quad & \textrm{in} \quad   L^2(0, T; \bm{L}^2(\Omega)), \\ 
		p_\varrho \rightharpoonup p \quad & \textrm{in} \quad  L^2(0, T; H^1(\Omega)). 
	\end{align*} \numberthis 
\end{gather}
As in Section \ref{sec:existence_semi_galerkin}, we further find 
\begin{align}
	\vphi_\varrho  &\rightarrow \vphi \quad \quad \ \ \textrm{in} \quad C^0([0, T]; L^r(\Omega)),\label{conv:phi_strong} \\ 
	\psi'(\vphi_\varrho) &\rightarrow \psi'(\vphi) \quad \textrm{in} \quad L^1(\Omega_T), \label{conv:psi}
\end{align}
where $1 \leq r < \infty$ if $n \leq 2$ and $1 \leq r < 6$ if $n = 3$. Lastly, we apply elliptic theory, which entails that, cf.  \cite[Prop.~5.7.2]{taylor2010partial}
\begin{equation}\label{eq:H^2_phi_not_regularized}
	\norm{\vr^{1/4} \vphi_\vr}_{L^\infty(H^2_{\bm{n}})}^2 \leq C \big( \norm{\vr^{1/4} \Delta \vphi_\vr}_{L^\infty(L^2)} ^2+  \norm{ \vr^{1/4} \vphi_\vr}_{L^\infty(H^1)}^2 \big) < C, 
\end{equation}
where $C > 0$ independent of $\vr \in (0, 1)$. 

\subsection{Additional compactness results}

As mentioned above, for the limit $\varrho \rightarrow 0$ we need a different argument to obtain the necessary compactness result for $p$. Observe that for the weak solutions to the regularized problem the identity \eqref{eq:p_identity_new} already holds without the projection and that the functions on the right-hand side are weakly differentiable with respect to time. Taking advantage of these properties, we can show an uniform estimate for the differences $\tau_h p_\varrho - p_\varrho$ and an application of a version of the Aubin--Lions--Simon theorem yields strong convergence in the space $L^2(0, T; L^2(\Omega))$. \par 

\begin{lemma} \label{lem:p_strong_non_regular}
	There exists a subsequence of ${\varrho} \rightarrow 0$ such that, along this subsequence, 
	\begin{equation*}
		p_{\varrho} \rightarrow p \quad \textrm{in} \quad  L^2(0, T; L^2(\Omega)). 
	\end{equation*}
\end{lemma}

\begin{proof}
	Due to the \textit{a priori} estimates, we already have $p_\varrho \in L^2(0, T; X(\Omega))$. Recall that, if we can further show 
	\begin{equation}
		\norm{\tau_h {p_{\vr}} - {p_{\vr}}}_{L^1(0, T-h; W^{-3, 2}  (\Omega))} \dt \rightarrow 0 
		\quad \textrm{as } h \rightarrow 0 \quad \textrm{uniformly in } k, 
	\end{equation}
	for some subsequence ${\varrho} \rightarrow 0$, then \cite[Thm.~5]{simon1986compact} already yields the assertion due to the compact embedding $X(\Omega) \xhookrightarrow{cpt} L^2(\Omega) \xhookrightarrow{c}W^{-3, 2}(\Omega)$. 
	Using the notation from Lemma~\ref{lem:p_strong}, the identity \eqref{eq:weak5_reg} gives rise to 
	\begin{align*}
			&\int_{\Omega_{T-h}} \Delta_h p_\varrho \, \xi \dtx \\ 
			&\quad = \begin{aligned}[t]
				&\int_{\Omega_{T-h}} - \vr \Delta_h \theta_\vr \, \Delta \xi \dtx 
				+\int_{\Omega_{T-h}}  \tau_h M(\vphi_\vr)\, \Delta_h \theta_\vr   \,  \xi \dtx
				+ \int_{\Omega_{T-h}}  \theta_\vr \Delta_h M(\vphi_\vr)\,  \xi \dtx \\ 
				&+  \int_{\Omega_{T-h}}   \tau_h(M(\vphi_\vr) \alpha(\vphi_\vr)) \Delta_h(\nabla \cdot \bm{u}_\vr)   \, \xi \dtx
				+  \int_{\Omega_{T-h}}   \nabla \cdot \bm{u}_\vr \Delta_h(M(\vphi_\vr) \alpha(\vphi_\vr))   \,\xi \dtx
			\end{aligned} \\ 
			&\quad = I + II + III + IV + V, 
	\end{align*}
	for all $\xi \in L^\infty(0, T-h; W^{3, 2}_{0}(\Omega))$
	and observe that $III, IV$ and $V$ can be treated analogously to the previous arguments. The same is true for $I$, since \eqref{eq:translations_I} holds independently for all $\vr \in (0, 1)$. 
	For $II$, we note that $W^{3, 2}_{0}(\Omega) \hookrightarrow L^\infty(\Omega)$ and compute 
		\begin{align*}
		\norm{\tau_h  M(\vphi_\varrho) \phi}_{X} 
		&\leq  \norm{\tau_h  M(\vphi_\varrho) \xi}_{L^2} + \norm{\tau_h  (M'(\vphi_\varrho)\nabla \vphi_\varrho)  \xi}_{\bm{L}^2} + \norm{\tau_h  M(\vphi_\varrho) \nabla  \xi}_{\bm{L}^2}\\ 
		& \leq \overline{M} \norm{\xi}_{L^2} + \overline{M} \norm{\tau_h  \nabla \vphi_\varrho}_{\bm{L}^2} \norm{ \xi}_{L^\infty} + \overline{M} \norm{\nabla \xi}_{\bm{L}^2}\\ 
		& \leq  \overline{M} \norm{\xi}_{L^2} + \overline{M} \norm{\tau_h  \vphi_\varrho}_{H^1} \norm{\xi}_{W^{3, 2}_{0}} + \overline{M} \norm{\xi}_{X}
		 \leq C  \norm{\xi}_{W^{3, 2}_{0}}. 
	\end{align*}
	Now we can argue just as before to find 
		\begin{align*}
			\left| \int_{\Omega_{T-h}}  \tau_h M(\vphi_\vr)\, \Delta_h \theta_\vr  \,  \xi \dtx \right|
		&\leq \norm{\Delta_h \theta_k }_{L^1(0, T-h; X')} \norm{ \tau_h M(\vphi_k)\, \xi}_{L^\infty(0, T-h; X)}\\ 
		& \leq  C \norm{\Delta_h \theta_k }_{L^1(0, T-h; X')} \norm{\xi}_{L^\infty(0, T-h; W^{3, 2}_{0})}. 
	\end{align*} 
	and conclude the proof. 
\end{proof}

	As in the previous section, we would now like to deduce 
	\begin{align*}
		\bm{u}_\varrho \rightarrow \bm{u} \quad \textrm{in} \quad L^2(0, T; \bm{X})  
		\quad \textrm{and} \quad 
		\theta_\varrho  \rightarrow \theta  \quad \;\textrm{in} \quad L^2(\OT). 
	\end{align*} 
	along suitable subsequences of $\varrho \searrow 0$. 
	Due to missing uniform estimates for $\theta_\vr$ in $L^2(0, T; H^{1+ \gamma}_{\Gamma_D}(\Omega))$, the respective arguments need to be modified. 
	
	\begin{lemma}\label{lemma:u_strong_not_reg}
		There exists a subsequence of $\vr \rightarrow 0$ such that, along this not relabeled subsequence, 
		\begin{equation*}
			\bm{u}_{\vr} \rightarrow \bm{u} \quad \textrm{in} \quad  L^2(0, T; \bm{X}(\Omega)). 
		\end{equation*}
	\end{lemma}
	
\begin{proof}
	First of all, we obtain for $\vr \searrow 0$ that 
	\begin{align*}
		 \vr &\int_{\OT} \nabla \theta_\vr \cdot \nabla (\alpha(\vphi_\vr) (\nabla \cdot \bm{\eta}) * \phi ) \dtx\\ 
		&= \vr \int_{\OT}  \alpha'(\vphi_\vr) \nabla \theta_\vr \cdot \nabla \vphi_\vr \,(\nabla \cdot \bm{\eta}) * \phi +  \alpha(\vphi_\vr )\nabla \theta_\vr  \cdot  (\nabla \cdot \bm{\eta}) * \nabla \phi  \dtx \\ 
		& \leq \vr C \Big( 
		 \begin{aligned}[t]
			&\norm{\nabla \theta_\vr}_{L^\infty (\bm{L}^2)} \norm{\nabla \vphi_\vr}_{L^2(L^4)} \norm{\bm{\eta}}_{L^2(\bm{X})} \norm{\vphi}_{L^\infty(L^{4/3})}  \\ 
			&+ \norm{\alpha(\vphi_\vr)}_{L^\infty(L^\infty )}
			\norm{\nabla \theta_\vr}_{L^2(\bm{L}^2)} \norm{\nabla \phi}_{L^\infty(\bm{L}^1)} \norm{\bm{\eta}}_{L^2(\bm{X})}\Big)
		\end{aligned}\\ 
		&  \leq \vr^{1/4} C( \bm{\eta}) \Big( \norm{ \vr^{1/2} \nabla \theta_\vr}_{L^\infty (\bm{L}^2)} \norm{ \vr^{1/4} \vphi_\vr}_{L^\infty(H^2_{\bm{n}})} 
		+ \norm{\vr^{1/2} \nabla \theta_\vr}_{L^\infty (\bm{L}^2)} 
		 \Big)  \rightarrow 0, \label{eq:conv_u_not_regularized} \numberthis 
	\end{align*}
	since $\norm{ \vr^{1/2} \nabla \theta_\vr}_{L^\infty (\bm{L}^2)}  + \norm{ \vr^{1/4} \vphi_\vr }_{L^\infty(H^2_{\bm{n}})}  < C$ for
	some $C >0$ independent of $\vr > 0$, cf. \eqref{eq:a_priori_not_regularized} and \eqref{eq:H^2_phi_not_regularized}. As this holds for all $\bm{\eta} \in L^2(0, T; \bm{X}(\Omega))$, the compactness results \eqref{conv:comapactness} and similar arguments as in Section~\ref{sec:limit_reg} allow us to pass to the limit in \eqref{eq:weak3_reg} and we arrive at 
	\begin{align*}
		\int_{\OT} \C_\nu (\vphi ) \E(\partial_t \bm{u}) :  \E(\bm{\eta}) &+ \WE(\vphi , \E(\bm{u})) : \E(\bm{\eta}) -  \alpha(\vphi) p\, (\nabla \cdot \bm{\eta})* \phi \dtx \\ 
		&  = \int_{\OT} \bm{f} \cdot \bm{\eta} \dtx + \int_0^T \int_{\bm{\Gamma_N}} \bm{g} \cdot \bm{\eta} \dH \dt , 
	\end{align*}
	which again holds for all $\bm{\eta} \in L^2(0, T; \bm{X}(\Omega))$. As in Lemma~\ref{lem:convergence_u_reg}, we test this equation and \eqref{eq:weak3_reg} with both $\bm{u}_k \chi_{[0, t]} -  \bm{u}\chi_{[0, t]}$ and $\pt\bm{u}_k \chi_{[0, t]} - \pt \bm{u}\chi_{[0, t]}$. The same computations as before now lead to 
		\begin{align*}
		&\norm{ \bm{u}_\vr - \bm{u}}_{\bm{X}}^2 (t) \\
		& \quad    \leq C \int_0^t \norm{ \bm{u}_\vr - \bm{u}}_{\bm{X}}^2 (\tau)\, d\tau  + \norm{ \bm{u}_\vr - \bm{u}}_{\bm{X}}^2 (0) \\ 
		&\quad   +C \Big(
		\begin{aligned}[t]
			& \norm{[\C_\nu(\vphi) -  \C_\nu(\vphi_\vr)] \E(\pt \bm{u})}_{L^2(\bm{L}^2)}^2
			+ \norm{[\C(\vphi) - \C(\vphi_\vr )]  \E( \bm{u})}_{L^2(\bm{L}^2)}^2 \Big) 
		\end{aligned}\\ 
		&\quad    +C \Big(
		 \begin{aligned}[t]
			&\norm{[\C (\vphi_\vr) \Tau(\vphi_\vr) - \C(\vphi) \Tau(\vphi)]}_{L^2(\bm{L}^2)}^2 
			+ \norm{\alpha(\vphi_\vr) p_\vr - \alpha(\vphi) p}_{L^2(L^2)}^2 \Big) 
		\end{aligned}\\ 
	&    \quad  +  \vr \int_{\Omega_T} \left| \nabla \theta_\vr \cdot \nabla \big( \alpha(\vphi_\vr) (\nabla \cdot  (\bm{u}_\vr  -  \bm{u} ) * \phi ) \big)   \right| \dtx\\ 
	&\quad +  \vr \int_{\Omega_T} \left| \nabla \theta_\vr \cdot \nabla \big(  \alpha(\vphi_\vr) (\nabla \cdot  \pt (\bm{u}_\vr  -  \bm{u} ) * \phi ) \big)  \right| \dtx. 
	\end{align*}
	Hence, \eqref{eq:conv_u_not_regularized} and the arguments from above yield the assertion. 
\end{proof}

Finally, we have to show the strong convergence of the volumetric fluid content $\theta_\vr$. Here, we exploit that for all $\vr > 0$ the equation \eqref{eq:p_identity_new} defines an elliptic problem, whose solution operators, considered as an operator from $L^2(\Omega) \rightarrow L^2(\Omega)$, are uniformly bounded and strongly continuous. 

\begin{lemma}\label{lemma:convergence_p_not_reg}
	There exists a subsequence of $\vr \rightarrow 0$ such that, along this not relabeled subsequence, 
	\begin{equation*}
		\theta_\vr \rightarrow \theta \quad \textrm{in} \quad  L^2(0, T; L^2(\Omega)). 
	\end{equation*}
\end{lemma}

\begin{proof}
	We define the family of operators 
	\begin{equation*}
		- \vr \Delta + M(\vphi_\vr(t))  : X(\Omega) \rightarrow X'(\Omega), 
		\quad  v \mapsto \left( w \mapsto \int_\Omega \vr \nabla v \cdot \nabla w + M(\vphi_\vr(t)) v w \dx  \right)
	\end{equation*}
	and deduce with the help of the Lax-Milgram theorem that these are bijective. Let $f \in L^2(\Omega) \subset X'(\Omega)$ and set $\tilde{v} = (- \vr \Delta + M(\vphi_\vr(t)) )^{-1}$, then we can test with $\tilde{v} \in X(\Omega)$ and obtain 
	\begin{equation*}
		\vr \int_\Omega \nabla \tilde{v} \cdot  \nabla \tilde{v} + M(\vphi_\vr (t)) \tilde{v}^2 \dx = \int_\Omega f \tilde{v} \dx . 
 	\end{equation*}
 	Recalling that $M$ is uniformly positive, an application of Young's inequality implies 
 	\begin{equation}\label{eq:boundeness_laplace_inverse}
 		\norm{\tilde{v}}_{L^2}^2 +  \vr \norm{ \nabla v}_{L^2}^2 \leq C \norm{f}_{L^2}^2, 
 	\end{equation}
 	where $C >0$ is independent of $t \in [0, T]$ and $\varrho >0$. In particular, the family of linear operators $(- \vr \Delta + M(\vphi_\vr(t)) )^{-1}$ is uniformly bounded, i.e., 
 	\begin{equation} \label{eq:boundeness_laplace_inverse_2}
 		(- \vr \Delta + M(\vphi_\vr(t)) )^{-1} : L^2(\Omega) \rightarrow L^2(\Omega),  \quad \norm{(- \vr \Delta + M(\vphi_\vr(t)) )^{-1}}_{\mathcal{L}(L^2)} <  C. 
 	\end{equation}
 	On account of the separability of $X(\Omega)$ and with the help of the fundamental lemma of the calculus of variations, it follows from \eqref{eq:weak5_reg} that for all $\xi \in X(\Omega)$
 	\begin{equation*}
 		\int_\Omega \vr_j \nabla \theta_{\vr_j}(t) \cdot \nabla \xi + M(\vphi_{\vr_j}(t))\theta_{\vr_{j}}(t) \xi \dx = 
 		\int_\Omega  \big(  p_{\vr_j} (t)+ M(\vphi_{\vr_j}(t)) \alpha (\vphi_{\vr_j}(t)) \nabla \cdot  \bm{u}_{\vr_j}(t) \big)  \, \xi \dx
 	\end{equation*} 
 	for almost all $t \in (0, T)$ and all $(\vr_j)_{j \inN}$ in some subsequence with $\vr_j \searrow 0$, i.e., 
 	\begin{equation*}
 		\theta_{\vr_j}(t) = (- \vr \Delta + M(\vphi_{\vr_j}(t)) )^{-1} \big( p_{\vr_j} (t)+ M(\vphi_{\vr_j}(t)) \alpha (\vphi_{\vr_j}(t)) \nabla \cdot  \bm{u}_{\vr_j}(t) \big). 
 	\end{equation*}
 	We will now show the convergence 
 	\begin{equation} \label{conv:strong_continuity_laplace}
 		(- \vr \Delta + M(\vphi_{\vr_j}(t)) )^{-1} f \rightarrow M^{-1}(\vphi(t)) f \quad \textrm{as} \quad j \rightarrow \infty
 	\end{equation}
 	 for all $f \in L^2(\Omega)$ and almost every $t \in (0, T)$, which implies pointwise a.e.~convergence of $\theta_{\vr_j}$ in $L^2(\Omega)$. To this end, let $f \in C^\infty_c(\Omega)$ and set $v_j \coloneqq (- \vr \Delta + M(\vphi_{\vr_j}(t)) )^{-1} f $ for all $j \inN$, as well as $v \coloneqq M^{-1}(\vphi(t)) f$. Since $\vphi \in L^2(0, T; H^1(\Omega))$, it follows that $v \in X(\Omega)$ for a.e.~$t \in (0, T)$ and hence, 
 	\begin{equation*}
 		\vr_j \int_\Omega \nabla v \cdot \nabla \xi  + M(\vphi_{\vr_j}(t)) v \, \xi \dx 
 		 = \int_\Omega f \xi + \vr_{j} \nabla v \cdot \nabla \xi + (M(\vphi_{\vr_j}) (t) - M(\vphi(t))) v \, \xi \dx
 	\end{equation*}
 	for all $\xi \in X(\Omega)$. On the other hand, all $v_j$, $j \inN$, satisfy a similar equation by definition and we obtain after subtraction 
 	\begin{equation*}
 		 \int_\Omega \vr_j \nabla (v - v_j) \cdot \nabla \xi + M(\vphi_{\vr_j}(t))(v- v_j) \, \xi \dx
 		= \int_\Omega \vr_j \nabla v \cdot \nabla \xi + (M(\vphi_{\vr_j}) (t) - M(\vphi(t))) v \, \xi \dx
 	\end{equation*} 
 	for all $\xi \in X(\Omega)$. Testing this equation with $v- v_j \in X(\Omega )$ yields 
 	\begin{align*}
 		\underline{M} \norm{v - v_j}_{L^2}^2 &+ \vr_j \norm{\nabla v - \nabla v_j}_{L^2}^2\\ 
 		&\leq \vr^{1/2}_j \norm{\nabla v}_{L^2} \left( \norm{\vr^{1/2}_j \nabla v}_{L^2} + \norm{ \vr^{1/2} \nabla  v_j}_{L^2} \right) 
 		+ \norm{(M(\vphi_{\vr_j}) (t) - M(\vphi(t))) v }_{L^2} \norm{v- v_j}_{L^2}. 
 	\end{align*}
 	Since the term $\norm{\nabla v}_{L^2} \left( \norm{\vr^{1/2}_j \nabla v}_{L^2} + \norm{ \vr^{1/2}_j \nabla  v_j}_{L^2} \right) $ is bounded, which easily follows from \eqref{eq:boundeness_laplace_inverse}, it vanishes as $\vr_j \rightarrow 0$. Due to the strong convergence $\vphi_\vr \rightarrow \vphi$ in $C^0([0, T]; L^2(\Omega))$, cf. \eqref{conv:phi_strong}, a similar argument as for \eqref{conv:no_subsq} shows $(M(\vphi_{\vr_j}) (t) - M(\vphi(t))) v \rightarrow 0$ in $L^2(\Omega)$. Thus, 
 	\begin{equation*}
 		v_j \rightarrow v \quad \textrm{in} \quad L^2(\Omega) \quad \textrm{as} \quad j \rightarrow \infty. 
 	\end{equation*}
 	As $C^\infty_c(\Omega)$ is dense in $L^2(\Omega)$ and the family $(- \vr \Delta + M(\vphi_{\vr_j}(t)) )^{-1}$ is bounded in $\mathcal{L}(L^2)$, we deduce the convergence postulated in \eqref{conv:strong_continuity_laplace}. This implies
 	\begin{align*}
 		\theta_{\vr_j} (t)
 		&= (- \vr \Delta + M(\vphi_{\vr_j}(t)) )^{-1} \big( p_{\vr_j} (t)+ M(\vphi_{\vr_j}(t)) \alpha (\vphi_{\vr_j}(t)) \nabla \cdot  \bm{u}_{\vr_j}(t) \big)\\ 
 		&=\begin{aligned}[t]
 			&(- \vr \Delta + M(\vphi_{\vr_j}(t)) )^{-1} \Big( p_{\vr_j} (t)+ M(\vphi_{\vr_j}(t)) \alpha (\vphi_{\vr_j}(t)) \nabla \cdot  \bm{u}_{\vr_j}(t) -  p (t) - M(\vphi(t)) \alpha (\vphi(t)) \nabla \cdot  \bm{u}(t) \Big) \\ 
 			&+ (- \vr \Delta + M(\vphi_{\vr_j}(t)) )^{-1} \Big(  p (t) + M(\vphi(t)) \alpha (\vphi(t)) \nabla \cdot  \bm{u}(t) \Big)
 		\end{aligned} \\ 
 	& \rightarrow M(\vphi)^{-1}\big(  p (t) + M(\vphi(t)) \alpha (\vphi(t)) \nabla \cdot  \bm{u}(t) \big), 
 	\end{align*}
 	where the first term vanishes due to \eqref{eq:boundeness_laplace_inverse_2} and since the compactness properties for $p, \bm{u}$ and $\vphi$ allow us to assume that without loss of generality 
 	\begin{equation*}
 		p_{\vr_j} (t)+ M(\vphi_{\vr_j}(t)) \alpha (\vphi_{\vr_j}(t)) \nabla \cdot  \bm{u}_{\vr_j}(t) \rightarrow p (t) + M(\vphi(t)) \alpha (\vphi(t)) \nabla \cdot  \bm{u}(t) \quad \textrm{in} \quad L^2(\Omega). 
 	\end{equation*}
	In particular, $\theta_{\vr_j}$ converges to $M(\vphi)^{-1}\big(  p (t) + M(\vphi(t)) \alpha (\vphi(t)) \nabla \cdot  \bm{u}(t) \big) = \theta$ pointwise a.e.~in $L^2(\Omega)$. Finally, we test \eqref{eq:weak5_reg} with $\vartheta \theta_{\vr_j}$ with $\vartheta \in C^\infty([0, T])$ and apply the fundamental lemma of the calculus of variations to obtain 
	\begin{equation*}
		\norm{\theta_{\vr_j}(t)}_{L^2}^2 \leq C \norm{ (p_{\vr_j} + M(\vphi_{\vr_j}) \alpha(\vphi_{\vr_j}) \nabla \cdot \bm{u}_{\vr_j}) (t)}_{L^2}^2 
	\end{equation*}
	for almost all $t \in (0, T)$. Since $ \norm{ (p_{\vr_j} + M(\vphi_{\vr_j}) \alpha(\vphi_{\vr_j}) \nabla \cdot \bm{u}_{\vr_j}) (t)}_{L^2}^2  \rightarrow  \norm{ p + M(\vphi) \alpha(\vphi) \nabla \cdot \bm{u}}_{L^2}^2 $ in $L^1(0, T)$, we can apply Lebesgue's generalized convergence theorem, concluding the proof. 
\end{proof}

\subsection{Limit process}

Limit passage is now very similar to Section \ref{sec:limit_reg}, with the difference that we do not have to restrict ourselves to test functions in linear subspaces.
We also point out that the regularizations $\varrho^{1/2} \Delta \vphi_\varrho$ and $\vr \nabla \theta$ vanishes in the limit, which can be seen by using the weak convergences $\varrho^{1/2} \Delta \vphi_\varrho \rightharpoonup 0$ and $\varrho \nabla \theta_{\vr} \rightharpoonup 0$ when passing to the limit in 
\begin{gather*}
	\int_{\OT} \varrho^{1/2} \Delta \vphi_\varrho \Delta \zeta \dtx \rightarrow 0,\\ 
		\int_{\OT} \varrho \nabla \theta_\varrho \Delta \xi \dtx \rightarrow 0, 
\end{gather*}
for all $\zeta \in L^2(0, T; H_{\bm{n}}^2(\Omega))$ and $\xi \in L^2(0, T; X(\Omega))$, respectively. Keeping these differences in mind, we can pass to the limit and find that $(\vphi, \mu, \bm{u}, \theta, p)$ satisfy the equations \eqref{eq:weak1}, \eqref{eq:weak2},\eqref{eq:weak4} and \eqref{eq:weak5}, but as we saw in Lemma~\ref{lemma:u_strong_not_reg}, we only obtain 
	\begin{align*}
	\int_{\OT} \C_\nu (\vphi ) \E(\partial_t \bm{u}) :  \E(\bm{\eta}) &+ \WE(\vphi , \E(\bm{u})) : \E(\bm{\eta}) -  \alpha(\vphi) p\, (\nabla \cdot \bm{\eta})* \phi \dtx \\ 
	&  = \int_{\OT} \bm{f} \cdot \bm{\eta} \dtx + \int_0^T \int_{\bm{\Gamma_N}} \bm{g} \cdot \bm{\eta} \dH \dt , 
\end{align*}
instead of \eqref{eq:weak3}. Below, we briefly remark on how to pass to the limit in which this convolution vanishes. 
\par 
Moreover, since $H^2(\Omega)$ is dense in $H^1(\Omega)$, we can infer that \eqref{eq:weak2} also holds for all $\zeta \in L^2(0, T; H^1(\Omega)) \cap L^\infty(\OT)$. Similarly, we find that \eqref{eq:weak5} also holds for all $\xi \in L^2(\OT)$. \par
\medskip
As before, the strong convergences 
\begin{align*}
	\vphi_\varrho \rightarrow \vphi \quad &\textrm{in} \quad C^0([0, T]; L^2(\Omega)),\\ 
	\bm{u}_{\varrho} \rightarrow \bm{u} \quad &\textrm{in} \quad C^0([0, T]; L^2(\Omega)),\\ 
	\theta_\varrho \rightarrow \theta \quad &\textrm{in} \quad C^0([0, T]; X'(\Omega)) \numberthis \label{conv:theta_strong_dual}
\end{align*}
follow with the help of the Aubin-Lions-Simon theorem. Along with \eqref{conv:initial_strong}, this implies that the initial conditions \eqref{eq:weak_initial} are also fulfilled. 
\par 
\medskip 
Lastly, we exploit weak/weak* lower semi-continuity and Fatou's lemma on \eqref{eq:a_priori_not_regularized} to find 
\begin{align*}
	\norm{\vphi(t)}_{H^1}^2 
	&+ \norm{\psi(\vphi)(t)}_{L^1} 
	+ \norm{\bm{u}(t)}_{\bm{X}}^2
	+  \norm{\theta(t)}_{L^2}^2 \\
	&+\norm{\pt \varphi}_{L^2((H^1)')}^2
	+\norm{\pt \theta}_{L^2(X')}^2
	+\norm{ \mu}_{L^2(H^1)}^2 
	+ \norm{ p}_{L^2(H^1)}^2
	+ \norm{\pt \bm{u}}_{L^2(\bm{X})}^2 \\ 
	&\leq C( \norm{\vphi_{0}}_{H^1}^4 +  \frac{1}{\vepsilon^2}  \norm{\psi(\vphi_{0})}_{L^1}^2 +  \norm{\bm{u}_0}_{\bm{X}}^4 + \norm{\theta_{0}}_{L^2}^4 + 1), \label{eq:a_priori_not_regularized_2} \numberthis 
\end{align*} 
for almost all $t \in (0, T)$ and some $C > 0$ which only depends on the initial conditions. In particular, the right-hand side is now independent of $\norm{\theta_0}_{X}$, which will be important for more general initial conditions.

\subsection{Vanishing convolution} \label{sec:vanishing_convolution}

Concerning the initial conditions, we can choose an arbitrary sequence $\{ \theta_{\vr, 0}\}_{\vr} \subset X(\Omega)$ such that $\theta_{\vr, 0} \rightarrow \theta_0$ in $L^2(\Omega)$ and since the right-hand side of \eqref{eq:a_priori_not_regularized_2} is independent of $\norm{\theta_{0, \vr}}_{X}$, we still obtain uniform \textit{a priori} estimates. The embedding $L^2(\Omega) \hookrightarrow X'(\Omega)$ along with \eqref{conv:theta_strong_dual} then imply that we obtain $\theta(0) = \theta_0$ for the limit. \par 
As we already remarked in the preliminaries, if $\phi$ is a standard convolution kernel it holds that $\eta * \phi_\vr \rightarrow \eta$ for all $\eta \in L^p(\Omega)$, $p \in [1, \infty)$, as $\vr \rightarrow 0$. Along with the estimate $\norm{\eta * \phi_\vr}_{L^p} \leq  \norm{\eta}_{L^p} \norm{\phi_\vr}_{L^1}$ and the fact that $\norm{\phi_\vr}_{L^1} = 1$, we obtain that 
\begin{equation*}
	(\nabla \cdot \bm{\eta})  * \phi_\vr \rightarrow \nabla \cdot \bm{\eta} \quad \textrm{in} \quad L^2(\OT)
\end{equation*}
for all $\bm{\eta} \in L^2(0, T; \bm{X}(\Omega))$. Therefore, we only need to verify that the same compactness properties as before can be deduced. Firstly, we remark that our \textit{a priori} estimates are independent of the convolution kernel $\phi$. Moreover, it is obvious that the argument for the strong convergence of $p$ in $L^2(0, T; L^2(\Omega))$ remains valid. For the strong convergence of $\bm{u}$, we find that 
\begin{align*}
	  &\Big|  \int_{\Omega_t} 
		 [\alpha(\vphi_k) p_\vr - \alpha(\vphi) p] (\nabla \cdot(\bm{u}_\vr - \bm{u})) * \phi_\vr 
		+  [\alpha(\vphi_k) p_\vr- \alpha(\vphi) p] (\nabla \cdot(\pt \bm{u}_\vr - \pt \bm{u})) * \phi_\vr \dtx  \Big| \\ 
	& \quad \leq C \norm{\alpha(\vphi_k) p_\vr - \alpha(\vphi) p}_{L^2(L^2)} + \rho_{\bm{u}} \norm{ \bm{u}_\vr - \bm{u} }_{L^2(\bm{X})}  \norm{\phi_\vr}_{L^\infty(L^1)}\\ 
	&\quad \leq C \norm{\alpha(\vphi_k) p_\vr - \alpha(\vphi) p}_{L^2(L^2)} + \rho_{\bm{u}} \norm{ \bm{u}_\vr - \bm{u} }_{L^2(\bm{X})} 
\end{align*}
for some $C = C(\rho_{\bm{u}}) > 0$ and $\rho_{\bm{u}} >0$ sufficiently small. We emphasize that $\norm{\phi_\vr}_{L^\infty(L^1)} = 1$ independently of $\vr >0$. Hence, a similar estimate as in \eqref{eq:u_gronwall}, but without the last line, holds with constants that are independent of the convolution kernel and we infer the desired strong convergence for $\bm{u}$. \par 
Lastly, we take advantage of the identity $p_\vr = M(\varphi_\vr)(\theta_\vr - \alpha(\vphi_{\vr}) \nabla \cdot \bm{u}_\vr)$, which follows from \eqref{eq:weak5} and holds pointwise almost everywhere, and deduce strong convergence for $\theta$ from the compactness properties of $\vphi_\vr, \bm{u}_\vr, p_\vr$. \par 
\medskip
After passing to the limit it only remains the use the pointwise identity for $p$ and replace the respective terms in the other equations.  

\begin{remark}
	Lastly, we point out that maximal regularity theory as applied in Section \ref{sec:existence_semi_galerkin} also yields the regularity $\bm{u} \in L^2(0, T; \bm{W}^{1, q}_{\Gamma_D}(\Omega))$ if $\bm{u}_0 \in  \bm{W}^{1, q}_{\Gamma_D}(\Omega)$ with $q > 2$ sufficiently small, cf. Lemma \ref{lem:cauchy_problem}. \\ 
	For regularized problems, one can even show $\bm{u}_\varrho \in L^2(0, T; \bm{H}^{1+ \delta}_{\Gamma_D}(\Omega))$, where $\bm{H}^{1+ \delta}_{\Gamma_D}(\Omega)$ is some Bessel potential space and $\delta > 0$. Again, the proof relies on maximal regularity theory as well as elliptic regularity in Bessel potential spaces, cf. \cite[Thm.~1]{haller2019higher}. Note that since these arguments utilize the $L^\infty(0, T; H^2_{\bm{n}}(\Omega))$ norm of $\vphi_\varrho$, we do not obtain an uniform estimate in the Bessel potential spaces.  
\end{remark}

	\paragraph{Acknowledgments} 
	The last author is supported by the Graduiertenkolleg 2339 IntComSin of the Deutsche Forschungsgemeinschaft (DFG, German Research Foundation) – Project-ID 321821685. The support is gratefully acknowledged. We would also like to thank Jonas Stange for his careful proofreading \new{ and the anonymous referees for helpful comments leading to the improvement of this paper. }
	
	\paragraph{Conflict of interests and data availability statement} There is no conflict of interests.
	There is no associated data to the manuscript.

	\bibliographystyle{siam}
	\small
	\bibliography{refs}

\begin{thebibliography}{10}

\bibitem{Abels+2012}
{\sc H.~Abels}, {\em Pseudodifferential and Singular Integral Operators: An
  Introduction with Applications}, De Gruyter, Berlin, Boston, 2012.

\bibitem{alt2013lineare}
{\sc H.~Alt}, {\em Linear Functional Analysis: An Application-Oriented
  Introduction}, Universitext, Springer London, 2012.

\bibitem{ARENDT20071}
{\sc W.~Arendt, R.~Chill, S.~Fornaro, and C.~Poupaud}, {\em ${L}^p$-maximal
  regularity for non-autonomous evolution equations}, J. Diff. Equ., 237
  (2007), pp.~1--26.

\bibitem{auriault1980dynamic}
{\sc J.~Auriault}, {\em Dynamic behaviour of a porous medium saturated by a
  {N}ewtonian fluid}, Int. J. Eng. Sci., 18 (1980), pp.~775--785.

\bibitem{biot1941general}
{\sc M.~A. Biot}, {\em General theory of three-dimensional consolidation}, J.
  Appl Phys., 12 (1941), pp.~155--164.

\bibitem{biot1956theory}
\leavevmode\vrule height 2pt depth -1.6pt width 23pt, {\em Theory of
  deformation of a porous viscoelastic anisotropic solid}, J. Appl. Phys., 27
  (1956), pp.~459--467.

\bibitem{biot1957elastic}
{\sc M.~A. Biot and D.~G. Willis}, {\em The elastic coefficients of the theory
  of consolidation}, J. Appl. Mech.,  (1957).

\bibitem{blowey1991cahn}
{\sc J.~F. Blowey and C.~M. Elliott}, {\em The {C}ahn--{H}illiard gradient
  theory for phase separation with non-smooth free energy {P}art {I:}
  {M}athematical analysis}, European J. Appl. Math., 2 (1991), pp.~233--280.

\bibitem{bociu2021multilayered}
{\sc L.~Bociu, S.~Canic, B.~Muha, and J.~T. Webster}, {\em Multilayered
  poroelasticity interacting with {S}tokes flow}, SIAM J. Math. Anal., 53
  (2021), pp.~6243--6279.

\bibitem{bociu2016analysis}
{\sc L.~Bociu, G.~Guidoboni, R.~Sacco, and J.~T. Webster}, {\em Analysis of
  nonlinear poro-elastic and poro-visco-elastic models}, Arch. Ration. Mech.
  Anal., 222 (2016), pp.~1445--1519.

\bibitem{bociu2023mathematical}
{\sc L.~Bociu, B.~Muha, and J.~T. Webster}, {\em Mathematical effects of linear
  visco-elasticity in quasi-static {B}iot models}, Journal of Mathematical
  Analysis and Applications,  (2023), p.~127462.

\bibitem{BOCIU2023127462}
{\sc L.~Bociu, B.~Muha, and J.~T. Webster}, {\em Mathematical effects of linear
  visco-elasticity in quasi-static {B}iot models}, J. Math. Anal., 527 (2023),
  p.~127462.

\bibitem{bonetti2002model}
{\sc E.~Bonetti, P.~Colli, W.~Dreyer, G.~Gilardi, G.~Schimperna, and
  J.~Sprekels}, {\em On a model for phase separation in binary alloys driven by
  mechanical effects}, Phys. D, 165 (2002), pp.~48--65.

\bibitem{both2019gradient}
{\sc J.~W. Both, K.~Kumar, J.~M. Nordbotten, and F.~A. Radu}, {\em The gradient
  flow structures of thermo-poro-visco-elastic processes in porous media},
  2019.
\newblock \href{https://arxiv.org/abs/1907.03134}{arXiv: 1907.03134}.

\bibitem{both21global}
{\sc J.~W. Both, I.~S. Pop, and I.~Yotov}, {\em Global existence of weak
  solutions to unsaturated poroelasticity}, ESAIM: M2AN, 55 (2021),
  pp.~2849--2897.

\bibitem{brunk2024structurepreservingapproximationcahnhilliardbiot}
{\sc A.~Brunk and M.~Fritz}, {\em Structure-preserving approximation of the
  {C}ahn-{H}illiard-{B}iot system}, 2024.
\newblock \href{https://arxiv.org/abs/2407.12349}{arXiv: 2407.12349}.

\bibitem{larcht1982effect}
{\sc J.~Cahn and F.~Larché}, {\em The effect of self-stress on diffusion in
  solids}, Acta Metall., 30 (1982), pp.~1835--1845.

\bibitem{cahn1996cahn}
{\sc J.~W. Cahn, C.~M. Elliott, and A.~Novick-Cohen}, {\em The
  {C}ahn--{H}illiard equation with a concentration dependent mobility: {M}otion
  by minus the laplacian of the mean curvature}, European J. Appl. Math., 7
  (1996), pp.~287--301.

\bibitem{cahn1958free}
{\sc J.~W. Cahn and J.~E. Hilliard}, {\em Free energy of a nonuniform system.
  {I.} {I}nterfacial free energy}, J. Chem. Phys., 28 (1958), pp.~258--267.

\bibitem{MR1807441}
{\sc M.~Carrive, A.~Miranville, and A.~Pi\'{e}trus}, {\em The
  {C}ahn--{H}illiard equation for deformable elastic continua}, Adv. Math. Sci.
  Appl., 10 (2000), pp.~539--569.

\bibitem{COLLI2012428}
{\sc P.~Colli, S.~Frigeri, and M.~Grasselli}, {\em Global existence of weak
  solutions to a nonlocal {C}ahn--{H}illiard--{N}avier--{S}tokes system}, J.
  Math. Anal., 386 (2012), pp.~428--444.

\bibitem{colli2017asymptotic}
{\sc P.~Colli, G.~Gilardi, E.~Rocca, and J.~Sprekels}, {\em Asymptotic analyses
  and error estimates for a {C}ahn--{H}illiard type phase field system
  modelling tumor growth}, Discrete Contin. Dyn. Syst. Ser. S, 10 (2017).

\bibitem{coussy2004poromechanics}
{\sc O.~Coussy}, {\em Poromechanics}, Wiley, 2004.

\bibitem{Dore93}
{\sc G.~Dore}, {\em ${L}^p$ regularity for abstract differential equations}, in
  Functional Analysis and Related Topics, 1991, H.~Komatsu, ed., Berlin,
  Heidelberg, 1993, Springer Berlin Heidelberg, pp.~25--38.

\bibitem{ebenbeck2019analysis}
{\sc M.~Ebenbeck and H.~Garcke}, {\em Analysis of a
  {C}ahn--{H}illiard--{B}rinkman model for tumour growth with chemotaxis}, J.
  Diff. Equ., 266 (2019), pp.~5998--6036.

\bibitem{edmunds_triebel_1996}
{\sc D.~E. Edmunds and H.~Triebel}, {\em Function Spaces, Entropy Numbers,
  Differential Operators}, Cambridge Tracts in Mathematics, Cambridge
  University Press, 1996.

\bibitem{EGERT20141419}
{\sc M.~Egert, R.~Haller-Dintelmann, and P.~Tolksdorf}, {\em The {K}ato square
  root problem for mixed boundary conditions}, J. Funct. Anal., 267 (2014),
  pp.~1419--1461.

\bibitem{elliot_garcke00}
{\sc C.~Elliott and H.~Garcke}, {\em {On the {C}ahn--{H}illiard equation with
  degenerate mobility}}, SIAM J. Math. Anal., 27 (1996), pp.~404--423.

\bibitem{engel2006one}
{\sc K.~Engel and R.~Nagel}, {\em One-Parameter Semigroups for Linear Evolution
  Equations}, Graduate Texts in Mathematics, Springer New York, 1999.

\bibitem{fritz2023wellposedness}
{\sc M.~Fritz}, {\em On the well-posedness of the {C}ahn--{H}illiard--{B}iot
  model and its applications to tumor growth}, 2023.
\newblock \href{https://arxiv.org/abs/2310.07050}{arXiv: 2310.07050}.

\bibitem{garcke_2003}
{\sc H.~Garcke}, {\em {On Cahn--Hilliard systems with elasticity}}, Proc. Roy.
  Soc. Edinburgh Sect. A, 133 (2003), p.~307–331.

\bibitem{GARCKE2005165}
{\sc H.~Garcke}, {\em On a {C}ahn--{H}illiard model for phase separation with
  elastic misfit}, Ann. Inst. H. Poincaré Anal. Non Linéaire, 22 (2005),
  pp.~165--185.

\bibitem{lam06}
{\sc H.~Garcke and K.~Lam}, {\em Global weak solutions and asymptotic limits of
  a {C}ahn--{H}illiard--{D}arcy system modelling tumour growth}, AIMS
  Mathematics, 1 (2016), pp.~318--360.

\bibitem{MR4126782}
{\sc H.~Garcke, K.~Lam, and A.~Signori}, {\em On a phase field model of
  {C}ahn--{H}illiard type for tumour growth with mechanical effects}, Nonlinear
  Anal. Real World Appl., 57 (2021), pp.~Paper No. 103192, 28.

\bibitem{garcke16darcy}
{\sc H.~Garcke and K.~F. Lam}, {\em Global weak solutions and asymptotic limits
  of a {C}ahn--{H}illiard--{D}arcy system modelling tumour growth}, AIMS
  Mathematics, 1 (2016), pp.~318--360.

\bibitem{groger1989aw}
{\sc K.~Gr{\"o}ger}, {\em A ${W}^{1, p}$-estimate for solutions to mixed
  boundary value problems for second order elliptic differential equations},
  Mathematische Annalen, 283 (1989), pp.~679--687.

\bibitem{haller2019higher}
{\sc R.~Haller-Dintelmann, H.~Meinlschmidt, and W.~Wollner}, {\em Higher
  regularity for solutions to elliptic systems in divergence form subject to
  mixed boundary conditions}, Ann. Mat. Pura Appl., 198 (2019), pp.~1227--1241.

\bibitem{han2014existence}
{\sc D.~Han, X.~Wang, and H.~Wu}, {\em Existence and uniqueness of global weak
  solutions to a {C}ahn--{H}illiard--{S}tokes--{D}arcy system for two phase
  incompressible flows in karstic geometry}, J. Differential Equations, 257
  (2014), pp.~3887--3933.

\bibitem{HERZOG2011802}
{\sc R.~Herzog, C.~Meyer, and G.~Wachsmuth}, {\em Integrability of displacement
  and stresses in linear and nonlinear elasticity with mixed boundary
  conditions}, J. Math. Anal., 382 (2011), pp.~802--813.

\bibitem{MR3794344}
{\sc E.~Holland and R.~E. Showalter}, {\em Poro-visco-elastic compaction in
  sedimentary basins}, SIAM J. Math. Anal., 50 (2018), pp.~2295--2316.

\bibitem{MR4649995}
{\sc A.~Hosseinkhan and R.~E. Showalter}, {\em Semilinear degenerate
  {B}iot-{S}ignorini system}, SIAM J. Math. Anal., 55 (2023), pp.~5643--5665.

\bibitem{juengelweak}
{\sc X.~Huo, A.~J\"{u}ngel, and A.~Tzavaras}, {\em Existence and weak-strong
  uniqueness for {M}axwell-{S}tefan--{C}ahn--{H}illiard systems}, 2022.

\bibitem{lowengrub2013analysis}
{\sc J.~Lowengrub, E.~Titi, and K.~Zhao}, {\em Analysis of a mixture model of
  tumor growth}, European J.Appl. Math., 24 (2013), pp.~691--734.

\bibitem{lunardi1995analytic}
{\sc A.~Lunardi}, {\em Analytic Semigroups and Optimal Regularity in Parabolic
  Problems}, Modern Birkh{\"a}user Classics, Springer Basel, 1995.

\bibitem{mow1980biphasic}
{\sc V.~C. Mow, S.~Kuei, W.~M. Lai, and C.~G. Armstrong}, {\em Biphasic creep
  and stress relaxation of articular cartilage in compression: {T}heory and
  experiments}, ASME J. Biomech. Eng.,  (1980).

\bibitem{onuki1989ginzburg}
{\sc A.~Onuki}, {\em {G}inzburg--{L}andau approach to elastic effects in the
  phase separation of solids}, J. Phys. Soc. Jap., 58 (1989), pp.~3065--3068.

\bibitem{PhysRevA.38.434}
{\sc Y.~Oono and S.~Puri}, {\em Study of phase-separation dynamics by use of
  cell dynamical systems. {I.} {M}odeling}, Phys. Rev. A, 38 (1988),
  pp.~434--453.

\bibitem{pazy2012semigroups}
{\sc A.~Pazy}, {\em Semigroups of Linear Operators and Applications to Partial
  Differential Equations}, Applied Mathematical Sciences, Springer New York,
  2012.

\bibitem{phillips1953perturbation}
{\sc R.~S. Phillips}, {\em Perturbation theory for semi-groups of linear
  operators}, Transactions of the American Mathematical Society, 74 (1953),
  pp.~199--221.

\bibitem{riethmüller2023wellposedness}
{\sc C.~Riethmüller, E.~Storvik, J.~W. Both, and F.~A. Radu}, {\em
  Well-posedness analysis of the {C}ahn--{H}illiard--{B}iot model}, 2023.
\newblock \href{https://arxiv.org/abs/2310.18231}{arXiv: 2310.18231}.

\bibitem{sacco2019comprehensive}
{\sc R.~Sacco, G.~Guidoboni, and A.~G. Mauri}, {\em A comprehensive physically
  based approach to modeling in bioengineering and life sciences}, Academic
  press, 2019.

\bibitem{SHOWALTER2000310}
{\sc R.~Showalter}, {\em Diffusion in poro-elastic media}, J. Math. Anal., 251
  (2000), pp.~310--340.

\bibitem{MR2102316}
{\sc R.~E. Showalter and U.~Stefanelli}, {\em Diffusion in poro-plastic media},
  Math. Methods Appl. Sci., 27 (2004), pp.~2131--2151.

\bibitem{MR1876882}
{\sc R.~E. Showalter and N.~Su}, {\em Partially saturated flow in a poroelastic
  medium}, Discrete Contin. Dyn. Syst. Ser. B, 1 (2001), pp.~403--420.

\bibitem{simon1986compact}
{\sc J.~Simon}, {\em {Compact sets in the space {$L^p(0,T;B)$}}}, Ann. Mat.
  Pura Appl. (4), 146 (1987), pp.~65--96.

\bibitem{sohr2001navier}
{\sc H.~Sohr}, {\em The {N}avier--{S}tokes Equations: An Elementary Functional
  Analytic Approach}, Birkh{\"a}user advanced texts, Springer Basel, 2001.

\bibitem{STORVIK2022107799}
{\sc E.~Storvik, J.~W. Both, J.~M. Nordbotten, and F.~A. Radu}, {\em A
  {C}ahn--{H}illiard--{B}iot system and its generalized gradient flow
  structure}, Appl. Math. Lett., 126 (2022), p.~107799.

\bibitem{storvik2024sequential}
{\sc E.~Storvik, C.~Riethmüller, J.~W. Both, and F.~A. Radu}, {\em Sequential
  solution strategies for the {C}ahn--{H}illiard--{B}iot model}, 2024.
\newblock \href{https://arxiv.org/abs/2401.13358}{arXiv: 2401.13358}.

\bibitem{taylor2010partial}
{\sc M.~Taylor}, {\em Partial Differential Equations I: Basic Theory}, Applied
  Mathematical Sciences, Springer New York, 2010.

\bibitem{triebel1978interpolation}
{\sc H.~Triebel}, {\em Interpolation Theory, Function Spaces, Differential
  Operators}, North-Holland Publishing Company, 1978.

\bibitem{van2023mathematical}
{\sc C.~van Duijn and A.~Mikeli{\'c}}, {\em Mathematical theory of nonlinear
  single-phase poroelasticity}, J. Nonlinear Sci., 33 (2023), p.~44.

\bibitem{vzenivsek1984existence}
{\sc A.~{\v{Z}}en{\'\i}{\v{s}}ek}, {\em The existence and uniqueness theorem in
  {B}iot's consolidation theory}, Aplikace matematiky, 29 (1984), pp.~194--211.

\end{thebibliography}

	\appendix 
	\new{
	\section{Proof of Theorem \ref{th:max_reg_non_autonomous}} \label{app:proof_max_reg}
	
	The following lemma studies the inverse operators to families of perturbed autonomous abstract Cauchy-problems associated with a family of operators $\{A_i\}$. 
	
	\begin{lemma}\label{lemma:bound_M}
		Assume that $\{A_i : i \in \mathcal{I}\} \subset \mathcal{L}(D, Y)$, for some index set $\mathcal{I}$, is a family of linear operators such that $A^i \in \mathcal{MR}$ and $A_i - A_j \in \mathcal{L}(Y)$ with $\norm{A_i - A_j}_{\mathcal{L}(Y)} \leq C_{\mathcal{I}}$ for all $i, j \in \mathcal{I}$. Moreover, define the operator  $\mathfrak{L}_{A_i}$ as  
		\begin{align*}
			D(\mathfrak{L}_{A_i}) &= \{u \in \textnormal{MR}(0,T) : u(0) = 0 \} 
			\quad \begin{aligned}[t]
				&\textrm{where} \quad
				\textnormal{MR}(a, b) \coloneqq W^{1, p}(a,b; Y) \cap L^p(a,b; D),  \\ 
				&\textrm{and} \quad \norm{\cdot}_{MR} \coloneqq \norm{\cdot}_{W^{1, p}(a, b; Y)} + \norm{\cdot}_{L^p(a, b; D)}, 
			\end{aligned}\\ 
			\mathfrak{L}_{A_i} u (t)&= \pt u (t)+ A_i u(t).   
		\end{align*}
		Then, there exists a constant $ M \geq 0$ such that 
		\begin{align*}
			\norm{(\lambda + \mathfrak{L}_{A_i})^{-1}}_{\mathcal{L}(L^p(a, b; Y),  \textnormal{MR}(a,b))} \leq M \quad  \textrm{and} \quad  
			\norm{(1+ \lambda)(\lambda + \mathfrak{L}_{A_i})^{-1}}_{\mathcal{L}(L^p(a, b; Y))} \leq M
		\end{align*} 
		for all intervals $(a, b) \subset (0, T)$, all $i \in \mathcal{I}$ and all $\lambda \geq 0$. 
	\end{lemma}
	\begin{proof}
		From the arguments in \cite[Sec.~1]{ARENDT20071} it follows for all $i \in \mathcal{I}$ that the operators $- \mathfrak{L}_{A_i}$ have empty spectrum and generate nilpotent $C^0$-semigroups $(T^i(t))_{t \geq 0}$ on $L^p(0, T; Y)$. In particular, there exist constants $\omega_i \geq 0$ and $M_i \geq 1$ such that, cf. \cite[Thm.~2.2]{pazy2012semigroups}, 
		\begin{align*}
			\norm{T^i(t)}_{\mathcal{L}(L^p(0, T; Y))} \leq M_i \, e^{\omega_i t} \quad \textrm{for all} \quad t \geq 0. 
		\end{align*}
		For any two indices $i, j \in \mathcal{I}$, we compute for all $u \in D(\mathcal{-L}_{A_i}) = D(\mathcal{-L}_{A_j})$
		\begin{equation*}
			(\mathfrak{L}_{A_i}  -\mathfrak{L}_{A_j})u(t)= (A_i - A_j) u (t). 
		\end{equation*}
		Since $D(\mathcal{-L}_{A_i})$ is dense in $L^p(0, T; Y)$ and $A_i - A_j \in \mathcal{L}(Y)$, this implies $\mathfrak{L}_{A_i}  -\mathfrak{L}_{A_j} \in \mathcal{L}(L^p(0, T; Y))$ with 
		\begin{align*}
			\norm{\mathfrak{L}_{A_i}  -\mathfrak{L}_{A_j}}_{\mathcal{L}(L^p(0, T; Y))} \leq C_{\mathcal{I}}
		\end{align*} 
		Therefore, $- \mathfrak{L}_{A_j} = - \mathfrak{L}_{A_i} + (\mathfrak{L}_{A_i}  -\mathfrak{L}_{A_j})$ is a perturbation of the generator of a $C_0$-semigroup by a bounded, linear operator and \cite[Thm.~3.2]{phillips1953perturbation} implies 
		\begin{equation*}
			\norm{T^j(t)}_{\mathcal{L}(L^p(0, T; Y))} \leq M_i \, e^{\omega^* t} \quad \textrm{for all} \quad t \geq 0, 
		\end{equation*}
		where $\omega^* = \omega_i + M_i \norm{\mathfrak{L}_{A_i}  -\mathfrak{L}_{A_j}}_{\mathcal{L}(L^p(0, T; Y))} \leq \omega_i + C_{\mathcal{I}}$. 
		Hence, there exists some constant $C >0$ such that 
		\begin{equation*}
			\sup_{ \substack{i \in \mathcal{I}\\ t \in [0, T]}} \norm{T^i(t)}_{\mathcal{L}(L^p(0, T; Y))} \leq C
		\end{equation*}
		and since the semigroups $(T^i(t))_{t \geq 0}$ are nilpotent, i.e. $T^i(t) = 0$ for all $t > T$ and all $i \in \mathcal{I}$, we find some $M \geq 1$ such that 
		\begin{equation*}
			\norm{T^i(t)}_{\mathcal{L}(L^p(0, T; Y))} \leq M \, e^{- t} \quad \textrm{for all} \quad t \geq 0
		\end{equation*}
		for all $i \in \mathcal{I}$. 
		The generation theorem for semigroups \cite[II Thm.~3.8]{engel2006one} now yields 
		\begin{gather*}
			\norm{(\lambda + \mathfrak{L}_{A^i})^{-1}}_{\mathcal{L} (L^p(0, T;Y),  \textnormal{MR}(0,T))} \leq \frac{M}{\lambda +1} \leq M, \\ 
			\norm{(1+ \lambda)(\lambda + \mathfrak{L}_{A^i})^{-1}}_{\mathcal{L} (L^p(0, T;Y))}  \leq M
		\end{gather*}
		for all $\lambda \geq 0$. Now the result follows analogously to \cite[Lem.~1.2]{ARENDT20071}. 
	\end{proof}
	
	Now, we fix an operator $A \in \mathcal{L}(D,Y)$ and perturb the associated abstract Cauchy-problem with a time-dependent, operator valued, function $B$. This lemma is a modification of \cite[Prop.~1.3]{ARENDT20071} and includes an additional bound for the solution $u$ in terms of the right-hand side and the initial value. 
	
	\begin{lemma}\label{lemma:pertubation_result}
		Let $(a, b) \subset (0, T)$ and $A \in \mathcal{L}(D, Y)$ with $A \in \mathcal{MR}$. Suppose that $B: (a, b) \rightarrow \mathcal{L}(D, Y)$ is strongly measurable and that there exists some $\eta \geq 0$ such that 
		\begin{equation*}
			\norm{B(t) y}_{Y} \leq \frac{1}{2M}  \norm{y}_D + \eta \norm{y}_Y
		\end{equation*}  
		for all $y\in D$, $t \in (a, b)$, where $M$ is the constant from Lemma \ref{lemma:bound_M}. Then, the abstract Cauchy problem 
		\begin{equation}\label{eq:cauchy_perturbed}
			\pt u + A u + B(t) u = f \quad \textrm{a.e. on } (a, b), \quad u(a) = x
		\end{equation} 
		has a unique solution for all $x \in (Y, D)_{\frac{1}{p'}, p}$, $f \in L^p(a, b; Y)$ where $p \in (1, \infty)$ and $p' = \frac{p}{p-1}$. Moreover, there exists a constant $C > 0$ such that 
		\begin{equation*}
			\norm{u}_{\textnormal{MR}(a, b)} 
			\leq C \norm{y}_{(D, Y)_{\frac{1}{p'}, p}} + 4 M  \textnormal{e}^{|b-a|\lambda} \norm{f}_{L^p(a, b; Y)}, 
		\end{equation*}
		where $\lambda \geq 0$ only depends on $M$. 
	\end{lemma}
	
	The following proof closely follows the arguments by Arendt et al.~in \cite{ARENDT20071}. 
	
	\begin{proof}
		For now we assume $y = 0$ and define the operator $\tilde{B} \in \mathcal{L}(\textnormal{MR}(a,b), L^p(a, b; Y))$ by  
		\begin{equation*}
			(\tilde{B}u)(t) = B(t) u(t). 
		\end{equation*} 
		The assumption on $B$ along with Minkowski's inequality then gives rise to 
		\begin{align*}
			\norm{\tilde{B}u}_{L^p(a, b; Y)} \leq \left(\int_a^b \norm{B(t)u(t)}_Y^p \dt \right)^{\frac{1}{p}} 
			\leq  \Big( \frac{1}{2M}  \norm{u}_{L^p(a,b; D)} + \eta \norm{u}_{L^p(a, b; Y)} \Big) . 
		\end{align*}
		Defining the operator $\mathfrak{L}$ analogously to above,  
		we see with the help of Lemma \ref{lemma:bound_M} that 
		\begin{align*}
			\norm{\tilde{B} (\lambda + \mathfrak{L})^{-1} f}_{L^p(0, T; Y)} 
			&\leq  \Big(  \frac{1}{2M}  \norm{(\lambda + \mathfrak{L})^{-1} f}_{L^p(a,b; D)} + \eta \norm{(\lambda + \mathfrak{L})^{-1} f}_{L^p(a, b; Y)} \Big) \\ 
			&\leq   \Big( \frac{1}{2M}  \norm{(\lambda + \mathfrak{L})^{-1} f}_{\textnormal{MR}(a,b)}  + \frac{\eta M}{1+ \lambda} \norm{f}_{L^p(a, b; Y)} \Big) \\ 
			& \leq   \frac{1}{2}\norm{ f}_{L^p(a, b; Y)} + \frac{ \eta M}{1+ \lambda} \norm{f}_{L^p(a, b; Y)}
		\end{align*}
		for all $\lambda \geq 0$. In particular, there exists some $\lambda = \lambda (M, \eta) \geq 0$ such that 
		\begin{equation*}
			\norm{\tilde{B} (\lambda + \mathfrak{L})^{-1} }_{\mathcal{L}(L^p(0, T; Y))} \leq \frac{3}{4}, 
		\end{equation*} 
		which implies that the operator $I + \tilde{B}(\lambda + \mathfrak{L})^{-1}$ is invertible. A simple calculation further yields the invertibility of 
		\begin{equation*}
			\lambda + \mathfrak{L} + \tilde{B} = (I + \tilde{B}(\lambda + \mathfrak{L})^{-1} )(\lambda + \mathfrak{L}) \in \mathcal{L} (D(\mathfrak{L}), L^p(a, b; y)) 
		\end{equation*}
		and therefore the unique solvability of the problem 
		\begin{equation*}
			\pt v + (A+ \lambda) v + B(t) v = g \quad \textrm{a.e. on } (a, b), \quad v(a) = 0
		\end{equation*}
		for all $g \in L^p(a, b; Y)$. Moreover, it holds that 
		\begin{equation*}
			(\lambda + \mathfrak{L} + \tilde{B})^{-1} = (\lambda + \mathfrak{L})^{-1} (I + \tilde{B}(\lambda + \mathfrak{L})^{-1} )^{-1}, 
		\end{equation*}
		from which we deduce 
		\begin{align*}
			\norm{(\lambda + \mathfrak{L} + \tilde{B})^{-1}}_{\mathcal{L}(L^p(a, b; Y); D(\mathfrak{L}))}
			&\leq \norm{(\lambda + \mathfrak{L})^{-1}}_{\mathcal{L}(L^p(a, b; Y), \textnormal{MR}(a, b))} \norm{(I + \tilde{B}(\lambda + \mathfrak{L})^{-1} )^{-1}}_{\mathcal{L}(L^p(a, b; Y))}\\ 
			&\leq 4M. 
		\end{align*}
		Hence, 
		\begin{equation*}
			\norm{v}_{\textnormal{MR}(a, b)} \leq 4M \norm{g}_{L^p(a, b; Y)}. 
		\end{equation*}
		\par 
		To get back to our original problem, we consider $g(t) = \textnormal{e}^{- \lambda (t-a)} f(t)$ and $y = 0$. Then, the function $u(t) = \textnormal{e}^{ \lambda (t- a)} v(t)$ is the unique solution of \eqref{eq:cauchy_perturbed} for $x = 0$ with 
		\begin{equation*}
			\norm{u}_{\textnormal{MR}(a, b)} \leq \textnormal{e}^{|b-a|\lambda} \norm{v}_{\textnormal{MR}(a, b)} 
			\leq 4M\textnormal{e}^{|b-a|\lambda} \norm{f}_{L^p(a, b; Y)}. 
		\end{equation*}
		\par 
		In order to treat non-trivial initial conditions, i.e. $y \in (D, Y)_{\frac{1}{p'}, p}$, we introduce the space 
		\begin{equation}\label{def:Tr}
			\textnormal{Tr} \coloneqq \{ u(a) : u \in \textnormal{MR}(a, b)\} \quad 
			\textnormal{with the norm} \quad 
			\norm{x}_{\textnormal{TR}} = \inf\{ \norm{u}_{\textnormal{MR}(a, b)}  : y = u(a) \}
		\end{equation}
		and note that, cf. \cite[Prop.~1.2.10]{lunardi1995analytic}, 
		\begin{equation}\label{eq:isomorp_Tr}
			(D, Y)_{\frac{1}{p'}, p} \cong \textnormal{Tr}. 
		\end{equation}
		Moreover, for any $w \in \textnormal{MR}(a, b)$ with $w(a) = y$, the results from above give rise to a unique $v \in \textnormal{MR}(a,b)$ such that 
		\begin{equation*}
			\pt v + (A + B(t))v = - \pt w - (A+B(t))w + f \quad \textrm{a.e. on  } (a,b), \quad  v(a) = 0. 
		\end{equation*}
		Therefore, $u \coloneqq v + w$ is the unique solution of \eqref{eq:cauchy_perturbed} and it holds that 
		\begin{align*}
			\norm{u}_{\textnormal{MR}(a, b)} \leq \norm{v + w}_{\textnormal{MR}(a, b)} 
			&\leq 4 M \textnormal{e}^{|b-a|\lambda} \norm{- \pt w - (A+B(t))w + f}_{L^p(a, b; Y)} + \norm{w}_{L^p(a,b; Y)}\\
			&\leq C \norm{w}_{\textnormal{MR}(a, b)} +  4 M\textnormal{e}^{|b-a|\lambda} \norm{f}_{L^p(a, b; Y)}. 
		\end{align*}
		Since this holds for all $w \in \textnormal{MR}(a, b)$ with $w(a) = y$, definition \eqref{def:Tr} and \eqref{eq:isomorp_Tr} lead to 
		\begin{equation*}
			\norm{u}_{\textnormal{MR}(a, b)} \leq C \norm{x}_{(D, Y)_{\frac{1}{p'}, p}} +  4 M\textnormal{e}^{|b-a|\lambda} \norm{f}_{L^p(a, b; Y)}, 
		\end{equation*}
		which concludes the proof. 
	\end{proof}
	
	As an immediate consequence, we obtain Theorem \ref{th:max_reg_non_autonomous}. 
	}

\end{document}